\documentclass[11pt,reqno]{amsart}
\usepackage[foot]{amsaddr}
\usepackage{color}
\usepackage[colorlinks=true, allcolors=blue]{hyperref}
\usepackage{amsmath, amssymb, amsthm, thmtools, thm-restate}
\usepackage{latexsym}
\usepackage{textcomp}
\usepackage{mathrsfs}
\usepackage{mathtools}
\usepackage[noabbrev,capitalize,nameinlink]{cleveref}
\crefname{equation}{}{}
\usepackage{fullpage}
\usepackage[noadjust]{cite}
\usepackage{graphics}
\usepackage{pifont}
\usepackage{tikz}
\usepackage{bbm}
\usepackage[T1]{fontenc}

\usetikzlibrary{arrows.meta}

\usepackage{environ}
\usepackage{framed}
\usepackage{url}
\usepackage[linesnumbered,ruled,vlined]{algorithm2e}
\usepackage[noend]{algpseudocode}
\usepackage[labelfont=bf]{caption}
\usepackage{cite}
\usepackage{framed}
\usepackage[framemethod=tikz]{mdframed}
\usepackage{appendix}
\usepackage{graphicx}
\usepackage[textsize=tiny]{todonotes}
\usepackage{tcolorbox}
\usepackage{enumerate}
\allowdisplaybreaks[1]
\usepackage{enumerate}

\crefname{algocf}{Algorithm}{Algorithms}

\crefname{equation}{}{} 
\AtBeginEnvironment{appendices}{\crefalias{section}{appendix}} 

\usepackage[color,final]{showkeys} 

\colorlet{refkey}{orange!20}
\colorlet{labelkey}{blue!30}

\crefname{algocf}{Algorithm}{Algorithms}

\numberwithin{equation}{section}
\newtheorem{theorem}{Theorem}[section]
\newtheorem{proposition}[theorem]{Proposition}
\newtheorem{lemma}[theorem]{Lemma}
\newtheorem{claim}[theorem]{Claim}
\crefname{claim}{Claim}{Claims}
\newtheorem{observation}{Observation}
\newtheorem{corollary}[theorem]{Corollary}

\newtheorem*{question*}{Question}

\theoremstyle{definition}
\newtheorem{definition}[theorem]{Definition}

\newtheorem{phenomenon}[theorem]{Phenomenon}
\newtheorem{question}[theorem]{Question}
\newtheorem*{definition*}{Definition}

\theoremstyle{remark}
\newtheorem{remark}{Remark}

\DeclarePairedDelimiter{\ceil}{\lceil}{\rceil}
\DeclarePairedDelimiter{\floor}{\lfloor}{\rfloor}

\newcommand{\mb}{\mathbb}

\newcommand{\on}{\operatorname}

\newcommand{\ER}{Erd\H{o}s-R\'enyi }

\newcommand{\rank}{\mathrm{rank}}
\newcommand{\Span}{\mathrm{Span}}
\newcommand{\supp}{\mathrm{supp}}
\newcommand{\Bin}{\mathrm{Bin}}
\newcommand{\Ber}{\mathrm{Ber}}

\allowdisplaybreaks

\newcounter{subcase}[section]

\title{On the Rank, Kernel, and Core of Sparse Random Graphs}

\author[DeMichele]{Patrick DeMichele}

\address[DeMichele, Glasgow]{Department of Computer Science, Stanford University, Stanford, CA.}

\email{demichel@stanford.edu}

\author[Glasgow]{Margalit Glasgow}

\email{mglasgow@stanford.edu}

\thanks{M.G. supported by NSF award DGE-1656518.}

\author[Moreira]{Alexander Moreira}

\address[Moreira]{Department of Mathematics, Stanford University, Stanford, CA.}

\email{amoreira@stanford.edu}

\begin{document}
\maketitle

\begin{abstract}
We study the rank of the adjacency matrix $A$ of a random \ER graph $G\sim \mathbb{G}(n,p)$. It is well known that when $p = (\log(n) - \omega(1))/n$, with high probability, $A$ is singular. We prove that when $p = \omega(1/n)$, with high probability, the corank of $A$ is equal to the number of isolated vertices remaining in $G$ after the Karp-Sipser leaf-removal process, which removes vertices of degree one and their unique neighbor. We prove a similar result for the random matrix $B$, where all entries are independent Bernoulli random variables with parameter $p$. Namely, we show that if $H$ is the bipartite graph with bi-adjacency matrix $B$, then the corank of $B$ is with high probability equal to the max of the number of left isolated vertices and the number of right isolated vertices remaining after the Karp-Sipser leaf-removal process on $H$.
Additionally, we show that with high probability, the $k$-core of $\mathbb{G}(n, p)$ is full rank for any $k \geq 3$ and $p = \omega(1/n)$. This partially resolves a conjecture of Van Vu~\cite{vu_crmt} for $p = \omega(1/n)$. Finally, we give an application of the techniques in this paper to gradient coding, a problem in distributed computing.

\end{abstract}

\section{Introduction} \label{sec:introduction}

Let $A \sim \mathbb{A}(n, p)$ be a random symmetric matrix in $\{0, 1\}^{n \times n}$ where each entry above the diagonal is an i.i.d. Bernoulli random variable that equals $1$ with probability $p$ and $0$ otherwise, and where the diagonal entries are $0$. Equivalently, $A$ is the adjacency matrix of an \ER graph $G \sim \mathbb{G}(n, p)$ on $n$ nodes where an edge appears between each pair of nodes with probability $p$. 

Determining the rank of $A$ has been a question of considerable interest. Initial results by Tao, Costello, and Vu \cite{symmetric} proved that $A$ is non-singular with high probability if $p$ is a fixed constant. This result was later improved by Costello and Vu in \cite{cv_graphs} to show that $A$ is non-singular with high probability if $p > (1+\varepsilon)\log(n)/n$. This work also showed that $p = \log(n)/n$ is a threshold for singularity as $A$ is singular with high probability if $p < (1-\varepsilon)\log(n)/n$.  Recent work by Basak and Rudelson \cite{br} shows this threshold is sharp in the sense that if $p\geq \log(n)/n+\omega(1/n)$, then $A$ is non-singular with high probability, and if $p\leq \log(n)/n-\omega(1/n)$, then $A$ is singular with high probability. Furthermore, they prove the same result holds for the asymmetric matrix $B\sim \mathbb{B}(n,p)$ in $\{0,1\}^{n\times n}$ where each entry is an independent Bernoulli random variable with parameter $p$. More generally, an fundamental question in discrete random matrix theory, dating back to Komlos \cite{komlos}, studies the exact probability that a random matrix is singular. See the survey of Van Vu~\cite{vu_crmt} and the references therein for discussion of the singularity problem.

Following these threshold results, a natural question is determining the rank of $A$ and $B$ when $p$ is small enough that these random matrices are singular with high probability. One approach to understand the rank, previously emphasized by Costello and Vu, is through the following phenomenon: 
\begin{phenomenon}\label{phe:1}
Linear dependencies among the columns of random matrices arise from small structures.
\end{phenomenon}

Costello and Vu \cite{cv_sparse}  showed that when $p=\Omega(\log(n)/n)$, with high probability, all dependencies in the column space of $A$ result from collections of $k=\Theta(1)$ columns with at most $k-1$ non-zero rows. These so-called \textit{non-expanding sets} allowed Costello and Vu to prove that when $p=\Omega(\log(n)/n)$, with high probability,
\begin{equation*}
    \text{rank($A$)}= \min_{S\subseteq V(G(A))} (n-|S|+|N(S)|)
\end{equation*}
where $N(S)$ denotes the neighborhood of the vertex set $S$, $G(A)$ denotes the graph with adjacency matrix $A$, and $V(G)$ denotes the vertex set of a graph $G$.

Other work has shown additional instances of Phenomenon~\ref{phe:1}. Addario-Berry and Eslava~\cite{hitting} showed that with probability $1- o(1)$, if you expose the non-zero entries of $A$ or $B$ randomly one at time, the matrix becomes singular at the exact moment when any row or column of all zeros disappears. For $B \sim \mathbb{B}(n, p)$, with $\log(n)/n \leq p \leq 1/2$, Tikhomirov \cite{tikhomirov} and Jain et al. \cite{jain} show that the probability of singularity of $B$ is nearly exactly the probability that $B$ has an all-zero row or column (or two equivalent rows and columns when $p = 1/2$). Further work  \cite{jain_corank, bernoulli_rank} shows that for various ranges of $p$, the probability that $B$ has corank $k$ is nearly the probability that $k$ rows are all zero.

A second approach towards understanding the rank of $A \sim \mathbb{A}(n,  p)$ or $B \sim \mathbb{B}(n, p)$ is via the Karp-Sipser core of a graph, defined below. Originally introduced in \cite{ks} to study maximal matchings, the Karp-Sipser core of graph $G$ is obtained by iteratively performing a leaf-removal process, which at each step removes a vertex of degree one and its unique neighbor. At the end of the process, there remain only isolated vertices and a subgraph of minimum degree at least two, which we denote $G_{\on{KS}}$. Notably, this leaf-removal process, which we also call \em peeling\em,  does not change the corank of the graph\footnote{We say the corank of a graph to mean the corank of its adjacency matrix.}, and so the corank of $G$ equals the corank of the $G_{\on{KS}}$ plus the number of isolated vertices remaining after the leaf-removal process. The following lower bound on the corank of $A$ immediately follows:
\begin{equation}\label{eq:lb_A}
    \on{corank}(A) \geq |I_{\on{KS}}|,
\end{equation}
where $I_{\on{KS}}$ is the set of isolated vertices remaining after the Karp-Sipser leaf-removal on $G(A)$. For the asymmetric matrix $B$, a similar lower bound holds. We define $H(B)$ to be the bipartite graph with bi-adjacency matrix $B$ and left and right vertex sets $L$ and $R$ respectively. Then using the invariance of the corank during peeling, we obtain:
\begin{equation}\label{eq:lb_B}
    \on{corank}(B) \geq \max\left(|I_{\on{KS}} \cap L|, |I_{\on{KS}} \cap R| \right),
\end{equation}
where $I_{\on{KS}}$ is the set of isolated vertices remaining after the Karp-Sipser leaf-removal on $H(B)$.

A central question asks if these lower bounds on the corank are tight. Two important works yield results on this front for $p = \Theta(1/n)$. Bordenave, Lelarge, and Salez show in \cite{bls} that as $n$ grows, the corank of the Karp-Sipser core of $G(A)$ divided by $n$ tends to zero. That is, almost surely, $$\lim_{n \rightarrow \infty} \frac{\on{corank}(A) - |I_{\on{KS}}|}{n} = 0.$$ A related asymptotic formula was shown by Coja-Oghlan et al.~\cite{sparse_rank} for the asymmetric random matrix $B \sim \mathbb{B}(n, p)$. They showed\footnote{This follows from evaluating their rank formula and comparing to the asymptotic size they give for $I_{\on{KS}}$.} that in probability,
$$\lim_{n \rightarrow \infty} \frac{\on{corank}(B) - |I_{\on{KS}} \cap L|)}{n} = 0.$$
Both of these results are proved without explicitly connecting the corank to the size of $I_{\on{KS}}$ (or $I_{\on{KS}} \cap L$), but rather by showing that asymptotically these two values coincide. Due to their connections with the size of maximal matchings, the asymptotic size of $G_{\on{KS}}$ and $I_{\on{KS}}$ are well known, and can be studied via differential equations. Karp and Sipser \cite{ks} and the tighter analysis of Aronson, Frieze and Pittel in \cite{aronson} showed in that for $G \sim \mathbb{G}(n, p)$, in probability, $$\lim_{n \rightarrow \infty} \frac{|I_{\on{KS}}|}{n} = \frac{\gamma^* + \gamma_* + \gamma^*\gamma_* - 1}{pn},$$
where $\gamma_*$ is the smallest root to the equation $x = pn\exp(-pne^{-x})$, and $\gamma^* = pne^{-\gamma_*}$. Analogously, for the bipartite graph $H(B)$ with $B \sim \mathbb{B}(n, p)$, Coja-Oghlan et al. showed that in probability, $$\lim_{n \rightarrow \infty} \frac{|I_{\on{KS}} \cap L|}{n} = \frac{\gamma^* + \gamma_* + \gamma^*\gamma_* - 1}{pn}.$$

Our main result is to show that the lower bounds in eqs.~\ref{eq:lb_A} and \ref{eq:lb_B} are sharp for $p = \omega(1/n)$. In the vein of Phenomenon~\ref{phe:1}, we prove a new structural characterization of linear dependencies in $A$ and $B$, which holds with high probability, and tightens Costello and Vu's notion of non-expanding sets. On the event that this characterization holds, it follows almost directly that \em the Karp-Sipser leaf-removal and isolated vertex removal process is precisely equivalent to removing all vertices whose respective rows and columns are part of linear dependencies\em. Equivalently, on this event, the Karp-Sipser core of $A$ is full rank. We state this result in the following theorem.

\begin{restatable}[Rank of $A$]{theorem}{thmksA}\label{thm:ks}
Let $p = \omega(1/n)$ and $p \leq 1/2$ and $A \sim \mathbb{A}(n, p)$. Let $G = G(A)$. With probability $1 - o(1)$,
$$\on{corank}(A\left({G_{\on{KS}}}\right)) = 0.$$ 

Equivalently, with probability $1 - o(1)$,
$$\on{corank}(A) = |I_{\on{KS}}|,$$
where $I_{\on{KS}}$ is the set of isolated vertices remaining after the Karp-Sipser leaf-removal on $G$. 

\end{restatable}

For the asymmetric matrix $B$, the story is slightly complicated by the fact that typically, there are ``complex'' linear dependencies that involve large sets of rows or columns of $B$. These linear dependencies can not be characterized using the same simple structural characterization we prove for the symmetric matrix. Fortunately, we are able to a prove a modified characterization result which shows that with high probability, such ``complex'' linear dependencies only appear in the rows of $B$ or the columns of $B$, \em but not both\em. It follows on this event that the bi-adjacency matrix of the Karp-Sipser core of $H(B)$ has full row or column rank. We state this result and the resulting characterization of the rank of $B$ in the following theorem. 


\begin{restatable}[Rank of $B$]{theorem}{thmksB}\label{thm:ksB}
Let $p = \omega(1/n)$ and $p \leq 1/2$ and $B \sim \mathbb{B}(n, p)$. Let $H = H(B)$ be a bipartite graph with $n$ left vertices $L$ and $n$ right vertices $R$. With probability $1 - o(1)$, either 
\begin{enumerate}
    \item $\on{corank}(B({H_{\on{KS}}})) = 0,$ or
    \item $\on{corank}(B({H_{\on{KS}}})^\top) = 0.$
\end{enumerate}

It follows that with probability $1 - o(1)$,
$$\on{corank}(B) = \max(|I_{\text{KS}} \cap L|, |I_{\text{KS}} \cap R|),$$

where $I_{\on{KS}}$ is the set of isolated vertices remaining after the Karp-Sipser leaf-removal on $H$.
\end{restatable}

Our main tool in proving these theorems is a new characterization of all \em minimally \em linearly dependent sets of columns that occur with constant probability in $A$ and $B$.

\begin{definition}\label{def:minimal}
A matrix in $M \in \mathbb{R}^{m \times k}$ is a \em minimal linear dependency\em~if it satisfies the following two properties:
\begin{enumerate}
\item There exists $x \in \mathbb{R}^k$ such that $Mx = 0$ and $\on{supp}(x) = [k]$. 
\item 
$\on{rank}(M) = k-1$. 
\end{enumerate}
\end{definition}
We sometimes will abbreviate and call such a matrix a minimal dependency.

It is a simple exercise to show that every linearly dependent set of columns has a subset of columns which form a minimal linear dependency. One can additionally show that in any matrix $M$, if the $i$th column $M_i$ is in the span of the remaining columns $\{M_j\}_{j \neq i}$, then there exists some subset of columns of $M$ including $M_i$ which is a minimal dependency (see Lemma~\ref{lemma:support_minimal}). Let $$\mathcal{M}_k \subset \bigcup_{m \geq 1} \{0,1\}^{m \times k}$$ denote the set of \{0, 1\}-matrices of $k$ columns which are minimal linear dependencies. For our main results, we distinguish among $\mathcal{M}_k$ one particular class of minimal linear dependencies, illustrated in Figure~\ref{fig:dependencies_char}(a).  

\begin{figure}
    \centering
    \includegraphics[width=15cm]{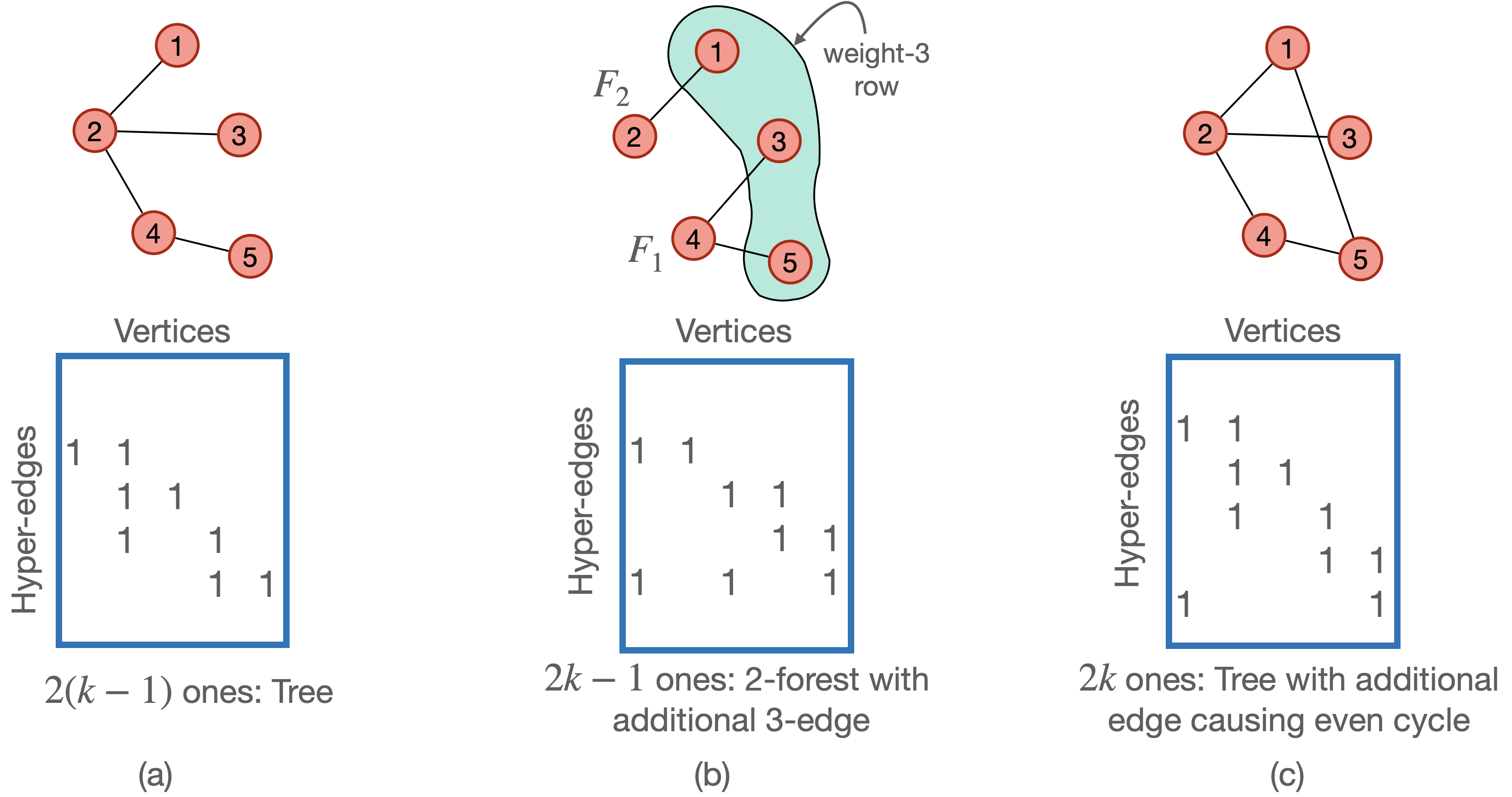}
    \caption{Instances of the three classes of minimal dependencies of $k = 5$ columns that occur with constant probability in $A$ and $B$ with $p = \Theta(1/n)$. For $p = \omega(1/n)$, with probability $1-o(1)$ only tree dependencies (a) appear. (a) An element of $\mathcal{T}_k$. (b) An element of $\mathcal{T}_k^+$ (see Section~\ref{sec:small}). (c) An element of $\mathcal{T}_k^C$ (see Section~\ref{sec:small}).}\label{fig:dependencies_char}
\end{figure}

\begin{definition}[Tree dependency]\label{def:tree} Define $\mathcal{T}_k$ as the set of matrices $M\in \bigcup_m \{0,1\}^{m\times k}$ with exactly $2k - 2$ $1$'s such that the non-zero columns of $M$ form the edge-vertex incidence matrix of a tree on $k$ vertices. We call a matrix in $\cup_k \mathcal{T}_k$ a \em tree dependency\em.
\end{definition}

The following theorem shows that with high probability, all minimal dependencies in $A$ are tree dependencies. 

\begin{restatable}[Characterization of minimal dependencies in $A$]{theorem}{charsquare}\label{char:square}
Let $A \sim \mathbb{A}(n, p)$, where $p = \omega(1/n)$ and $p \leq 1/2$. With probability $1 - o(1)$, for all $k$, all minimal dependencies of $k$ columns of $A$ are in $\mathcal{T}_k$.
\end{restatable}
Similarly, either all subsets of columns of $B$ or of $B^\top$ that are minimal linear dependencies, are tree dependencies. 
\begin{restatable}[Characterization of minimal dependencies in $B$]{theorem}{charbgc}\label{char:bgc}
Let $B \sim \mathbb{B}(n, p)$, where $p = \omega(1/n)$ and $p \leq 1/2$. With probability $1 - o(1)$, one of the following holds:
\begin{enumerate}
    \item For all $k$, all minimal dependencies of $k$ columns of $B$ are in $\mathcal{T}_k$.
    \item For all $k$, all minimal dependencies of $k$ columns of $B^\top$ are in $\mathcal{T}_k$.
\end{enumerate} 
\end{restatable}

\begin{remark}
We do not expect the theorem above can be improved, in the sense that often $B$ or $B^\top$ do have minimal linear dependencies with $\Theta(n)$ columns. Intuitively, this is because short and fat rectangular matrices contain minimal dependencies with many columns. When we peel off the rows and columns involved in small minimal linear dependencies, the remaining matrix, $B(H_{\on{KS}})$, may be rectangular if there are more small minimal linear dependencies in the rows than in the columns (or vice versa). This forces a minimal dependency with many columns in $B$ (or $B^\top$).
\end{remark}

\begin{remark}
When $p=\Theta(1/n)$, Theorems~\ref{char:square} and \ref{char:bgc} fail due to the presence of two other classes of minimal dependencies which appear with constant probability in $A$ and $B$. These two classes are pictured in Figure~\ref{fig:dependencies_char}(b)(c) and defined formally in Section~\ref{sec:small}. 
\end{remark}

\subsection{Additional Results}

A second substructure of random graphs which has received considerable attention is the $k$-core of $G$: the largest induced subgraph of $G$ with minimal degree $k$. The $k$-core, which we denote $\on{core}_k(G)$, can be obtained by iteratively peeling off vertices of degree less than $k$. When the average degree of $G$ is greater than some constant $c = c(k)$, the $k$-core suddenly emerges and has linear size in the number of vertices of $G$ \cite{cores}. The exact size of the $k$-core is given by Pittel, Spencer, and Wormald in \cite{pittel}. 

Van Vu conjectured in \cite{vu_crmt} that for $p = \Theta(1/n)$, for a constant $k \geq 3$, almost surely, the adjacency matrix of $\on{core}_k(G)$ is non-singular. We prove the following result, which confirms the conjecture of Vu for $p = \omega(1/n)$. Let $\on{core}_k(A)$ denote the adjacency matrix of the $k$-core of the graph with adjacency matrix $A$.
\begin{restatable}{theorem}{thmcore}\label{thm:core}
Let $A \sim \mathbb{A}(n, p)$, where $p = \omega(1/n)$ and $p \leq \frac{3\log(n)}{n}$. For any constant $k \geq 3$, with probability $1 - o(1)$,
$$\on{corank}(\on{core}_k(A)) = 0.$$ 
\end{restatable}

We remark that Ferber et al. \cite{core_fkss} recently resolved the conjecture of Vu, after a preliminary version of the present paper was made publicly available. This was independent work, and the proof methods are quite different.

We additionally show an application of our techniques to a problem in distributed computing. For this application, we show that our characterization of minimal dependencies also holds for some ensembles of sparse rectangular random matrices. Using this characterization, we can bound the distance from a fixed vector to the span of a random matrix. This is useful for analyzing \em gradient codes\em, an object used to provide redundancy in distributed computing. We discuss the particular results in Section~\ref{sec:gc}.

\subsection{Notation}
For a graph $G$, we use $A(G)$ to denote the adjacency matrix of $G$. We use $G(A)$ to denote the graph with adjacency matrix $A$. For a bipartite graph $H = (L \cup R, E)$, we use $B(H)$ to denote the bi-adjacency matrix of $H$, where $L$ indexes rows of $B(H)$, and $R$ indexes columns. We use $H(B)$ to denote the bipartite graph with bi-adjacency matrix $B$. We use $V(G)$ or $V(A)$ to denote the vertex set of a graph $G$ or $G(A)$.

For $n \in \mathbb{N}$, let $[n]$ denote the set $\{1, \ldots , n\}$. For a matrix $M \in \mathbb{R}^{m \times n}$ and $S \subseteq [n]$, we let $M_S$ denote the submatrix of $M$ with columns indexed by $S$. For $S \subseteq [m]$, we let $M^S$ denote the submatrix of $M$ with rows indexed by $S$. Furthermore, for an index $i\in \mathbb{N}$, we let $M_i$ denote the $i$th column of $M$ and we let $M_{\backslash i}$ denote $M$ with the $i$th column removed. We define $M^{\backslash i}$ for rows analogously. We let $M^{(i)}$ denote $M$ with both the $i$th row and $i$th column removed. Finally, we let $M^{\top}$ denote the transpose of $M$. We say the corank of $M$ equals $m - \on{rank}(M)$.

For vectors $x \in \mathbb{R}^n$ and $y \in \mathbb{R}^m$, let  $(x, y) \in \mathbb{R}^{m + n}$ denote the concatenation of $x$ and $y$. We use $e_i$ to denote the $i$th standard basis vector.


\subsection{Outline}
In Section~\ref{sec:overview}, we give a technical overview of the proofs of the main theorems. We break down the proofs of Theorems~\ref{char:square} and \ref{char:bgc} thematically into Sections~\ref{sec:small} and \ref{sec:large}, where we prove the main technical lemmas. We conclude the proof in Section~\ref{sec:char}. In Section~\ref{sec:ks}, we prove Theorems~\ref{thm:ks} and \ref{thm:ksB} using Theorems~\ref{char:square} and \ref{char:bgc}.

\section{Technical Overview}\label{sec:overview}

In this section, we give a high level overview of the proof of the main theorems in this paper.

\subsection{Overview of Proposition~\ref{char:square} and Proposition~\ref{char:bgc}}  For simplicity, in this overview we primarily focus on the asymmetric matrix $B \sim \mathbb{B}(n, p)$, where $p = d/n$.  We provide comments on when and how the techniques differ for the symmetric matrix $A \sim \mathbb{A}(n, p)$.

Our proof of each characterization result is divided into two domains: a ``small'' domain and a ``large'' domain. The small domain considers possible kernel vectors with support at most $\frac{n}{d}$. The large domain considers possible kernel vectors with support at least $\frac{n}{d}$.  Thematically, this is similar to the more-involved casework based on levels of ``structuredness'' of vectors used classically to show that the kernel of $A$ is empty for $p > \log(n)/n$. (See eg.~\cite{invertibility, tikhomirov, jain})

\subsubsection{Small Domain}
The main idea for the small domain comes from the following two simple observations. 

\begin{observation}\label{observation1}
Let $M \in \{0, 1\}^{m \times k}$ be a matrix. Then if some row of $M$ contains exactly one $1$, then there is no $x \in \mathbb{R}^k$ such that $Mx = 0$ and $\supp(x) = [k]$.
\end{observation}
\begin{proof} If there is only one column of $M$ with a $1$ at a certain row-index, that column is linearly independent of all other columns in $M$. Thus, no dependency including all columns of $M$ exists. \end{proof}

\begin{observation}\label{observation2}
Let $M \in \{0, 1\}^{m \times k}$ be a matrix.
If $M$ has support size less than $2k-2$, then $M$ is not a minimal linear dependency.
\end{observation}
\begin{proof} The requirement that $M$ is rank $k-1$ means that at least $k-1$ rows must have at least one $1$ in them. Hence it follows from Observation~\ref{observation1} that each of these rows must contain at least two 1s. Thus, if $M$ is a minimal dependency, it contains at least $2k-2$ $1$'s.
\end{proof}

Within the small domain, we subdivide into two cases, again depending on the size of the support of the kernel vector. The results for the small domain are proved in Section~\ref{sec:small}.

\paragraph{\bf{Small Case 1}}
The goal of this case is to prove that with high probability, for any $k \leq \Theta(n/d^2)$, and $S \subset [n]$ with $|S| = k$, if $B_S$ is a minimal dependency, then $B_S \in \mathcal{T}_k$. We will take a union bound over all such sets $S$ to prove this. Fix $S$ with $|S| = k$, and define $L$ to be the number of $1$s in $B_S$. Due to Observation~\ref{observation2}, we only must consider the case when $L \geq 2k - 2$. Further, if $L = 2k - 2$, then if $B_S$ is a minimal dependency, it must also be in $\mathcal{T}_k$ (see Lemma $\ref{lemma:classification}$).  

By Observation~\ref{observation1}, $B_S$ is not a minimal dependency if there exists a row with exactly one $1$ in $B_S$. To lower bound this probability, we condition on $L$, consider the random process which exposes the $L$ positions of $1$s in $B_S$ randomly one at a time. Consider the random walk which increases by $1$ every time we expose a $1$ in an already occupied row, and stays constant otherwise. As long as the value of this random walk is less than $L/2$ at the time all $L$ $1$s are exposed, we know that $B_S$ is not a minimal dependency. When $L \geq 2k - 1$, this occurs with small enough probability that after a union bound on all sets $S$, we succeed with probability $1-o(1)$.

For the case of symmetric matrices, we modify this argument slightly by considering the random walk which also increases every time a $1$ is exposed in the ``symmetric'' portion of $A_S$, that is, a row of $A_S$ indexed by an element of $S$. 

\paragraph{\bf{Small Case 2}}
This case rules out with high probability any minimal dependencies of $k$ columns of $B$ for $\frac{n}{d^2} \lesssim k \lesssim \frac{n}{d}$. Again we use Observation~\ref{observation1} and take a union bound over all submatrices $B_S$ for sets $S$ of size $k$. We use a Chernoff bound to show that with very high probability, $B_S$ has a row with a single $1$.

\subsubsection{Large Domain}
As $k$ approaches $n/d$, we can no longer expect $B_S$ to have rows with a single $1$ with high enough probability, so we need to look beyond Observation~\ref{observation1} to rule out linear dependencies. Our main tool in ruling out larger minimal dependencies are the following Littlewood-Offord anti-concentration results for sparse vectors, from Costello and Vu~\cite{cv_graphs}.




\begin{restatable}[c.f. \cite{cv_graphs} Lemma 8.2]{lemma}{losparse}\label{lo_sparse}
Let $v \in \mathbb{R}^n$ be a deterministic vector with support at least $m$. Let $z \in \mathbb{R}^n$ be a random vector with i.i.d.~Bernoulli entries with parameter $p \leq 1/2$. For any $c$,
$$\Pr\left[v^\top z=c\right] \leq \left(\frac{1}{\sqrt{\pi/2}}\right)\frac{1}{\sqrt{mp}}+\left(e^{(\ln(2)-1)mp}\right).$$
In particular, for $mp\geq 9$, we have:
$$\Pr\left[ v^\top z=c\right] \leq \frac{1}{\sqrt{mp}}.$$
\end{restatable}

In the symmetric matrix case, we will additionally use a quadratic Littlewood-Offord result, originally due to Costello, Tao, and Vu~\cite{symmetric}, and adapted to sparse random vectors in \cite{cv_graphs}. 

\begin{lemma}[c.f. \cite{cv_graphs} Lemma 8.4]\label{lo_quadratic}
Let $M \in \mathbb{R}^{n \times n}$ be a deterministic matrix with a least $m$ non-zero entries in each of $m$ distinct columns of $M$. Let $z \in \mathbb{R}^n$ be the random vector with i.i.d.~Bernoulli entries with parameter $p \leq 1/2$. Then for any fixed $c$,
$$\Pr\left(z^\top Mz = c\right) = O\left(\frac{1}{(mp)^{1/4}}\right).$$
\end{lemma}



Again we subdivide into two cases.
\paragraph{\bf{Large Case 1}}
This case rules out with high probability minimal dependencies of $k$ columns of $B$ for $n/d \lesssim k \leq n/c$, for some constant $c$. In this case, we again perform a union bound over all submatrices $B_S$ for sets $S$ of size $k$. For each $S$, we consider the random process $B^{[1]}_S, B^{[2]}_S, \cdots B^{[n]}_S = B_S$, where we expose rows of $B_S$ one at a time. (Recall that $B^{[i]}_S$ is the submatrix of $B$ given by restricting to the columns in $S$ and the first $i$ rows.) At each step $i$, we keep track of the kernel of $B_S^{[i]}$. Before any rows are exposed, the dimension of the kernel is equal to $k$. Since our goal is to show that the kernel of $B_S$ contains no vectors with support size $k$, we leverage Lemma~\ref{lo_sparse} as follows: if the kernel of $B^{[i]}_S$ contains a vector $v$ with support $k$, then it is likely that the next row exposed, $B_S^{\{i+1\}}$, will \em not \em be orthogonal to $v$. After exposing enough rows, we show that with high probability, we ``knock out'' all candidate kernel vectors with large support. Notice that we only need to ``knock out'' such vectors at most $k$ times.

This approach of Large Case 1 breaks down when $k$ approaches $n$. For instance, in the extreme case when $k = n$, we would need to ``knock-out'' a kernel vector \em every single time \em we expose a row of $B$. This certainly won't occur with high probability. Hence we use a different approach when $k \geq n/c$ for some constant $c$.  

\paragraph{\bf{Large Case 2}}
Our strategy for ruling out dependencies on the order of $\Theta(n)$ columns is inspired by the approach of Ferber, Kwan, Sauermann~\cite{simple} to show the almost sure non-singularity of $B$ when $p > \log(n)/n$. For the symmetric matrix $A$, we combine these ideas with some techniques used by Costello, Tau, and Vu in~\cite{symmetric} to show almost sure non-singularity of $A$ when $p = 1/2$. The following simple observation is key to our argument for both $A$ and $B$.
\begin{observation}\label{obs_3}
Let $M \in \mathbb{R}^{n \times n}$ be a matrix. If $Mv = 0$, and $v$ has support size $k$, then there are at least $k$ columns $M_i$ such that $M_i \in \on{Span}(\{M_j\}_{j \neq i})$.
\end{observation}

For the asymmetric case, recall that our goal is to show that either $B$ or $B^\top$ has no kernel vector with support size at least $n/c$. Let $\mathcal{E}_i$ be the event that $B$ has a \em left \em kernel vector with support size at least $n/c$, \em and \em that column $B_i$ is in the span of the remaining columns $\{B_j\}_{j \neq i}$. Then by Markov's Law and Observation~\ref{obs_3}, the probability that both $B$ and $B^\top$ have kernel vectors with support size at least $n/c$ is at most $c\mathbb{P}[\mathcal{E}_i]$. To bound this probability, consider the matrix $B_{/i}$ formed by removing the $i$th column of $B$. We define $\mathcal{F}_i$ to be the event that $B_{/i}$ has a left kernel vector with support size at least $n/c$, \em and \em that column $B_i$ is in the span of the remaining columns $\{B_j\}_{j \neq i}$. Observe that any left kernel vector of $B$ is a left kernel vector of $B_{/i}$, and so $\mathcal{E}_i \subseteq \mathcal{F}_i$. To bound the probability of $\mathcal{F}_i$, we condition on $B_{/i}$ and  leverage the independence of $B_i$ from $B_{/i}$. Assume $B_{/i}$ has some left kernel vector $v$ with $|\on{supp}(v)| \geq n/c$. Lemma~\ref{lo_sparse} tells us that with probability $1 - O(1/\sqrt{d})$, the column $B_i$ is not orthogonal to $v$, and hence $B_i$ is not in the span of the remaining columns $\{B_j\}_{j \neq i}$. It follows that $\Pr[\mathcal{E}_i] \leq \Pr[\mathcal{F}_i] = O(1/\sqrt{d})$.

In the symmetric case, we let $\mathcal{E}_i$ be the event that the column $A_i$ is in the span of the remaining columns $\{A_j\}_{j \neq i}$. Again we use Markov's Law and Observation~\ref{obs_3} to show that the probability that $A$ has a kernel vector of support size of at least $n/c$ is at most $c\mathbb{P}[\mathcal{E}_n]$. To bound $\mathbb{P}[\mathcal{E}_n]$, we condition on $A^{(n)}$, the matrix formed by removing the $n$th row and column of $A$, and leverage the randomness of $A_n$. 

We consider two main cases: one in which $A^{(n)}$ has a kernel vector with large support, and one in which it doesn't. If $A^{(n)}$ has a kernel vector $v$ with large support (on the order of $\Theta(n)$), we use Lemma~\ref{lo_sparse} to show that with probability $1 - O(1/\sqrt{d})$, $A_n^{/n}$ is orthogonal to $v$, and hence $A_n$ is not in the span of the remaining columns $\{A_j\}_{j \neq n}$ (recall that $A_n^{/n}$ denotes the vector with all but the $n$th entry $A_n$). If $A^{(n)}$ has no kernel vectors with large support, we construct a ``pseudoinverse'' $P$ for $A^{(n)}$, for which $(A_n^{/n})^\top PA_n^{/n} \neq 0$ implies that $A_n$ is not in the span of $\{A_j\}_{j \neq n}$. We leverage some additional results from the small domain and Large Case 1 to show that $P$ is dense. Lemma~\ref{lo_quadratic} then guarantees that $(A_n^{/n})^\top PA_n^{/n} \neq 0$ with probability $1 - O(1/\sqrt[4]{d})$.

The results for the large domain are proved in Section~\ref{sec:large}.

\begin{remark}
The roadmap of this proof shows the two places where our techniques break down when $d$ is a constantFirst, in Small Case 1, our union bound only succeeds in ruling out minimal dependencies outside of $\mathcal{T}_k$ with probability $e^{-\Theta(dk)}$. When $k$ and $d$ are constant, this probability no longer goes to $0$. One can check that with constant probability, there are some minimal dependencies outside $\mathcal{T}_k$, but they are contained in two other classes of special structures, which we formally define in Section~\ref{sec:small} as $\mathcal{T}_k^+$ and $\mathcal{T}_k^C$ (pictured in Figure~\ref{fig:dependencies_char}).
Second, our results for Large Case 2, based on the Littlewood-Offord lemmas, require $d$ to go to infinity to get a result which holds with probability that tends to $1$.
\end{remark}

\subsection{Overview of Theorem~\ref{thm:ks} and Theorem~\ref{thm:ksB}} Theorems~\ref{thm:ks} and \ref{thm:ksB} follow directly by combining two linear-algebraic claims about the Karp-Sipser leaf-removal process with the results of Theorems~\ref{char:square} and \ref{char:bgc}. Specifically, to prove Theorem~\ref{thm:ks} on the rank of $A$, we use the following three claims:
\begin{enumerate}
    \item The corank of a matrix is invariant under the Karp-Sipser leaf removal process.
    \item If the Karp-Sipser core $A_{\on{KS}}$ has a vector in $x$ its kernel, then there must be a kernel vector $y$ of $A$ whose support contains the support of $x$.
    \item If $Ay = 0$, then by the characterization of minimal dependencies in Theorem~\ref{char:square}, for all $i \in \on{supp}(y)$, the vertex $i$ will be removed during the Karp-Sipser leaf-removal process or become isolated after this process. Thus, vertex $i$ will not be in the Karp-Sipser core of $A$.
\end{enumerate}
The proof of Theorem~\ref{thm:ksB} on the rank of $B$ is similar. These claims are proved in Section~\ref{sec:ks}. 

\subsection{Overview of Theorem~\ref{thm:core}}
To prove Theorem~\ref{thm:core}, we need to show that the kernel of $\on{core}_k(A)$ is empty. We use the same four cases as in the proof of Theorems~\ref{char:square} and \ref{char:bgc} above. 

The techniques used to rule out minimal dependencies in each of these cases is largely the same as before. In the small domain, we leverage the fact that any subset of $\ell$ columns of $\on{core}_k(A)$ for $k \geq 3$ must have at least $3\ell$ $1$'s.

In Large Case 2, we adapt our argument in a more involved way. We define $\mathcal{E}_i$ to be the event that $\on{core}_k(A)_i$ is in the span of the remaining columns of $\on{core}_k(A)$ \em and \em that the $k$-core is ``normal'' with respect to vertex $i$. Formally, we say the $k$-core is normal with respect to $i$ if vertex $i$ is in the $k$-core, and if we remove vertex $i$ from the $k$-core, the remaining graph still has minimum degree at least $k$. Equivalently and symbolically, this means
$$V(\text{core}_k(A)) = i \cup V(\text{core}_k(A^{(i)})).$$

We proceed to show that the $k$-core is normal with respect to all but $o(n)$ vertices. It follows from Observation~\ref{obs_3} and Markov's Law, that the probability that $\on{core}_k(A)$ has a kernel vector with support size at least $n/c$ is at most $c\Pr[\mathcal{E}_n] + o(1)$. 

We can bound $\Pr[\mathcal{E}_n]$ by conditioning on $A^{(n)}$ and leveraging the independent randomness of $A_n$. In a similar vein as before, we consider two main cases: one in which $\text{core}_k(A^{(n)})$ has a kernel vector with large support, and one in which it doesn't. The logic is similar as before.

We prove these results in Section~\ref{sec:core}.

\section{The Small Domain}\label{sec:small}
\subsection{Definitions and Lemmas for Characterizing Minimal Dependencies}
We gather here a few definitions and lemmas that we use in the small domain. We summarize notation in Table~\ref{NotationTable}.

\begin{remark}
While it is sufficient for our main results to consider only tree dependencies (Definition~\ref{def:tree}), we state our results in the small domain to hold for $d \geq d_0$, for $d_0$ a universal constant. For constant values of $d$, with constant probability, $A$ and $B$ include small minimal dependencies that are not tree dependencies, defined below. We believe understanding these dependencies may be of independent interest.
\end{remark}

\begin{table}[t]
\begin{center}
\begin{tabular}{| c | c | }
  \hline 
  $\mathcal{M}_k$    & Set of minimal dependencies with $k$ columns (Definition~\ref{def:minimal}) \\
  \hline
  $\mathcal{T}_k$    & Set of tree dependencies of $k$ columns (Definition~\ref{def:tree}, Figure~\ref{fig:dependencies_char}(a)) \\
  \hline
  $\mathcal{T}_k^+$    & Set of two-forest dependencies of $k$ columns (Definition~\ref{def:forest}, Figure~\ref{fig:dependencies_char}(b)) \\
  \hline
  $\mathcal{T}_k^C$    & Set of tree-with-added-edge dependencies of $k$ columns (Definition~\ref{def:cycle}, Figure~\ref{fig:dependencies_char}(c)) \\
  \hline
$\mathcal{S}_{L, k}$    & Subset of matrices  in $\mathcal{M}_k$ with exactly $L$ ones. (Definition~\ref{def:Sell}) \\
  \hline
  $\mathcal{S}_{L, k}'$    & Subset of matrices  in $\mathcal{M}_k$ with exactly $L$ ones and exactly $k$ non-zero columns. (Definition~\ref{def:Sellprime}) \\
  \hline
\end{tabular}
\caption{Notation in this work}\label{NotationTable}
\end{center}
\vspace{-0.8cm}
\end{table}

\begin{definition}[Two-forest dependency]\label{def:forest}
Let $\mathcal{T}_k^+$ be the set of matrices $M\in \bigcup_m \{0,1\}^{m\times k}$ with exactly $2k - 1$ $1$'s satisfying the following:

\begin{enumerate}
    \item $M$ has $k-1$ non-zero rows: $k-2$ rows supported on two entries and one row supported on three entries.
    \item The submatrix of $M$ restricted to the rows of support $2$ is the edge-vertex incidence matrix of a forest $F$ with two connected components $F_1,F_2$.
    \item The row of support size three contains $1$'s at column indices $a,b,c$ where $a,b \in F_i$ are connected by an even-length path, and $c \in F_j$ for $\{i, j\} = \{1, 2\}$.
\end{enumerate}
\end{definition}

\begin{definition}[Tree-with-added-edge dependency]\label{def:cycle}
Define $\mathcal{T}_k^C$ as the set of matrices $M\in \bigcup_m \{0,1\}^{m\times k}$ with $2k$ $1$'s such that the non-zero columns of $M$ form the edge-vertex incidence matrix of a tree on $k$ vertices, with an added edge between two vertices in the tree of odd distance from each other. The additional edge may create a multi-edge in the this graph.
\end{definition}

The first lemma allows us to show that the minimal dependencies we encounter with constant probability have the graph-structures depicted in Figure~\ref{fig:dependencies_char}.
\begin{definition}\label{def:Sell}
Define $\mathcal{S}_{L,k}$ to be the set of matrices $M \in \mathcal{M}_k$ which contain exactly $L$ ones. 
\end{definition}

\begin{definition}\label{def:Sellprime}
Define $\mathcal{S}_{L, k}' \subset \mathcal{S}_{L, k}$ to be the subset of matrices in $\mathcal{S}_{L, k}$ which have exactly $k$ non-zero rows. 
\end{definition}

The following lemma relates these sets to the sets $\mathcal{T}_k, \mathcal{T}_k^+$ and $\mathcal{T}_k^C$ introduced earlier in Definitions~\ref{def:tree}, \ref{def:forest}, and \ref{def:cycle}.

\begin{restatable}[Classification of Dependencies]{lemma}{lemmaclassification}\label{lemma:classification}
We have the following three equivalences:
\begin{enumerate}
    \item $\mathcal{S}_{2k-2, k} = \mathcal{T}_k$.
    \item $\mathcal{S}_{2k-1, k} = \mathcal{T}_k^+$.
    \item $\mathcal{S}_{2k, k}' = \mathcal{T}_k^C$.
\end{enumerate}
\end{restatable}


Lemma~\ref{lemma:classification} is proved in Appendix~\ref{apx:matrices}. The main step in its proof is the following ``connectivity'' lemma, also proved in Appendix~\ref{apx:matrices}.

\begin{restatable}[Connectivity of Minimal Dependencies]{lemma}{lemmaconnect}\label{lemma:connect}
Let $M \in \mathcal{M}_k$, and let $G$ be the hypergraph on $k$ vertices with hyperedge-vertex incidence matrix $M$, where the rows of $M$ index hyperedges, and the columns index vertices. Then all vertices in $G$ are connected.
\end{restatable}

\subsection{Small Case 1}

Our main goal in the Small Case 1 is to prove the following general proposition, which holds for $d$ as small as $1$. The corollary that follows contains the necessary ingredients from Small Case 1 for proving Theorems~\ref{char:square} and \ref{char:bgc}.

\begin{proposition}[Small Case 1]\label{prop:small}
Let $M$ be any of the following random matrices for $d \geq 1$:
\begin{enumerate}
    \item $M = A$, where $A \sim \mathbb{A}(n, \frac{d}{n})$
    \item $M = A^{[n - 1]}$, where $A \sim \mathbb{A}(n, \frac{d}{n})$
    \item $M = B$, where $B \sim \mathbb{B}(n, \frac{d}{n})$
\end{enumerate}
There exists a universal constant $c_{\ref{prop:small}}$ such that for any set $S \subset [n]$ of size $k\in[1,\frac{n}{8e^4d^2}]$, we have

\begin{enumerate}
    \item $\Pr[M_S \in \mathcal{M}_k \setminus \left(\mathcal{T}_k \cup \mathcal{T}_{k}^+ \cup \mathcal{T}_{k}^C \right)] =\left(e^{- d + c_{\ref{prop:small}}\log(d)}\right)^k\left(\frac{k}{n}\right)^{k+1}$.
    \item  $\Pr[M_S \in \mathcal{T}_k^+]\leq \left(\left(e^{- d+c_{\ref{prop:small}}\log( d)}\right)\left(\frac{k}{n}\right)\right)^k$.
    \item  $\Pr[M_S \in \mathcal{T}_k^C]\leq \left(\left(e^{- d+c_{\ref{prop:small}}\log( d)}\right)\left(\frac{k}{n}\right)\right)^k$.
    \item $\Pr[M_S \in \mathcal{T}_k]\leq  \left(e^{- d+c_{\ref{prop:small}}\log( d)}\right)^k\left(\frac{k}{n}\right)^{k-1}$.
\end{enumerate}

\end{proposition}
Union bounding over the sets $S$ yields the following corollary when $d = \omega(1)$.
\begin{corollary}\label{cor:small}
Let $A \sim \mathbb{A}(n, \frac{d}{n})$ and $B \sim \mathbb{B}(n, \frac{d}{n})$ for $d = \omega(1)$. Then each of the following hold with probability $1 - o(1)$:
\begin{enumerate}
    \item $\forall S \in [n]$ with $k = |S| \leq \frac{n}{8e^4d^2}$, $B_S \notin \mathcal{M}_k \setminus \mathcal{T}_k$.
    \item $\forall S \in [n]$ with $k = |S| \leq \frac{n}{8e^4d^2}$, $A_S \notin \mathcal{M}_k \setminus \mathcal{T}_k$.
    \item $\forall x$ with $|\supp(x)| \leq \frac{n}{8e^4d^2}$ and $n \in \supp(x)$, $A^{[n - 1]}x \neq 0$.
\end{enumerate}
\end{corollary}

\begin{proof}[Proof of Proposition~\ref{prop:small}]
We introduce some notation, pictured in Figure~\ref{fig:symmetric}. Fix a set $S$ of $k$ columns and consider the submatrix $M_S$ induced by these columns. If $M \in \{A, A^{[n-1]}\}$, we define the following partition of the entries of $M_S$. Let $E_{Sym}$ be the set of entries of $A_S$ whose rows are indexed by values in $S$. Let $E_{SymBD}$ be subset of entries in $E_{Sym}$ that are below the diagonal of $A$, and hence mutually independent. Let $E_{Asym}$ be the set of entries whose rows are not in $S$, and finally, let $E = E_{Asym} \cup E_{SymBD}$ be the full set of mutually independent entries that determine $M_S$. If $M = B$, we define $E_{Asym}$ to be all of the entries in $M_S$. Formally:
\begin{equation*}
\begin{split}
&E_{Sym} := \begin{cases}\{(i, j): i, j \in S \} & M \in \{A, A^{[n-1]}\} \\ \emptyset & M = B\end{cases}\\
&E_{SymBD} := \begin{cases}\{(i, j): i, j \in S, j < i \} & M \in \{A, A^{[n-1]}\} \\ \emptyset & M = B\end{cases}\\
&E_{Asym} := \begin{cases}\{(i, j): j \in S, i \notin S \} & M \in \{A, A^{[n-1]}\} \\ \{(i, j): j \in S\} & M = B\end{cases}\\
&E := E_{SymBD} \cup E_{Asym}
\end{split}
\end{equation*}
\begin{figure}
        \centering
        \begin{tabular}{cccc}
                 \includegraphics[width=4cm]{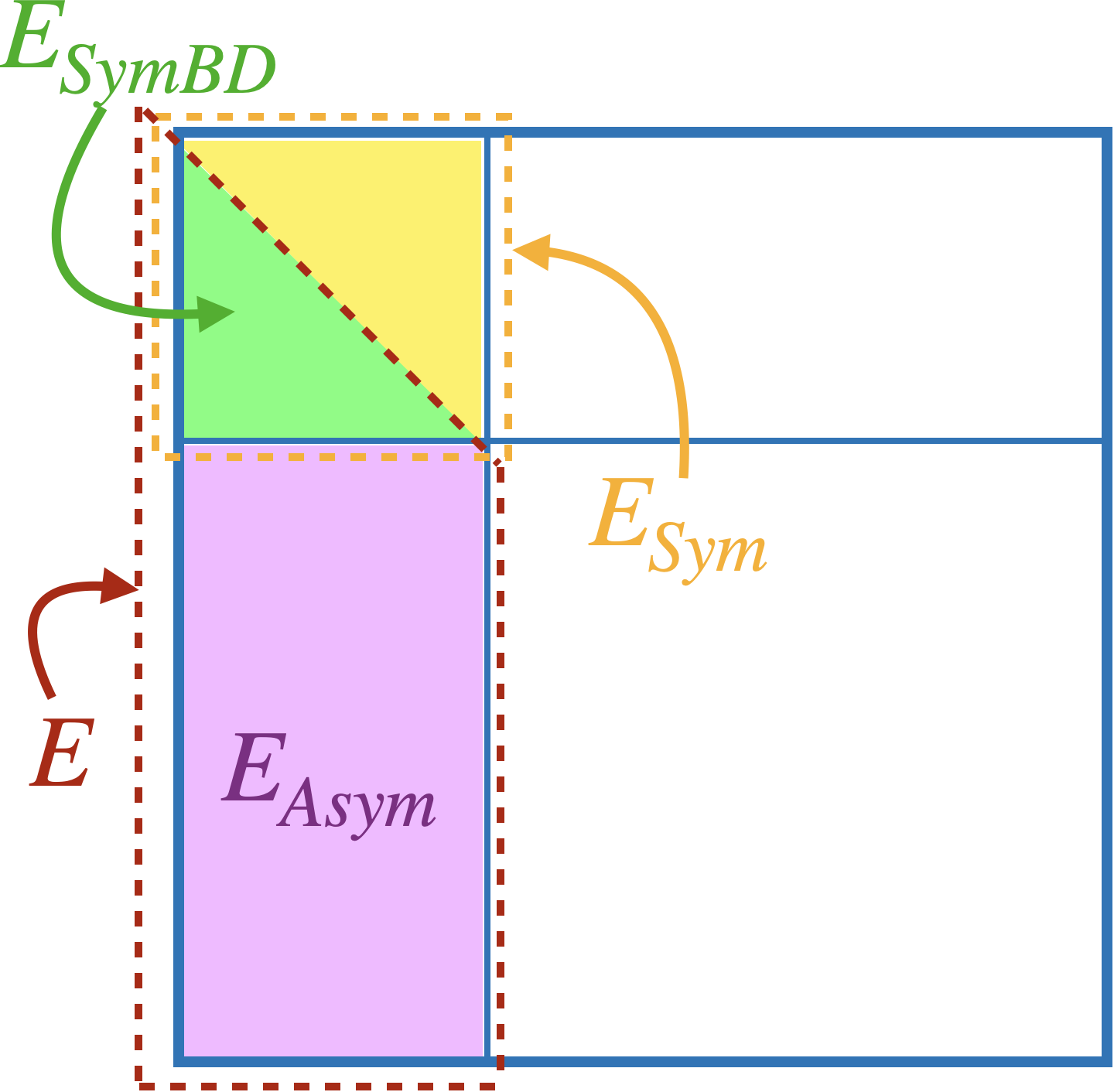} & \includegraphics[width=4cm]{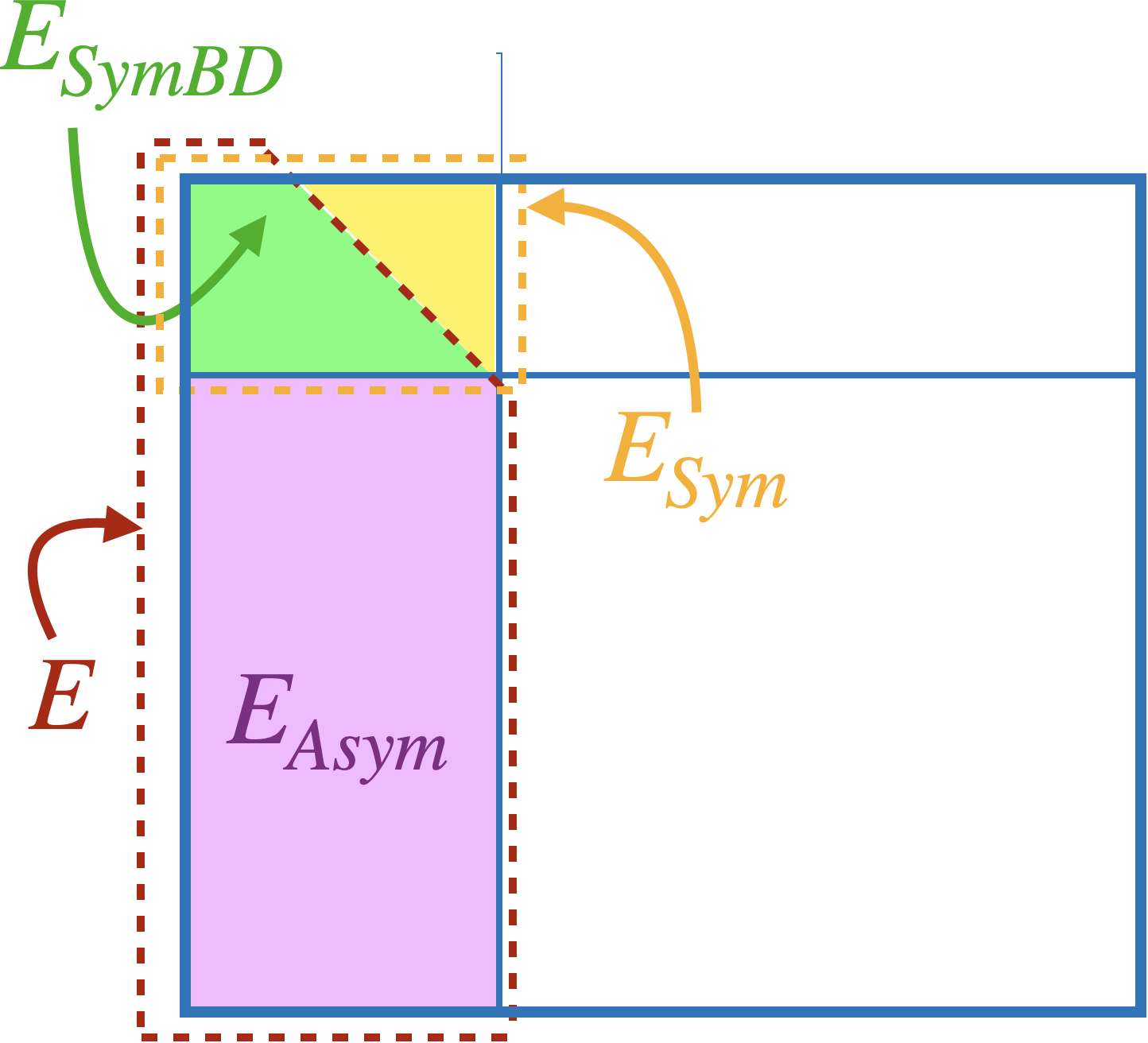}  & \includegraphics[width=4cm]{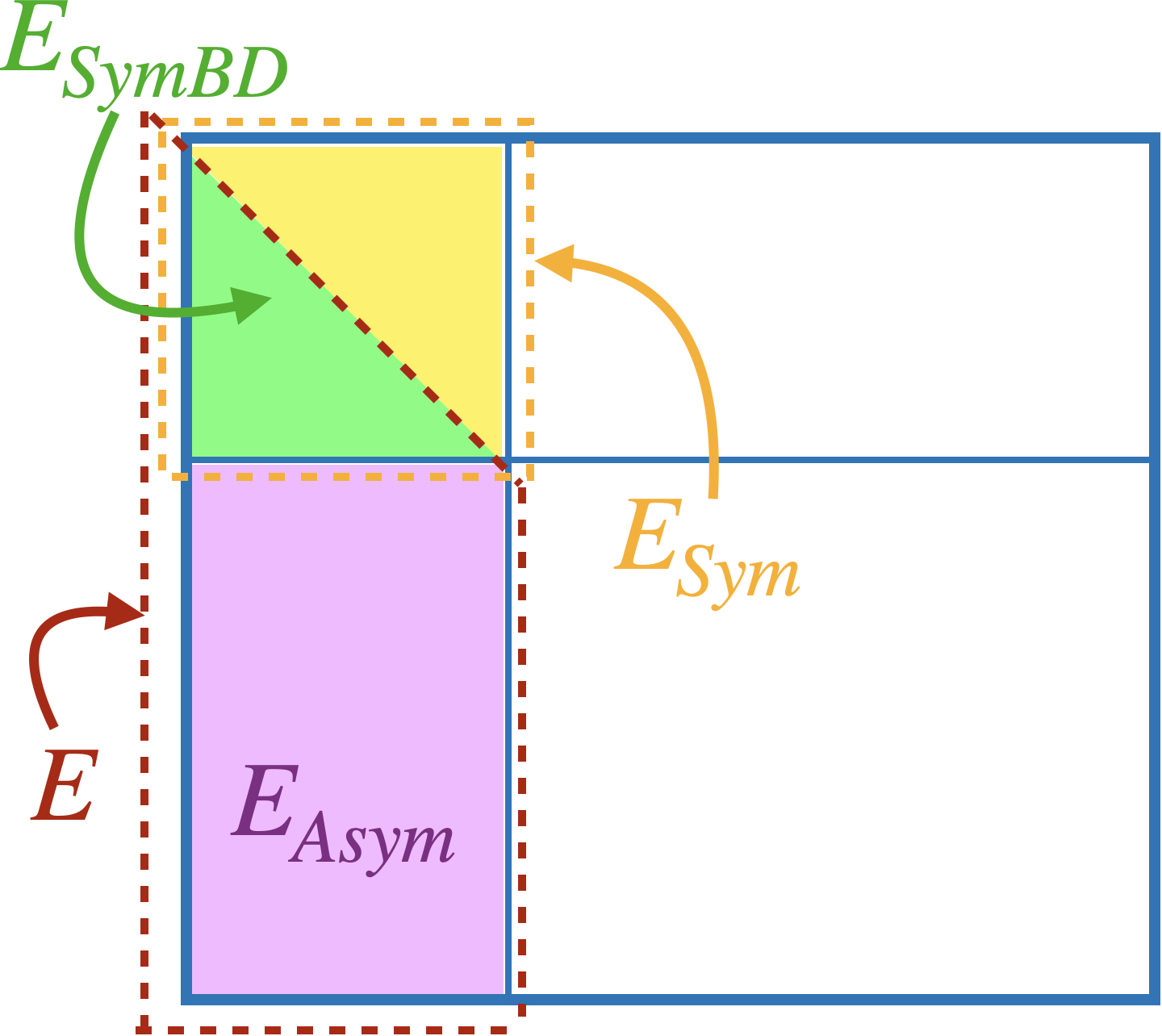} & \includegraphics[width=4cm]{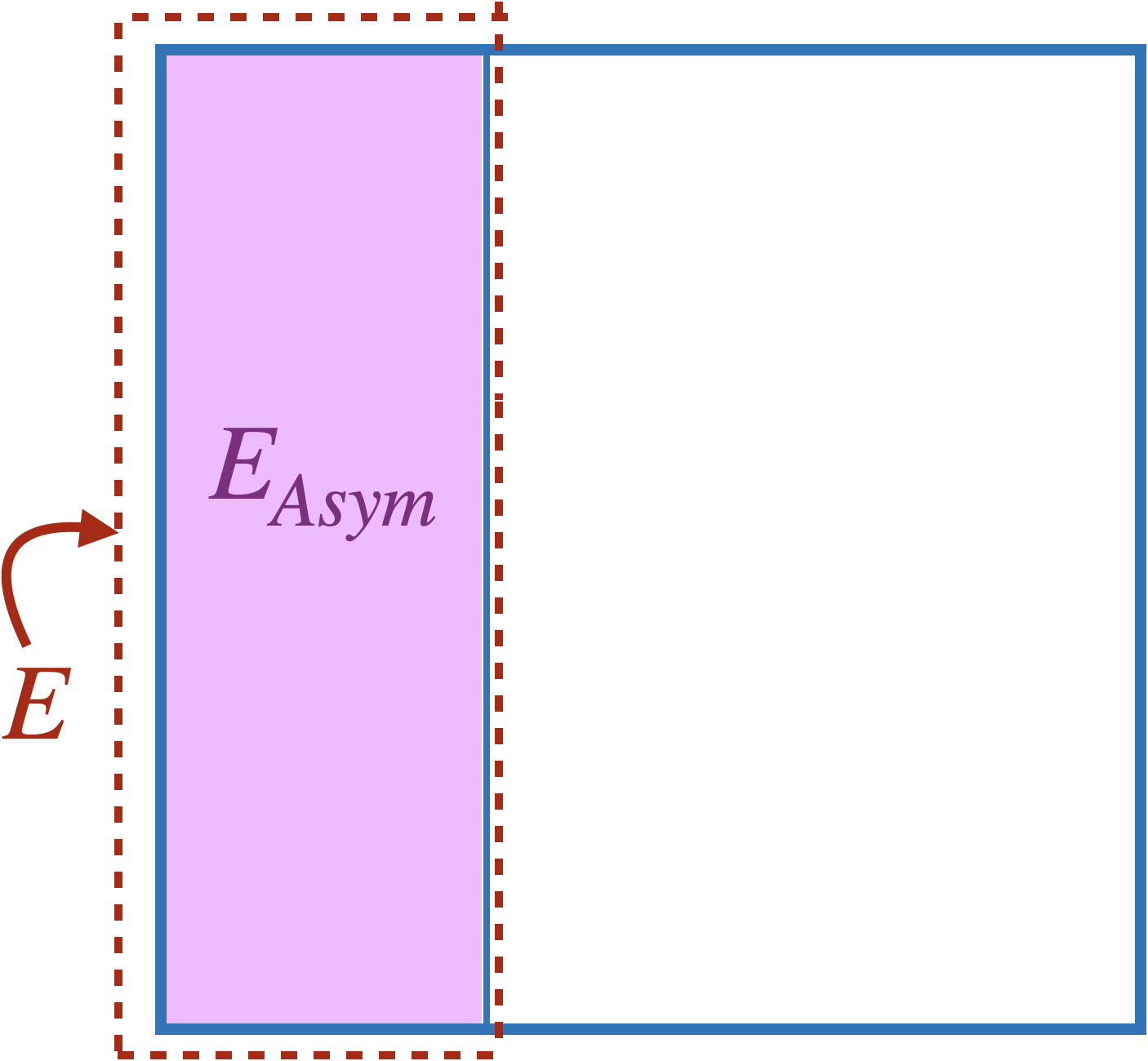}  \\
             a) $A$ & b)$A^{[n -1]}$ where $n \in S$ & c)$A^{[n -1]}$ where $n \notin S$ & d) $B$ 
        \end{tabular}
    \caption{Regions of $M$ for $M \in \{A, A^{[n -1]}, B\}$, with matrix rows permuted such that $S$ is at the top.}
    \label{fig:symmetric}
\end{figure}

We will couple the process of randomly exposing the $1$'s in $M_S$ with a random walk that counts the number of times a $1$ is exposed in $E_{Asym}$ in a row that already contains a $1$ or in $E_{SymBD}$. Observe that each consecutive $1$ is exposed in a random entry in $E$, sampled without replacement.
We condition on $L_S$, the number of $1$'s in $E$. Note that $L_S \sim \Bin(|E|, d/n)$. Let $(X_{i})_{i \in L_S}$ be the random walk that increases by one if the $i$th exposed $1$ is in $E_{SymBD}$ or if the $i$th exposed $1$ is in $E_{Asym}$ and is \em not \em the first $1$ exposed in its row. Otherwise, let $X_i = X_{i - 1}$. 

The following claims say when $X_{L_S}$ is small compared to the number of $1$'s in $M_S$, $M_S$ is not a minimal dependency.

\begin{claim}\label{claim:xrminimal}
If $X_{L_S} < \max\left(k + 1, \frac{L_S}{2}\right)$ and there are at least $2k + 1$ total $1$'s in $M_S$, then there is a row of $M_S$ (not in $S$) with a single $1$, so $M_S \notin \mathcal{M}_k$.
 \end{claim}
 \begin{proof}
 Let $P$ be the number of $1$'s in $E_{Asym}$, and let $Q$ be the number of $1$'s in $E_{SymBD}$ such that $Q + P = L_S$ and $2Q + P \geq 2k + 1$.
Let $Y:= X_{L_S} - Q$ be the number of $1$'s in $E_{Asym}$ which are not the first in their row.
Now the number of rows not in $S$ which have exactly one $1$ is at least $$ (P - Y) - Y = P - 2X_{L_S} + 2Q \geq \max(2k + 1, L_S) - 2X_{L_S} \geq \max(2k + 1, L_S) - \max (2k, L_S - 1) \geq  1.$$
 \end{proof}
 
\begin{claim}\label{claim:xrminimal4}
If $X_{L_S} < k + 1$, then $M_S \notin \mathcal{S}_{2k, k} \setminus \mathcal{T}_k^C$.
\end{claim}
\begin{proof}
We prove this by contradiction. Assume $M_S \in \mathcal{S}_{2k, k} \setminus \mathcal{T}_k^C$. Let $P$ be the number of $1$'s in $E_{Asym}$, and let $Q$ be the number of $1$'s in $E_{SymBD}$ such that $Q + P = L_S$ and $2Q + P = 2k$. Observe that if $M_S \in \mathcal{S}_{2k, k} \setminus \mathcal{T}_k^C$, then it must have exactly $k - 1$ non-zero rows. Indeed, if it had less than $k - 1$ non-zero rows, then $M_S$ would have rank at most $k - 2$, which contradicts $M_S$ being a minimal dependency. Further, if $M_S$ has more than $k-1$ non-zero rows, then by the fact that $M_S$ is a minimal dependency it must have $k$ non-zero rows (as each non-zero row must have at least two $1$'s and $M_S$ only has $2k$ $1$'s). However, if it has $k$ non-zero rows and is in $\mathcal{S}_{2k,k}$, then it must be in $\mathcal{S}_{2k,k}'$. Hence, by lemma $\ref{lemma:classification}$, $M_S\in \mathcal{T}_k^C$. We proceed in two cases:\\
\textbf{Case 1}: $Q = 0$. In this case, if each row in $M_S$ has at least two $1$'s, then we must have $$X_{L_S} = P - \text{number of occupied rows} = 2k - (k - 1) = k + 1 \geq  \max\left(k + 1, \frac{L_S}{2}\right),$$
since $L_S = P = 2k$.\\
\textbf{Case 2}: $Q > 0$. Note that in this case $M \in \{A, A^{[n - 1]}\}$. Let $Z$ be the number of occupied rows in $E_{Asym}$, such that $k - 1 - Z$ rows are occupied in $E_{SymBD}$. Observe that by the symmetry of $A$, number of occupied columns in $E_{SymBD}$ is at most $k - 1 - Z$.
Since we cannot have all-zero column (this would contradict $M_S$ being a minimal dependency), the number of columns occupied in $E_{Asym}$ that are not occupied in $E_{SymBD}$ must be at least $Z + 1$. However, by the connectivity of minimal dependencies (Lemma~\ref{lemma:connect}), the number of columns occupied in $E_{Asym}$ that are not occupied in $E_{SymBD}$ is at most $P - Z$. Note that here we also used the fact that $E_{SymBD}$ is non-empty. Thus we have $P - Z \geq Z + 1$. 
Summarizing, we have 
\begin{enumerate}
    \item $P - Z \geq Z + 1$
    \item $2k \leq P + 2Q$
    \item $Q + (P - Z) = X_{L_S} \leq k.$
\end{enumerate}
Combining (2) and (3), we achieve
$$Q + (P - Z) \leq \frac{P + 2Q}{2} \Rightarrow \frac{P - 2Z}{2} \leq 0,$$ which contradicts $(1)$.
\end{proof}

The proof of the next two claims are nearly identical to the proof of Claim~\ref{claim:xrminimal}, so we omit them.

\begin{claim}\label{claim:xrminimal2}
If $X_{L_S} < \max\left(k, \frac{L_S}{2}\right)$ and there are at least $2k - 1$ $1$'s total in $M_S$, then there is a row of $M_S$ (not in $S$) with a single $1$, so $M_S \notin \mathcal{M}_k$.
\end{claim}

\begin{claim}\label{claim:xrminimal3}
If $X_{L_S} < \max\left(k - 1, \frac{L_S}{2}\right)$ and there are at least $2k - 2$ $1$'s total in $M_S$, then there is a row of $M_S$ (not in $S$) with a single $1$, so $M_S \notin \mathcal{M}_k$.
\end{claim}

Recall from Observation~\ref{observation2} that any minimal dependency $M_S$ must have at least $2k - 2$ $1$'s total. By Lemma~\ref{lemma:classification}, any minimal dependency with $2k - 1$ $1$'s must be in $\mathcal{T}_k^+$. Further, any minimal dependency with $2k$ $1$'s must either be in $\mathcal{T}_k^C$ or in $\mathcal{S}_{2k, k} \setminus \mathcal{T}_k^C$.

Thus, any minimal dependency not in $(\mathcal{T}_k \cup \mathcal{T}_k^+ \cup \mathcal{T}_k^C)$ must have at least $2k + 1$ $1$'s or be in $(\mathcal{S}_{2k, k} \setminus \mathcal{T}_k^C)$. As a result, we have by Claim~\ref{claim:xrminimal} and Claim~\ref{claim:xrminimal4},
\begin{equation*}\Pr[M_S \in  \mathcal{M}_k \setminus (\mathcal{T}_k \cup \mathcal{T}_k^+ \cup \mathcal{T}^C)] \leq \Pr[X_{L_S} \geq \max\left(k + 1, L_S/2\right)]\end{equation*}
Note here the fact that if $M_S \in \mathcal{S}_{2k, k}$, then $\max(k + 1, L_S/2) = k + 1$.

Further, because minimal dependencies in $(\mathcal{T}_k^+ \cup \mathcal{T}_k^C)$ have at least $2k-1$ $1$'s, by Claim~\ref{claim:xrminimal2},
\begin{equation*}\Pr[M_S \in  (\mathcal{T}_k^+ \cup \mathcal{T}_k^C)] \leq \Pr[X_{L_S} \geq \max\left(k, L_S/2\right)]\end{equation*}

Finally by Claim~\ref{claim:xrminimal3},
\begin{equation*}
\Pr[M_S \in \mathcal{T}_k] \leq \Pr[X_{L_S} \geq \max\left(k - 1, L_S/2\right)].
\end{equation*}

To bound the probability that $X_{L_S}$ is large, we couple $X_i$ with a random walk $(Y_i)_{i \in L_S}$ which increases by one with probability $\frac{k(k + L_S)}{|E|}$ and otherwise stays constant. Observe that $Y_i$ stochastically dominates $X_i$, because there are at most $k + L_S$ rows --- and hence $k(k + L_S)$ locations in $E$ --- in which exposing a $1$  will increase $X_i$. Then conditioned on $L_S = \ell$, for any $j$, we have
\begin{equation*}
\begin{split}
\Pr\left[X_{\ell} \geq j\right] &\leq \Pr\left[\Bin\left(\ell, \frac{k + \ell}{n - k}\right) \geq j\right].
\end{split}
\end{equation*}


We now use Lemma~\ref{masterlemma:small}, proved in Appendix~\ref{sec:bounds}, to sum this probability over all values of $L_S$.

\begin{restatable}[Small Case Binomial Calculation]{lemma}{mastersmall}\label{masterlemma:small}
For constants $\gamma, d > 0$, for $k \leq \frac{n}{8e^4\gamma d^2}$, there exists a constant $c_{\ref{masterlemma:small}}$ such that for any $j \in \{k - 1, k, k + 1\}$ and $\gamma \geq 1/2$, we have
\begin{equation*}
    \sum_{\ell \geq 1}^{\infty}\Pr\left[\Bin\left(\ell, \frac{\ell + k}{\gamma n}\right) \geq \max\left(j, \frac{\ell}{2}\right)\right]\Pr[\Bin(\gamma nk, d/n) = \ell] \leq e^{-\gamma dk + c_{\ref{masterlemma:small}}k\log(\gamma d)}\left(\frac{k}{n}\right)^j.
\end{equation*}
\end{restatable}

We employ this lemma with $\gamma = \frac{|E|}{nk} \geq 1 - k/n$ and $j = \max\left(k+1, L_S/2\right)$, to achieve
\begin{equation*}\Pr[M_S \in \mathcal{M}_k \setminus (\mathcal{T} \cup \mathcal{T}_k^+ \cup \mathcal{T}^C)] \leq e^{-dk + (c_{\ref{masterlemma:small}}+1)k\log(d)}\left(\frac{k}{n}\right)^{k+1}.\end{equation*}

Further, letting $j = \max\left(k, L_S/2\right)$ gives:
\begin{equation}\label{smallcase:both}\Pr[M_S \in (\mathcal{T}_k^+ \cup \mathcal{T}_k^C)] \leq e^{-dk + (c_{\ref{masterlemma:small}}+1)k\log(d)}\left(\frac{k}{n}\right)^k.\end{equation}

Finally, setting $j = \max\left(k - 1, L_S/2\right)$, gives: 
\begin{equation*}\Pr[M_S \in \mathcal{T}_k] \leq e^{-dk + (c_{\ref{masterlemma:small}}+1)k\log(d)}\left(\frac{k}{n}\right)^{k - 1}.
\end{equation*}
Putting $c_{\ref{prop:small}} = c_{\ref{masterlemma:small}} + 1$ proves the lemma. Note that eq.~\ref{smallcase:both} establishes the second and third statements of the lemma simultaneously.
\end{proof}

\subsection{Small Case 2}


Our main goal of this section is to prove the following lemma.

\begin{lemma}[Small Case 2]\label{lemma:medium}
Let $M$ be any of the following random matrices for $d_0 \leq d \leq 3\log(n)$ for some universal constant $d_0$:
\begin{enumerate}
    \item $M = A$, where $A \sim \mathbb{A}(n, \frac{d}{n})$
    \item $M = A^{[n - 1]}$, where $A \sim \mathbb{A}(n, \frac{d}{n})$
    \item $M = B$, where $B \sim \mathbb{B}(n, \frac{d}{n})$
\end{enumerate}
Then
\begin{equation*}
\Pr\left[\exists x: Mx = 0, \frac{n}{8e^4d^2} \leq \supp(x) < \frac{9n}{d}\right] = o(1).
\end{equation*}
\end{lemma}

\begin{proof}
We will take a union bound over all possible sets $S \subset [n]$ of size $k \in \left[\frac{n}{8e^4d^2}, \frac{9n}{d}\right]$ of the probability that $M$ has a kernel vector with support $S$. By Observation~\ref{observation1}, it suffices to show that with probability $1 - o(1)$, for all such sets $S$, there is a row in the submatrix $M_S$ with exactly one $1$.

For each of the three choices of ensemble, the rows $(M_S)^{\{i\}}$ for $i \in [n] \setminus S \setminus \{n\}$ are mutually independent, so we have 
\begin{equation*}
\begin{split}
    p_k &:= \Pr[\exists x \in \mathbb{R}^{n}, \supp(x) = S, M x = 0] \leq \left(1 - k\frac{d}{n}\left(1 - \frac{d}{n}\right)^{k - 1}\right)^{n - k - 1}\\
    &\leq \left(1 - k\frac{d}{n} e^{-2dk/n} \right)^{n - k - 1} \leq e^{- k(n-k-1)\frac{d}{n} e^{-2dk/n}} \leq e^{- \frac{kd}{2} (e^{-2dk/n})}.
\end{split}
\end{equation*}

Note that $p_k$ does not depend on the choice of $S$, just its size $k$. Taking a union bound over all $k$ and all sets $S$ of size $k$, we achieve:

\begin{align*}
\sum_{k = \frac{n}{8e^4d^2}}^{\frac{9n}{d}} \sum_{S \subset [N], |S| = k} p_k &\leq \sum_{k=\frac{n}{8e^4d^2}}^{\frac{9n}{d}} \left(\frac{en}{k}\right)^ke^{- \frac{kd}{2} (e^{-2dk/n})} 
\leq \sum_{k=\frac{n}{8e^4d^2}}^{\frac{9n}{d}} e^{-\frac{dk}{2}(e^{-2dk/n}) + k\ln(en/k)} \\
&\leq n \max_{k_1,k_2,k_3 \in [\frac{n}{8e^4d^2}, \frac{9n}{d}]} \left(e^{-\frac{d}{2}(e^{-2dk_1/n}) + \ln(en/k_2)}\right)^{k_3} 
= n  \left(e^{-\frac{d}{2}(e^{-18}) + \ln(8e^5 d^2)}\right)^{\frac{n}{8e^4d^2}}\\
&= o(1)
\end{align*}
for sufficiently large $d$. 
\end{proof}

\section{The Large Domain}\label{sec:large}

\subsection{Statement of Lemmas}
In the large domain, we bound the probability of having minimal dependencies of $k$ columns for $k \geq \frac{9n}{d}$. Large Case 1 (Section~\ref{sec:large1}) considers the case when $\frac{9n}{d} \leq k < \frac{n}{C}$ for some constant $C$. Large Case 2 (Section~\ref{sec:large2}) considers the case where $\frac{n}{C} \leq k \leq n$. As discussed in the technical overview, our main tool in these two cases are the Littlewood-Offord anti-concentration results.

\subsection{Large Case 1}\label{sec:large1}
We prove the following lemma.

\begin{lemma}[Large Case 1]\label{lemma:medium2}
There exist universal constants $C$ and $d_0$, such that the following holds. Let $M$ be any of the following random matrices for $d_0 \leq d \leq 3\log(n)$:
\begin{enumerate}
    \item $M = A$, where $A \sim \mathbb{A}(n, \frac{d}{n})$
    \item $M = A^{[n-1]}$, where $A \sim \mathbb{A}(n, \frac{d}{n})$
    \item $M = B$, where $B \sim \mathbb{B}(n, \frac{d}{n})$
\end{enumerate}
Then
\begin{equation*}
\Pr\left[\exists x: Mx = 0, \frac{9n}{d} \leq \supp(x) < \frac{n}{C}\right] = o(1).
\end{equation*}
\end{lemma}
\begin{proof}
We union bound over all $k \in [9n/d, n/C]$ and all sets $S$ of size $k$ of the probability that there exists $x \in \mathbb{R}^n$ with $\supp(x) = S$ and $M x = 0$.

Fix a set $S$ of size $k$. We will consider the random process where we expose the $n - k - 1$ rows $M_S^{\{j\}}$ for $j \in [n - 1] \setminus S$ one at a time. Note that these rows are all mutually independent since they do not include rows indexed by $S$. Further, they do not include $n$th row (which does not exist in the case that $M = A^{[n-1]}$). For $i \leq n - k - 1$, let $\mathcal{N}_i \subset \mathbb{R}^k$ be the kernel of the first $i$ rows exposed, and let $\mathcal{D}_i \subseteq \mathcal{N}_i$ be the span of the set of vectors in $\mathcal{N}_i$ which have no zeros. Let $R_i$ be the dimension of $\mathcal{D}_i$.

If $R_i > 0$, then we can choose an arbitrary vector $v$ in $\mathcal{D}_i$ with support $k$, and by Lemma~\ref{lo_sparse}, with probability at least $1 - \frac{1}{\sqrt{kd/n}}$, the $(i + 1)$th row exposed is not orthogonal to $v$. In this case $R_{i + 1} = R_i - 1$.

If at any point $R_i$ becomes $0$, then this means there can be no dependency involving all the columns of $M_S$. It follows that since $R_0 = k$, we have 
\begin{equation}\label{eq:unionboundthis}
\Pr[R_{n - k - 1} \neq 0] \leq \Pr\left[\Bin\left(n - k - 1, 1 - \frac{1}{\sqrt{kd/n}}\right) < k\right].
\end{equation}

We employ the following computational lemma, proved in Appendix \ref{sec:bounds}:

\begin{restatable}[Large Case 1 Binomial Calculation]{lemma}{largegeneral}\label{large_general}
There exists constants $C$, $d_0$, and $c_{\ref{large_general}}$, such that for all $d \geq d_0$, we have 
\begin{equation*}
    \sum_{k = \frac{2n}{d}}^{n/C}\binom{n}{k}\Pr\left[\Bin\left(n - k - 1, 1 - \frac{1}{\sqrt{kd/n}}\right) \leq k\right] \leq e^{-c_{\ref{large_general}}n}
\end{equation*}
\end{restatable}

Consequently, we can union bound eq.~\ref{eq:unionboundthis} over all $k$ in the desired range and all $\binom{n}{k}$ sets $S$ of size $k$:
\begin{equation*}
    \Pr[\exists x: M^\top x = 0, 9n/d \leq \supp(x) \leq n/C] \leq e^{-c_{\ref{large_general}}n}.
\end{equation*}
This concludes the proof of Lemma \ref{lemma:medium2}.
\end{proof}

\subsection{Large Case 2}\label{sec:large2}
As outlined in the technical overview, the proof strategy for Large Case 2 for the symmetric matrix $A \sim \mathbb{A}(n, p)$ varies significantly from that of the asymmetric matrix $B \sim \mathbb{B}(n, p)$. Since the proof for $B$ is simpler, we begin with that matrix.
\subsubsection{Large Case 2 for $B$}
For the asymmetric matrix $B$, it is not possible to rule out both left and right kernel vectors of $B$ with large support. Indeed, often such dense kernel vectors do exist. We can, however, rule out having dense right and left kernel vectors \em simultaneously\em. We do so in the following lemma.
\begin{lemma}[Large Case 2 for $B$]\label{lemma:large2_B} 
Let $B \sim \mathbb{B}(n, \frac{d}{n})$. Let $\mathcal{S}_t$ be the set of vectors in $\mathbb{R}^n$ with support size at least $t$. For $t \geq \frac{n}{C}$, where $C$ is the constant from Lemma~\ref{lemma:medium2}, for $d_0 \leq d \leq 3\log(n)$ for some constant $d_0$,
\begin{equation*}
\Pr\left[\exists x, y \in \mathcal{S}_t: Bx = 0 \text{ and } y^\top B = 0 \right] \leq \left(\frac{n}{t}\right)^{3/2}\frac{1}{\sqrt{d}}.
\end{equation*}

\end{lemma}
\begin{proof}
For $i \in [n]$, let $H_i$ be the space spanned by the column vectors $B_1, B_2, \cdots B_{i - 1}, B_{i + 1}, \cdots B_n$. Following Observation~\ref{obs_3} from the technical overview, $Bx = 0$ and $y^\top B = 0$ imply that for all $i \in \text{supp}(x)$,
\begin{enumerate}
        \item $B_i \in H_i$;
    \item $\exists w \in \mathcal{S}_t: w^\top B_j = 0 \:\: \forall j \in [n] \setminus i$.
\end{enumerate}

Let $\mathcal{F}_i$ denote the event that these two events occur for index $i$. Let $X_i$ be the indicator of this event, and let $X = \sum_i X_i$. It follows that
\begin{equation*}
\Pr\left[\exists x, y \in \mathcal{S}_t: Bx = 0 \text{ and } y^\top B = 0 \right] \leq \Pr[X \geq t].
\end{equation*}

By Markov's inequality,
\begin{equation}\label{markov_as}
    \Pr\left[X \geq t\right] \leq \frac{\mathbb{E}[X]}{t} = \frac{n\Pr[\mathcal{F}_n]}{t}.
\end{equation}

Now by conditioning on $H_i$ and using the sparse Littlewood-Offord Lemma (Lemma~\ref{lo_sparse}), we have \begin{equation*}
\begin{split}
\Pr[\mathcal{F}_n] &= \Pr\left[\exists w \in \mathcal{S}_t: w^\top B_j = 0 \:\: \forall j \in [n - 1] \right]\Pr\left[B_n \in H_n  | \exists w \in \mathcal{S}_t: w^\top B_j = 0 \:\: \forall j \in [n - 1] \right]\\
&\leq \Pr\left[\exists w \in \mathcal{S}_t: w^\top B_j = 0 \:\: \forall j \in [n - 1] \right]\left(\max_{w \in \mathcal{S}_t}\Pr[B_n^\top w = 0]\right)\\
&\leq \max_{x \in \mathcal{S}_t}\Pr[B_n^\top w = 0]
\leq \frac{1}{\sqrt{td/n}}.
\end{split}
\end{equation*}
Here  the first inequality follows from the fact that $B_n \in H_n$ implies that $B_n$ is orthogonal to any normal of $H_n$. We are able to remove the conditioning in this step because $B_n$ is independent from $H_n$ (and hence also the normal $w$). The final inequality follows from Lemma~\ref{lo_sparse}, where we used the fact that $t \geq \frac{n}{C} \geq \frac{9n}{d}$ for $d$ large enough.

Plugging $\Pr[\mathcal{F}_n]$ in to eq.~\ref{markov_as}, we have $\Pr[X \geq t] \leq \left(\frac{n}{t}\right)^{3/2}\frac{1}{\sqrt{d}},$
which proves the lemma.
\end{proof}

\subsubsection{Large Case 2 for $A$}

For the symmetric matrix $A \sim \mathbb{A}(n, p)$, Large Case 2 reduces the probability of having a kernel vector with large support to the probability of having several smaller structures in $A$. The following is our main technical lemma for $A$:
\begin{lemma}[Technical Lemma for Large Case 2 for $A$]\label{lemma:large2_A}
Let $A \sim \mathbb{A}(n, d/n)$. Let $A^{(i)}$ denote the submatrix of $A$ obtained by removing the $i$th row and column. Let $\textbf{a} \in \mathbb{R}^{n - 1}$ be a random vector with i.i.d. $\Ber(d/n)$ entries. For any $u , r, s, t \in [n]$ with $r < s$,
\begin{equation}\label{eq:lem}
\begin{split}
\Pr\left[\exists x: Ax = 0, \supp(x) > t\right] &\leq 
\frac{n}{t}\max_{x \in \mathbb{R}^{n - 1}: \supp(x) \geq s}\Pr\left[x^\top \textbf{a} = 0\right]\\
&\quad +\frac{n}{t}\Pr\left[\exists x: A^{(n)}x = 0, r < |\supp(x)| < s \right]\\
&\quad + \frac{n}{t} \frac{dr}{n}\\
&\quad+ \frac{n}{t}\frac{n}{u}\Pr\left[\exists x \neq 0: A^{(n)}x = e_1, |\supp(x)| < s\right]\\
&\quad+ \frac{n}{t}\max_{X \in \mathcal{T}^{n - 1}_{n - 1 - r - u, s}}\Pr\left[\textbf{a}^\top X\textbf{a} = 0\right],
\end{split}
\end{equation}
where $\mathcal{T}^{m}_{\alpha, \beta}$ denotes the set of matrices in $\mathbb{R}^{m \times m}$ with some set of $\alpha$ columns that each have support size at least $\beta$.

\end{lemma}
Lemma~\ref{lemma:large2_A} breaks down the probability that there is a large linear dependency consisting of more than $t$ rows into the sum of the probabilities of several other events. Two of these probabilities (lines 1 and 5 of the right hand side of eq.~\ref{eq:lem} in the lemma) we can show to be small via anti-concentration lemmas. Two of the probabilities (lines 2 and 4 of the right hand side of eq.~\ref{eq:lem} in the lemma), we will eventually show to be small using the lemmas from Small Case 1, 2, and Large Case 1.

Instantiating Lemma~\ref{lemma:large2_A} with $t = \frac{n}{C}$, $s = \frac{n}{C}$, $r = \frac{n}{8e^4d^2}$, $u = \frac{n}{2}$, where $C$ is the constant from Lemma~\ref{lemma:medium2}, we obtain the following lemma.

\begin{lemma}[Large Case 2 for $A$]\label{lemma:large_plug}
Let $A \sim \mathbb{A}(n, \frac{d}{n})$ for $ \omega(1) = d \leq 3\log(n)$.
With $C$ equal to the constant from Lemma~\ref{lemma:medium2},
\begin{equation*}
\begin{split}
\Pr[\exists x: Ax = 0, \supp(x) > n/C] &\leq C\Pr[\exists x: A^{(n)}x = 0, \frac{n}{8e^4d^2} \leq |\supp(x)| \leq n/C] \\
&\qquad + 2C\Pr[\exists x: A^{(n)}x = e_1, |\supp(x)| \leq n/C] + o(1).
\end{split}
\end{equation*}
\end{lemma}
\begin{proof}
This follows immediately from plugging in these values of $t, s, r$ and $u$ into Lemma~\ref{lemma:large2_A} and applying the anti-concentration results in Lemmas~\ref{lo_sparse} and \ref{lo_quadratic} to the first and last terms in eq.~\ref{eq:lem}.
Indeed, Lemma~\ref{lo_sparse} shows that \begin{equation*}
    \frac{n}{t}\max_{x \in \mathbb{R}^{n - 1}: \supp(x) \geq s}\Pr\left[x^\top \textbf{a} = 0\right] \leq \frac{n}{t}\frac{1}{\sqrt{sd/n}} = o(1),
\end{equation*}
where we have used the fact that $sd/n = d/C \geq 9$ since $d = \omega(1)$.
Lemma~\ref{lo_quadratic} shows that
\begin{equation*}
    \frac{n}{t}\max_{X \in \mathcal{T}^{n - 1}_{n - 1 - r - u, s}}\Pr\left[\textbf{a}^\top X\textbf{a}  = 0\right] \leq O\left(\frac{1}{\sqrt[4]{\min(s, n - 1 - r - u)d/n}}\right) = o(1).
\end{equation*} The third term in eq.~\ref{eq:lem} equals $\frac{C}{8e^4d}$ which is also $o(1)$.
\end{proof}

To prove Lemma~\ref{lemma:large2_A}, we will use the following two linear-algebraic lemmas, proved in Appendix~\ref{apx:matrices}. 
\begin{restatable}{lemma}{nullspace}\label{claim:nullspace}
Let $A$ be a matrix with columns $A_i$ for $i \in [n]$. Let $H_i$ be the space spanned by the column vectors $A_1, A_2, \cdots A_{i - 1}, A_{i + 1}, \cdots A_n$. Let $S$ be the set of all $i$ such that $A_i \in H_i$. Then there exists some $y$ with $\text{supp}(y) = S$ such that $Ay = 0$.
\end{restatable}

\begin{restatable}{lemma}{basis}\label{claim:basis}
With the terminology of the previous lemma, $e_i \in \Span(A^\top )$, if and only if $A_i \notin H_i$.
\end{restatable}

\begin{proof}[Proof of Lemma~\ref{lemma:large2_A}]

For each $i \in [n]$, we define $H_i$ to be the space spanned by the column vectors $A_1, A_2, \cdots A_{i - 1}, A_{i + 1}, \cdots A_n$.

Then, as per Observation~\ref{obs_3}, $Ax = 0$ for some $x$ implies that for all $i \in \text{supp}(x)$, $A_i \in H_i$. Let $\mathcal{E}_i$ denote the event that $A_i \in H_i$. Let $X_i$ be the indicator of this event, and let $X = \sum_i X_i$. By Markov's inequality and the exchangeability of the columns,
\begin{equation}\label{markov}
     \Pr\left[\exists x : Ax = 0, |\supp(x)| \geq t \right] = \Pr\left[X \geq t\right] \leq \frac{\mathbb{E}[X]}{t} = \frac{n\Pr[\mathcal{E}_n]}{t}.
\end{equation}

We will break down the probability $\Pr[\mathcal{E}_n]$ into several cases, depending on the size of the support of vectors in the kernel of $A^{(n)}$. Let $S \subseteq [n - 1]$ be the set of all $i$ such that $e_i \in \Span(A^{(n)})$, such that by Lemma~\ref{claim:nullspace} and  Lemma~\ref{claim:basis},
\begin{equation*}
k:= \max(\supp(x): A^{(n)}x = 0) = n - 1 - |S|. 
\end{equation*}

\textbf{Case 1:} $A^{(n)}$ has a kernel vector $x$ with large support, that is, $k \geq s$. 

\textbf{Case 2:} $A^{(n)}$ has a kernel vector $x$ with medium support, that is $r < k < s$. 

\textbf{Case 3:} $A^{(n)}$ does not have any kernel vectors with large or medium support vectors in its kernel, that is, $k \leq r$. 

We can expand
\begin{equation}\label{cases}
\Pr[\mathcal{E}_n] = \Pr[\mathcal{E}_n| k \geq s]\Pr[k \geq s] + \Pr[\mathcal{E}_n|r < k < s]\Pr[r < k < s] + \Pr[\mathcal{E}_n| k \leq r]\Pr[k \leq r]. 
\end{equation}
For simplicity, define $\textbf{a} := A_n^{/n}$ to be the first $n - 1$ entries of the column $A_n$.

To evaluate the probability of the first case, we condition on $A^{(n)}$ and let $x$ be any vector of support at least $s$ in the kernel of $A^{(n)}$. Observe that $\mathcal{E}_n$ cannot hold if $x^\top\textbf{a}$ is non-zero. Indeed, if $x^\top \textbf{a} \neq 0$, then let $x' = (x_1, x_2,\ldots,x_{n-1},0)/(x^\top\textbf{a})$ such that $Ax' = e_n$. Then by Lemma~\ref{claim:basis}, $A_n \notin H_n$ and hence $\mathcal{E}_n$ does not occur. Since $\textbf{a}$ is independent from $x$, we have 

\begin{equation}\label{case_1}
\Pr[\mathcal{E}_n| k \geq s]\Pr[k \geq s] \leq \max_{x: \supp(x) \geq s}\Pr[x^\top\textbf{a} = 0].
\end{equation}

Combined with eq.~\ref{markov}, the contribution from this case yields the first term in the right hand side of eq.~\ref{eq:lem}.

For the second case, we bound:
\begin{equation}\label{case_2}
\Pr[\mathcal{E}_n|r < k <  s]\Pr[r < k < s] \leq \Pr[r < k < s] \leq \Pr\left[\exists x: A^{(n)}x = 0, r < |\supp(x)| < s \right].    
\end{equation}

Combined with eq.~\ref{markov}, the contribution from this case yields the second term in the right hand side of eq.~\ref{eq:lem}.

The third case will lead to the final three terms in the right hand side of eq.~\ref{eq:lem}. In this case, we will show conditions under which we can algebraically construct a vector $v$ such that $Av = e_n$. This will imply by Lemma~\ref{claim:basis} that $A_n \notin H_n$.

Recall that $S \subseteq [n - 1]$ is the set of all $i$ such that $e_i \in \Span(A^{(n)})$. For $i \in S$, let $w_i $ be any vector such that $A^{(n)}{w_i } = e_i$. We next construct a sort of ``pseudoinverse" matrix $B \in \mathbb{R}^{n - 1 \times n - 1}$ as follows: For $i \in S$, define $B_{ij}$ to be the $i$th entry of $w_i $. That is, for $i \in S$, the $i$th column of $B$ is $w_i $. Define all other entries of $B$ to be zero.

The following claim shows a condition for $\mathcal{E}_n$ not holding.
\begin{claim}\label{claim:condition}
If $\supp(\textbf{a}) \subseteq S$ and $\textbf{a}^\top B\textbf{a} \neq 0$, then $e_n \in \Span(A)$.
\end{claim}
\begin{proof}
Let ${w'} := B\textbf{a} = \sum_{i \in S}{\textbf{a}_iw_i }$ such that $A^{(n)}{w'} = \sum_{i \in S}{\textbf{a}_ie_i}$. Hence if $\supp(\textbf{a}) \subseteq S$, $A^{(n)}{w'} = \textbf{a}$. Define $w \in \mathbb{R}^n$ to be the vector with $w'$ in the first $n - 1$ entries and $-1$ in the final entry. Then the first $n - 1$ entries of $Aw$ are $0$, and the last entry is ${\bf{a}}^\top  w' = {\bf{a}}^\top  B \bf{a}$. Evidently, if ${\bf{a}}^\top  B {\bf{a}} \neq 0$, then $$\frac{Aw}{{\bf{a}}^\top  B {\bf{a}}} = e_n,$$ so $e_n \in \Span(A)$.
\end{proof} 
By definition, in the third case, we have $|S| \geq n - 1 - r$. Hence by Claim~\ref{claim:condition},
\begin{equation}\label{case_3}
\begin{split}
\Pr[\mathcal{E}_n \land k \leq r]
&\leq \Pr\left[\supp(\textbf{a}) \not\subseteq S \land |S| \geq n - 1 - r\right] + \Pr[\textbf{a}^\top B\textbf{a} = 0 \land |S| \geq n - 1 - r].
\end{split}
\end{equation}

Notice that $S$ is a function of $A^{(n)}$ and so $\textbf{a}$ is independent from $S$. It is easy to check that for any set $S$ of size at least $n - 1 - r$, 
\begin{equation}\label{case_3_a_S}
\Pr\left[\supp(\textbf{a}) \not\subseteq S\right] \leq 1 - \left(1 - \frac{d}{n}\right)^r \leq \frac{dr}{n}. 
\end{equation}

We will break up the second term in eq.~\ref{case_3} by conditioning on whether the support of $B$ has many entries or not, and using the independence of $\textbf{a}$ from $B$:

\begin{equation}\label{eq:T_term}
\Pr[\textbf{a}^\top B\textbf{a} = 0] \leq \Pr\left[B \notin \mathcal{T}^{n - 1}_{n - 1 - r - u, s}  \wedge |S| \geq n - 1 - r\right] + \max_{X \in \mathcal{T}^{n - 1}_{n - 1 - r - u, s}}\Pr\left[\textbf{a}^\top X\textbf{a} = 0\right]
\end{equation}

To further bound the first probability on the right hand side, observe that if $|S| \geq n - 1 - r$ and $B \notin \mathcal{T}^{n - 1}_{n - 1 - r - u, s}$, there must exist at least $u$ different values $i \in S$ such that $\supp(w_i ) \leq s$. So 
\begin{equation*}
\begin{split}
\Pr\left[B \notin \mathcal{T}^{n - 1}_{n - 1 - r - u, s}  \wedge |S| \geq n - 1 - r\right] &\leq \Pr\left[|\{i: \exists x \neq 0: A^{(n)}x = e_i, |\supp(x)| < s\}| \geq u \right] \\ &\leq \frac{n}{u}\Pr\left[\exists x \neq 0: A^{(n)}x = e_1, |\supp(x)| < s\right],
\end{split}
\end{equation*}
where the last inequality follows by Markov's inequality. Combining this with eqs.~\ref{eq:T_term}, \ref{case_3_a_S}, and \ref{case_3} yields
\begin{equation*}
\Pr[\mathcal{E}_n \land k \leq r] \leq \frac{dr}{n} + \max_{X \in \mathcal{T}^{n - 1}_{n - 1 - r - u, s}}\Pr\left[\textbf{a}^\top X\textbf{a} = 0\right] + \frac{n}{u}\Pr\left[\exists x \neq 0: A^{(n)}x = e_1, |\supp(x)| < s\right].
\end{equation*}

Plugging this and eqs.~\ref{case_1} and \ref{case_2} into eq.~\ref{cases} and finally eq.~\ref{markov} yields the lemma.
\end{proof}

\section{Proof of Characterization of Linear Dependencies: Theorem~\ref{char:square} and Theorem~\ref{char:bgc}}\label{sec:char}

We are now ready to put the results of the small and large domains together to prove Theorem~\ref{char:square} and Theorem~\ref{char:bgc}.

\subsection{Proof of Theorem~\ref{char:square}}
For the convenience of the reader, we restate this theorem:

\charsquare*

We will need the following two lemmas to show that the first two terms in the right hand size of Lemma~\ref{lemma:large_plug} are $o(1)$.

\begin{lemma}\label{lem:medium}
Let $A \sim \mathbb{A}(n, \frac{d}{n})$ with $ \omega(1) = d \leq 3\log(n)$ and let $C$ be as in Lemma~\ref{lemma:medium2}. Then
\begin{equation*}
\Pr\left[\exists x: Ax = 0, \frac{n}{8e^4d^2} < |\supp(x)| < \frac{n}{C} \right] = o(1),
\end{equation*}
and 
\begin{equation*}
\Pr\left[\exists x: A^{[n - 1]}x = 0, \frac{n}{8e^4d^2} < |\supp(x)| < \frac{n}{C} \right] = o(1).
\end{equation*}
\end{lemma}

\begin{proof}
This is immediate from lemma from Small Case 2, Lemmas~\ref{lemma:medium}, and from Large Case 1, Lemma~\ref{lemma:medium2}.
\end{proof}

\begin{lemma}\label{lem:basis}
Let $A \sim \mathbb{A}(n, \frac{d}{n})$ with $ \omega(1) = d \leq 3\log(n)$ and let $C$ be as in Lemma~\ref{lemma:medium2}. Then
\begin{equation*}
\Pr\left[\exists x: Ax = e_n, |\supp(x)| < \frac{n}{C} \right] = o(1).
\end{equation*}
\end{lemma}

We reduce Lemma~\ref{lem:basis} to Lemma~\ref{lem:small}, which we are better suited to prove with our lemmas from Small Case 1, 2, and Large Case 1.

\begin{lemma}\label{lem:small}
Let $A \sim \mathbb{A}(n, \frac{d}{n})$ with $ \omega(1) = d \leq 3\log(n)$ and let $C$ be as in Lemma~\ref{lemma:medium2}. Let $K = \{x: Ax = 0, |\supp(x)| \leq \frac{n}{C}\}$. Then
\begin{equation*}
\Pr\left[\left|\bigcup_{x \in K}{\supp(x)}\right| \geq \frac{n}{8e^4d^2} \right] = o(1).
\end{equation*}
Similarly, suppose $A' = A^{[n-1]}$ is $A$ with the last row removed. Let $K' = \{x: A'x = 0, |\supp(x)| \leq \frac{n}{C}\}$. Then
\begin{equation*}
\Pr\left[n \in \bigcup_{x \in K'}{\supp(x)} \right] = o(1).
\end{equation*}
\end{lemma}
 We first prove Lemma~\ref{lem:basis} from Lemma~\ref{lem:small}:
\begin{proof}[Proof of Lemma~\ref{lem:basis}]
First we consider vectors $x$ with $n \notin \supp(x)$.
Applying the first part of Lemma~\ref{lem:small} to $A^{(n)}$, with probability $1 - o(1)$, there exists some set $T \subseteq [n] \setminus \{n\}$ with $|T| \leq \frac{n}{8e^4d^2}$ such that $\supp(x) \subseteq T$ for all $x$ satisfying that $A^{(n)}x = 0$
and $\supp(x) \leq \frac{n}{C}$. With probability $1 - o(1)$, $\supp(A_n) \cap T = \emptyset$, so for all vectors $x$ with support in $[n - 1]$ and of size less than $\frac{n}{C}$, we do not have $Ax = e_n$. 

Next we consider vectors $x$ with $n \in \supp(x)$. If such an $x$ exists, that is, $Ax = e_n$ and $n \in \supp(x)$, then it must be the case that $A'x = 0$. By the second part of Lemma~\ref{lem:small}, the probability that such an $x$ exists is $o(1)$.
\end{proof}

\begin{proof}[Proof of Lemma~\ref{lem:small}]
Let $\mathcal{L}_2$ be the event that $$\exists x: Ax = 0, \frac{n}{8e^4d^2} < |\supp(x)| < \frac{n}{C}.$$ Recall that by Lemma~\ref{lem:medium}, this event occurs with probability $o(1)$.
To prove the first part, we will show that conditioned on $\mathcal{L}_2$ not occurring, we have $\left|\bigcup_{x \in K}{\supp(x)}\right| < \frac{n}{8e^4d^2}$.

Index the set $K$ as follows: $K = \{x^{(1)}, \cdots, x^{(|K|)}\}$. For $i = 1$ to $|K|$, let $y^{(i)} = \sum_{j = 1}^i{c_i x^{(i)}}$, where $c_i$ is chosen uniformly from the interval $[0, 1]$. It follows that with probability $1$, for all $i \leq |K|$, $$\supp(y^{(i)}) = \bigcup_{j \leq i}{\supp(x^{(i)})}.$$

If $\mathcal{L}_2$ does not occur, then $|\supp(x^{(i)})| \leq \frac{n}{8e^4d^2}$ for all $i$, so we have that $|\supp(y^{(i + 1)})| \leq |\supp(y^{(i)})| + \frac{n}{8e^4d^2}$. It follows that if $|\supp(y^{(|K|)})| \geq \frac{n}{8e^4d^2}$, then there exists some $i \leq |K|$ such that $|\supp(y^{(i)})| \in [\frac{n}{8e^4d^2}, \frac{2n}{8e^4d^2}]$. However, this would imply that $\mathcal{L}_2$ holds, which is a contradiction.

For the second part, by Lemma~\ref{lem:medium}, with probability $1 - o(1)$, $A'$ has no kernel vectors with support size $k \in [\frac{n}{8e^4d^2}, \frac{n}{C}]$. Further, the third implication of Corollary~\ref{cor:small} tells us that with probability $1 - o(1)$, $A'$ has no kernel vector with $n$ in the support and support size at most $\frac{n}{8e^4d^2}$. 
\end{proof}

We are now ready to prove Theorem~\ref{char:square}.
\begin{proof}[Proof of Theorem~\ref{char:square}]
First observe that if $\frac{3\log(n)}{n} \leq p \leq 1/2$, then the theorem is trivially true since $A$ is non-singular with probability $1 - o(1)$. Combining Lemma \ref{lemma:large_plug} with the Lemmas~\ref{lem:medium} and \ref{lem:basis} presented above, with probability $1 - o(1)$, for any $k \geq \frac{n}{8e^4d^2}$, there are no minimal dependencies of $k$ columns of $A$. Further applying Corollary~\ref{cor:small}, for $k \leq \frac{n}{8e^4d^2}$ we see that with probability $1 - o(1)$, all minimal dependencies of $k$ columns of $A$ must be in $\mathcal{T}_k$.
\end{proof}

\subsection{Proof of Theorem~\ref{char:bgc}}
For the convenience of the reader, we restate this theorem:

\charbgc*

\begin{proof}
First observe that if $\frac{3\log(n)}{n} \leq p \leq 1/2$, then the theorem is trivially true since $A$ is non-singular with probability $1 - o(1)$. Assuming $p \leq \frac{3\log(n)}{n}$, this theorem follows directly from combining Corollary~\ref{cor:small}, Lemma~\ref{lemma:medium}, Lemma~\ref{lemma:medium2}, and Lemma~\ref{lemma:large2_B}. We apply Corollary~\ref{cor:small}, Lemma~\ref{lemma:medium}, and Lemma~\ref{lemma:medium2} to both $B$ and $B^\top $ to show that with probability $1 - o(1)$, neither $B$ nor $B^\top $ has minimal dependencies of $k$ columns outside of $\mathcal{T}_k$  for $k \leq \frac{n}{C}$, where $C$ is the constant from Lemma~\ref{lemma:medium2}. Finally, Lemma~\ref{lemma:large2_B} shows that either $B$ or $B^\top $ has no minimal linear dependencies of more than $n/C$ columns.
\end{proof}

\section{Proof of Theorem~\ref{thm:ks} and Theorem~\ref{thm:ksB} on the Rank of A and B}\label{sec:ks}

We will need the following three simple linear algebraic lemmas.

\begin{lemma}[One-Sided Leaf Removal]\label{lemma:peeling}
Let $A \in \mathbb{R}^{n \times m}$, and let $A' \in \mathbb{R}^{n - 1 \times m - 1}$ be the matrix obtained be removing from $A$ a column with a single $1$ in it, and the row which indexed the location of the $1$. Then $\on{corank}(A) = \on{corank}(A').$ Furthermore, the union of the supports of vectors in the kernel of $A'$ is a subset of the union of the supports of vectors in the kernel of $A$.
\end{lemma}
\begin{proof}
Without loss of generality, assume the last column and first row were removed. Let $K'$ be the kernel of $A'$ and $K$ be the kernel of $A$. Define $\phi: \mathbb{R}^{m - 1} \rightarrow \mathbb{R}^m$, by
$\phi(x') = (x', -y)$, where $y = \sum_{i = 1}^{m - 1}A_{1, i}x'_i$. It suffices to show that $\phi$ is a one-to-one and onto mapping from $K'$ to $K$. The one-to-one is by definition. $\phi$ maps from $K'$ to $K$ since $(A\phi(x'))_j = (A'x')_j$ for $j \neq 1$, and $(A\phi(x'))_1 = \sum_{i = 1}^{m - 1}A_{1, i}x'_i-y = 0$. The mapping is onto since for any $x \in K$, writing $x = (-y, x')$, we must have $A'x' = 0$ and $y = \sum_{i = 1}^{m - 1}A_{1, i}x'_i$, because the first row of $A$ is orthogonal to $x$.
\end{proof}

\begin{lemma}[Leaf Removal]\label{lemma:peeling2}
Let $A \in \mathbb{R}^{n \times n}$ be a the adjacency matrix of a graph $G$. Let $A' \in \mathbb{R}^{n - 2 \times n - 2}$ be the adjacency matrix of the graph $G'$ obtained from $G$ by removing any vertex of degree one and its unique neighbor. Then $\on{corank}(A) = \on{corank}(A').$ Furthermore the set of vertices in the support of the kernel of $A'$ is contained in set of vertices in the support of the kernel of $A$.
\end{lemma}
\begin{proof}
We apply Lemma~\ref{lemma:peeling} twice. First we remove the column corresponding to the vertex and the row corresponding to its neighbor. This does not change the corank. Then we remove the row corresponding to the vertex and the column corresponding to its neighbor. This does not change the corank, since the matrix is square, so the corank of the row space is the same as the corank of the column space. 

To see the last statement, assume without loss of generality the leaf vertex was $n$ and its unique neighbor was $n - 1$. Then if $A'x' = 0$, we have $Ax = 0$, where $x = (x', 0, -y)$, where $y = \langle(A_{n - 1})_{[n-2]}, x'\rangle = \sum_{i = 1}^{n - 2}A_{2, i}x'_i$. Clearly the support of $x'$ is contained in the support of $x$.
\end{proof}

\begin{lemma}\label{lemma:support_minimal}
For $A \in \mathbb{R}^{m \times n}$, if $Ax = 0$, then for any $i \in \on{supp}(x)$, there exists some set $S$ with $i \in S$ such that $A_S \in \mathcal{M}_{|S|}$.
\end{lemma}
\begin{proof}
It suffices to show that if $A_{\supp(x)}$ is not a minimal dependency, then it is possible to find some subset $T \subsetneq \supp(x)$ such that $i \in T$, and there exists some $z$ with $\supp(z) = T$ such that $Az = 0$. By repeating this process, we will eventually find some set $S$ such that $A_S$ is a minimal dependency. 

If $A_{\supp(x)}$ is not a minimal dependency, then by definition, there exists some $y$ with $\supp(y) \subseteq \supp(x)$ such that $x^\top y = 0$ and $Ay = 0$. There must exist some index $j \neq i$ such that $\frac{y_j}{x_j} \neq \frac{y_i}{x_i}$, otherwise, we would not have $x^\top y = 0$. Let $z = x - \frac{x_j}{y_j}y$. Notice that $Az = 0$ and $j \notin \supp(z)$, so $\supp(z) \subsetneq \supp(x)$ and $i \in \supp(z)$.
\end{proof}




We are now ready to prove Theorem~\ref{thm:ks} using Theorem~\ref{char:square}. We restate Theorem~\ref{thm:ks} for the readers convenience. Recall that for a graph $G$, we use $G_{\on{KS}}$ to denote the Karp-Sipser core of $G$.

\thmksA*

\begin{proof}
Let $G$ be the graph with adjacency matrix $A$, and let $A'$ be the adjacency matrix of the graph $G'$ that remains after iteratively peeling off vertices of degree one from $G$ and their unique neighbors. Let $P_{\on{KS}} \subseteq [n]$ be the set of rows peeled during this process. By repeatedly applying Lemma~\ref{lemma:peeling2}, we have $$\on{corank}(A) = \on{corank}(A').$$

Let $I_{\on{KS}} \subseteq [n]$ be the set isolated vertices that remain. Let $G_{\on{KS}}$ be the graph that remains after removing those vertices, and $A_{\on{KS}} := A(G_{\on{KS}})$ be the associated adjacency matrix. Then $$\on{corank}(A_{\on{KS}}) = \on{corank}(A') - |I_{\on{KS}}|  =  \on{corank}(A) - |I_{\on{KS}}|.$$

It remains to show that on the event that the conclusion of Theorem~\ref{char:square} holds, which occurs with probability $1 - o(1)$, we have $\on{corank}(A_{\on{KS}}) = 0$. Let $D$ be the set of columns involved in linear dependencies of $A$:
$$D := \bigcup_{x: Ax = 0}{\on{supp}(x)} = \{ i \in [n]: \exists S \ni i: A_S \in \mathcal{M}_{|S|}\}.$$

The equivalence holds because by Lemma~\ref{lemma:support_minimal}, any column involved in a linear dependency must also be involved in a minimal linear dependency.

\begin{claim}\label{claim:contain} Conditional on the event in Theorem~\ref{char:square} holding,
$$D \subseteq P_{\on{KS}} \cup I_{\on{KS}}.$$
\end{claim}
\begin{proof}
Conditional on the event, we have
$$D = \{ i \in [n]: \exists S \ni i: A_S \in \mathcal{T}_{|S|}\}.$$

Consider any set $S$ for which $A_S \in \mathcal{T}_{|S|}$. By the definition of $\mathcal{T}_{|S|}$, $A_S$ is the edge-vertex incidence matrix of a tree. Hence the leaf-removal process and isolated vertex removal step will necessarily remove all the columns in $S$. This shows that $D \subseteq P_{\on{KS}} \cup I_{\on{KS}}$. 
\end{proof}

Suppose for the sake of contradiction $A_{\on{KS}}$ was singular, and had some non-zero kernel vector $x$. Then by iteratively applying the second implication of Lemma~\ref{lemma:peeling2}, we see that the vertices in the support of the kernel of $A$, namely $D$, must contain the vertices supporting $x$. This implies that $D$ contains a vertex in the Karp-Sipser core of $A$, which contradicts the previous claim. It follows that with probability $1 - o(1)$, $A_{\on{KS}}$ is non-singular.
\end{proof}

The proof of Theorem~\ref{thm:ksB} is nearly identical.

\thmksB*

\begin{proof}
Let $H'$ be the bipartite graph that remains after iteratively peeling off vertices of degree one from $H$ and their unique neighbors, and let $B'$ be the associated bi-adjacency matrix (where we have retained the same left-right partition as $H$). Let $P_{\on{KS}}$ be the set of vertices peeled during this process. By repeatedly applying Lemma~\ref{lemma:peeling}, we have $$\on{corank}(B) = \on{corank}(B').$$

Let $I_{\on{KS}} \subseteq [n]$ be the set isolated vertices in $H'$. Let $H_{\on{KS}}$ be the graph that remains after removing those vertices, and $B_{\on{KS}} := B(H_{\on{KS}})$ be the associated bi-adjacency matrix. Then 
\begin{equation}\label{eq:left_corank}
\on{corank}(B_{\on{KS}}) = \on{corank}(B') - |I_{\on{KS}} \cap L|  =  \on{corank}(B) - |I_{\on{KS}} \cap L|,   
\end{equation}
and 
\begin{equation}\label{eq:right_corank}
\on{corank}(B_{\on{KS}}^\top) = \on{corank}(B'^\top) - |I_{\on{KS}} \cap R|  =  \on{corank}(B^\top) - |I_{\on{KS}} \cap R|.
\end{equation}
It remains to show that on the event that the first conclusion of Theorem~\ref{char:bgc} holds, we have $\on{corank}(B_{\on{KS}}) = 0$ or $\on{corank}(B_{\on{KS}}^\top) = 0$. We condition on the event that the conclusion of Theorem~\ref{char:bgc} holds, which occurs with probability $1 - o(1)$. Without loss of generality, assume the first event holds, which implies that for all $S$ such that $B_S \in \mathcal{M}_{|S|},$ we have $B_S \in \mathcal{T}_{|S|}.$

Let $D$ be the set of columns involved in linear dependencies of $B$:
$$D := \bigcup_{x: Bx = 0}{\on{supp}(x)} = \{ i \in [n]: \exists S \ni i: B_S \in \mathcal{M}_{|S|}\}.$$ The following claim is proved in the same way as Claim~\ref{claim:contain}, so we omit its proof.
\begin{claim} Conditional on the first event in Theorem~\ref{char:bgc} holding,
$$D \subseteq (P_{\on{KS}} \cup I_{\on{KS}}) \cap R.$$
\end{claim}

Suppose for the sake of contradiction $B_{\on{KS}}$ did not have full column-rank, and had some non-zero kernel vector $x$ such that $ B_{\on{KS}}x = 0$. Then by iteratively applying the second implication of Lemma~\ref{lemma:peeling}, we see that the right vertices in the support of the kernel of $B$, namely $D$, must contain the vertices supporting $x$. This implies that $D$ contains a vertex in the right side of the Karp-Sipser core of $A$, which contradicts the previous claim. It follows that $\on{corank}(B_{\on{KS}}) = 0$.

To prove the final statement of Theorem~\ref{thm:ksB}, observe that from eq.~\ref{eq:left_corank} we have 
$$\on{corank}(B) = |I_{\on{KS}} \cap L|.$$ Further, since $\on{corank}(B_{\on{KS}}^\top) \geq \on{corank}(B_{\on{KS}}) = 0$ and $\on{corank}(B) = \on{corank}(B^\top)$, it follows from eqs.~\ref{eq:left_corank} and \ref{eq:right_corank} that $|I_{\on{KS}} \cap L| \geq |I_{\on{KS}} \cap R|$.
\end{proof}

\section{Proof of Theorem~\ref{thm:core}}\label{sec:core}
In the following section, we prove our result showing that if $p = \omega(1/n)$, then with high probability, the $k$-core of of $G \sim \mathbb{G}(n, p)$ is full rank. Throughout this section, we let $d := pn$. 

As outlined in Section~\ref{sec:overview}, the proof of the Theorem~\ref{thm:core} is divided in two main domains corresponding to the size the minimal dependencies we consider.

\subsection{Small Domain}
\subsubsection{Small Case 1}
We prove the following lemma, which rules out minimal linear dependencies of $k$ columns of $\on{core}_k(A)$ for $k < \frac{n}{d^7}$. Our approach is similar to the prior Small Case 1.

\begin{lemma}[Small Case 1 for $k$-Core]\label{lemma:small}
Let  $A \sim \mathbb{A}(n, \frac{d}{n})$, where $d_0 \leq d \leq 3\log(n)$ for some universal constant $d_0$. For a set $S$ of columns and subset $R$ of rows, let $\mathcal{E}_{S, R}$ be the event that:
\begin{enumerate}
    \item $S \subset R$;
    \item $1 \leq |S| \leq \frac{n}{d^7}$;
    \item $|\supp(A^R_S)| \geq 3|S|$;
    \item $A^R_S$ contains no more than one row with a single $1$.
\end{enumerate}

Then \[\Pr\left[\exists S, R: \mathcal{E}_{S, R} ~\on{occurs}\right] = O\left(\frac{1}{\sqrt[3]{n}}\right).\]







\end{lemma}
\begin{remark}
A slightly weaker version of this lemma would require that $A_S^R$ contains no rows with a single $1$. By taking $R$ to be the set of vertices in the $k$-core, this weaker version would suffice to rule out with high probability any minimal dependencies in $\on{core}_k(A)$ with $t$ columns for $t \leq \frac{n}{d^7}$. We state this lemma to be stronger because we will also want to rule out small minimal dependencies in $\on{core}_k(A)^{[n-1]}$, akin to how we studied $A^{[n-1]}$ for the prior results in this paper on the rank of $A$.
\end{remark}

\begin{proof}
We introduce some notation, pictured in Figure~\ref{fig:symmetric_core}. Fix a set $S$ of $k$ columns and consider the submatrix $A_S$ induced by these columns. Let $E_{Sym}$ be the set of entries of $A_S$ whose rows are indexed by values in $S$. Let $E_{SymBD}$ be subset of entries in $E_{Sym}$ that are below the diagonal of $A$, and hence mutually independent. Let $E_{Asym}$ be the set of entries who rows are not in $S$, and let $E = E_{Asym} \cup E_{SymBD}$ be the full set of mutually independent entries that determine $A_S$. Formally:
\begin{equation*}
\begin{split}
&E_{Sym} := \{(i, j): i, j \in S \}\\
&E_{SymBD} := \{(i, j): i, j \in S, j < i \} \\
&E_{Asym} := \{(i, j): j \in S, i \notin S \}\\
&E := E_{SymBD} \cup E_{Asym}
\end{split}
\end{equation*}
\begin{figure}
        \centering
        \includegraphics[width=5cm]{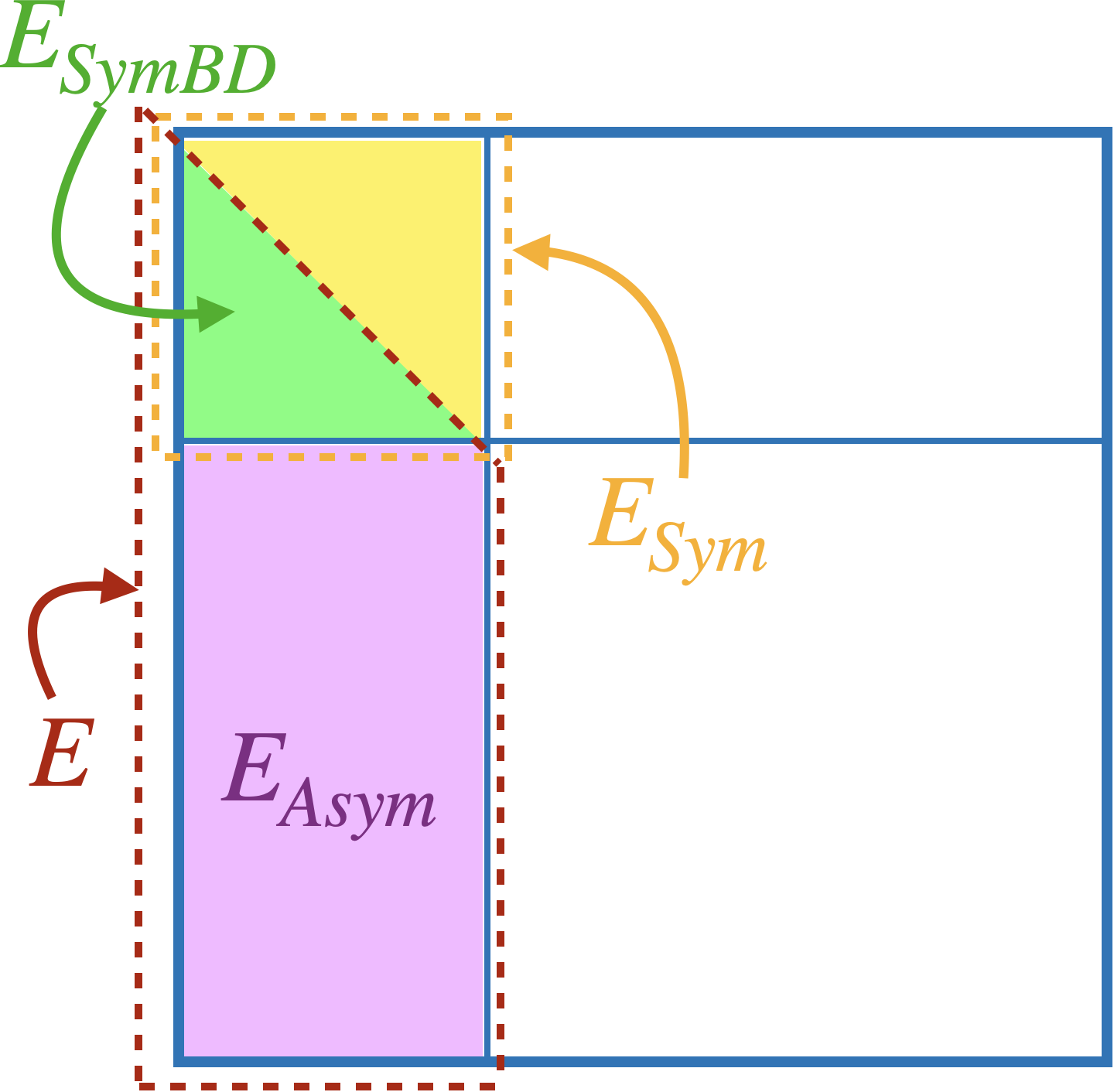}
    \caption{Regions of $A_S$ with $S = [|S|]$.}
    \label{fig:symmetric_core}
\end{figure}

For $R \subset [n]$, let $E_{Sym}[R], E_{SymBD}[R], E_{Asym}[R]$ and $E[R]$ denote the respective set additionally restricted to entries with $i \in R$. Observe that it suffices to consider $k \geq 2$, because otherwise having at least three $1$s in $A_S^R$ would imply having at least three rows of $A_S^R$ with a single $1$.

We will couple the process of randomly exposing the $1$'s in $A_S$ with a random walk that counts the number of times a $1$ is exposed in $E_{Asym}$ in a row that already contains a $1$ or in $E_{SymBD}$. Observe that each consecutive $1$ is exposed in a random entry in $E$, sampled without replacement. We condition on $L_S$, the number of $1$'s in $E$. Note that $L_S \sim \Bin(|E|, d/n)$. Let $(X_{i})_{i \in L_S}$ be the random walk that increases by one if the $i$th exposed $1$ is in $E_{SymBD}$ or if the $i$th exposed $1$ is in $E_{Asym}$ and is \em not \em the first $1$ exposed in its row. Otherwise, let $X_i = X_{i - 1}$.


The following claim says that $X_{L_S}$ must be large for $\mathcal{E}_{S, R}$ to occur for any $R$.
\begin{claim}
Suppose there exists some $R$ such that event $\mathcal{E}_{S, R}$ occurs. Then $X_{L_S} \geq \lceil{\frac{3k - 1}{2}}\rceil$.
\end{claim}
\begin{proof}
Let $P$ be the number of $1$'s in $E_{Asym}[R]$, and let $Q$ be the number of $1$s in $E_{SymBD}[R]$. Since  $S \subseteq R$, we have $2Q + P = |\supp(A_S^R)|$. It follows by the third condition for the event that $2Q + P \geq 3k$.
Let $Y$ be the number of $1$s in $E_{Asym}[R]$, which are \em not \em the first $1$ in their row. Then $Y \leq X_{L_S} - Q$.
Now the number of rows in  $A_S^R$ which have exactly one $1$ is at least $(P - Y) - Y$,
since there are exactly $P - Y$ non-empty rows in $A_S^R$, and at most $Y$ of those rows have \textit{more} than one $1$.
Now 
\[P - 2Y \geq P - 2X_{L_S} + 2Q \geq 3k - 2X_{L_S}.\]

Hence if $X_{L_S} < \lceil{\frac{3k - 1}{2}}\rceil,$ then since $X_{L_S}$ is an integer, we have \[X_{L_S} \leq \begin{cases}\frac{3k}{2} - 1 & k \:\on{even} \\ \frac{3k - 1}{2} - 1& k \:\on{odd} \end{cases} \leq \frac{3k}{2} - 1,\] so 
$3k - 2X_{L_S} \geq 3k - (3k - 2) \geq 2,$
meaning that at least two rows in $A_S^R$ have exactly one $1$. It follows that we must have $X_{L_S} \geq \lceil{\frac{3k - 1}{2}}\rceil$ if $A_S^R$ contains at most one row with a single one.
\end{proof}


It remains to show that $X_{L_S}$ is small enough with high enough probability to take a union bound over all $\binom{n}{k}$ sets $S$. We will break up this probability as follows. 

Let $Q := \ceil*{\frac{3k-1}{2}},$ and define $L_{max} := \max\left(\frac{1}{4}n^{\epsilon}k^{1 - \epsilon}, 3dk\right)$, where $\epsilon = \frac{1}{12}$. Then (for a fixed $S$),
\begin{equation*}
\Pr[\exists R: \mathcal{E}_{S, R} ~\on{occurs}] \leq \Pr\left[L_S > L_{max}\right] + \Pr\left[X_{L_S} \geq Q | L_S < L_{max} \right].
\end{equation*}

We consider first the second term in the equation above. To bound this term, we couple $X_i$ with a random walk that $(Y_i)_{i \in L_S}$ which increases by $1$ with probability $\frac{k + L_S}{n}$ and otherwise stays constant. Observe that $Y_i$ stochastically dominates $X_i$, because there are at most $k + L_S$ rows in $E$ in which placing a $1$ will increase $X_i$.
Then conditioned on $L_S = \ell$,
\begin{equation*}
\begin{split}
\Pr\left[X_{\ell} \geq Q\right] &\leq \Pr\left[\Bin\left(\ell, \frac{k + \ell}{n}\right) \geq Q\right].
\end{split}
\end{equation*}

The following lemma, which we prove in Appendix~\ref{apx:core}, bounds this probability.

\begin{restatable}[Binomial Bound 1 for $k$-core Small Case 1]{lemma}{claimxr}\label{claim:xr}
For any $d$ larger than some constant, for any $2 \leq k \leq \frac{n}{d^7}$, and for any $\ell \leq \max\left(\frac{1}{4}n^{\epsilon}k^{1 - \epsilon}, 3dk\right)$,
\begin{equation*}
    \Pr\left[\Bin\left(\ell, \frac{k + \ell}{n}\right) \geq Q\right] \leq 2\left(\frac{k}{3n}\right)^{k + \frac{1}{3}},
\end{equation*}
where $\epsilon = \frac{1}{12}$ and $Q = \ceil*{\frac{3k-1}{2}}$.
\end{restatable}

Next we bound the probability that $L_S \geq L_{max}$. Recall that $L_S$ is the number of $1$'s in $E_{Asym}$ and $E_{SymBD}$. Now $|E_{Asym} \cup E_{SymBD}| \leq kn$, and each entry in $E_{Asym} \cup E_{SymBD}$ is $1$ with probability $d/n$. Hence
\begin{equation*}
\begin{split}
\Pr[L_S \geq L_{max}] \leq \Pr\left[\Bin\left(kn, \frac{d}{n}\right) \geq L_{max}\right].
\end{split}
\end{equation*}

The following lemma, which we prove in Appendix~\ref{apx:core}, bounds this binomial tail probability.

\begin{restatable}[Binomial Bound 2 for $k$-core Small Case 1]{lemma}{rmax}\label{Rmax}
For sufficiently large $n$ and $d$, if $d \leq 3\log(n)$,
\begin{equation*}
\Pr\left[\Bin\left(kn, \frac{d}{n}\right) \geq  L_{max} \right] \leq 2\left(\frac{k}{en}\right)^ke^{-2\log(n)},
\end{equation*}
where $L_{max} = \max(3dk, \frac{1}{4}n^{1/12}k^{11/12})$.
\end{restatable}

We finish the proof of this lemma by union bounding over all $k$, and all sets $S$ of $k$ rows. Plugging in the result of Lemma~\ref{Rmax}, we have
\begin{equation*}
\sum_{k = 1}^{n/d^7}\binom{n}{k}\Pr[L_S > L_{max}] \leq \sum_{k = 1}^{n/d^7}\binom{n}{k}4\left(\frac{k}{en}\right)^ke^{-2\log(n)} \leq \sum_{k = 1}^{n/d^7} \frac{4}{n^2} \leq \frac{1}{n}.
\end{equation*}

Hence with probability at least $\frac{1}{n}$, for all sets of $k \leq n/d^7$ rows, there are at most $L_{max} = \max(3dk, \frac{1}{4}n^{1/12}k^{11/12})$ $1$'s among the rows. Assuming $L_S \leq L_{max}$, we plug in the result of Lemma~\ref{claim:xr} to obtain:
\begin{equation*}
\begin{split}
\sum_{k = 1}^{n/d^7}\binom{n}{k}\Pr[X_{L_S} \geq Q \land L_S \leq L_{max}] &\leq \sum_{k = 1}^{n/d^7}\binom{n}{k}4\left(\frac{k}{3n}\right)^{k + \frac{1}{3}}
\leq  4\sum_{k = 1}^{n/d^7}\left(\frac{ne}{k}\right)^k\left(\frac{k}{3n}\right)^{k + \frac{1}{3}} \\
&\leq 4\sum_{k = 1}^{n/d^7}\left(\frac{e}{3}\right)^k\left(\frac{k}{3n}\right)^{1/3} = O\left(\frac{1}{\sqrt[3]{n}}\right).
\end{split}
\end{equation*}

It follows that 
\begin{equation*}
    \Pr[\exists S, R: \mathcal{E}_{S, R} ~\on{occurs}] \leq \sum_{S \subset [n]: |S| \leq \frac{n}{d^7}}\Pr\left[L_S > L_{max}\right] + \Pr\left[X_{L_S} \geq Q | L_S < L_{max} \right] \leq \frac{1}{n} + O\left(\frac{1}{\sqrt[3]{n}}\right).
\end{equation*}
\end{proof}

\subsubsection{Small Case 2}
\begin{lemma}[Small Case 2 for $k$-core]\label{lemma:medium_core}
Let $A \sim \mathbb{A}(n, d/n)$. For a set of rows $R$ and a set of columns $S$, let $\mathcal{E}_{S, R}$ be the event that 
\begin{enumerate}
    \item $\frac{n}{d^7} \leq |S| \leq \frac{2n}{d}$;
    \item $|R| \geq \left(1 - \frac{1}{d^7}\right)n$;
    \item $A_S^R$ contains no more than one row with a single $1$.
\end{enumerate}

There exists a constant $d_0$ such that for all $d_0 < d \leq 3 \log n$,
\[\Pr\left[\exists S, R: \mathcal{E}_{S, R} ~\on{occurs}\right] = e^{-\Theta(n)}.\]

\end{lemma}
\begin{proof}
We union bound over all $\frac{n}{d^7} \leq k \leq \frac{2n}{d}$, and all $\binom{n}{k}$ sets $S$ of $k$ columns. We will lower bound the number of length $k$ rows in $A_S$ which have exactly one $1$. We consider only the $n - k$ mutually independent rows. Let $X_i$ be the event that the $i$th row, $A_S^{\{i\}}$, has one $1$ for $i \in [n] \setminus S$. Then
\begin{equation*}
\Pr[X_i] = k\frac{d}{n}\left(1 - \frac{d}{n}\right)^{k - 1}.
\end{equation*}
Let $c := \frac{n}{dk},$ such that $\frac{1}{2} \leq  c < d^6$. Now because $1 - x \geq e^{-\frac{x}{\sqrt{1 - x}}}$ for $0 \leq x \leq 1$, we have \[\left(1 - \frac{d}{n}\right)^{k - 1} \geq \left(1 - \frac{d}{n}\right)^{k} = \left(1 - \frac{d}{n}\right)^{\frac{n}{cd}} \geq \left(e^{-\frac{\frac{d}{n}}{\sqrt{1 - \frac{d}{n}}}}\right)^{\frac{n}{cd}}= e^{-\frac{1}{c\sqrt{1 - \frac{d}{n}}}} \geq \frac{1}{20},\]

so \begin{equation*}
\Pr[X_i] \geq\frac{e^{-\frac{1}{c\sqrt{1 - \frac{d}{n}}}}}{c} \geq \frac{1}{20c},
\end{equation*}

If $A_S^R$ has at most one row with a single one, then necessarily $\sum{X_i} \leq \frac{n}{d^7} + 1 \leq \frac{n}{80d^6}$, assuming sufficiently large $n$. By a Chernoff bound (see Lemma~\ref{chernoff} for reference), for $d$ larger than a large enough constant, for a fixed $S$,
\begin{equation*}
\begin{split}
\Pr[\exists R: \mathcal{E}_{S, R} ~\on{occurs}] &\leq \Pr\left[\sum{X_i} < \frac{n}{80d^6}\right]
\leq \Pr\left[\Bin\left(n - k, \frac{1}{20c}\right) < \frac{n}{80d^6}\right]\\
&\leq \Pr\left[\Bin\left(\frac{n}{2}, \frac{1}{20c}\right) < \frac{n}{80d^6}\right]
\leq e^{-\frac{n}{2}\frac{1}{40c}\left(1 - \frac{c}{2d^6}\right)^2}
\leq e^{-\frac{n}{320c}},  
\end{split}
\end{equation*}
where we have used in the final inequality the fact that $c \leq d^6$.

Union bounding over all $\binom{n}{k}$ choices of $S$, and all choices of $k$, we have
\begin{equation*}
\begin{split}
\Pr[\exists S, R: \mathcal{E}_{S, R} ~\on{occurs}] &\leq
\sum_{k = \frac{n}{d^7}}^{\frac{2n}{d}}\binom{n}{k} e^{-\frac{n}{320c}} \leq n\max_{c \in [\frac{1}{2}, d^6]}\left(\frac{ne}{\frac{n}{cd}}\right)^{\frac{n}{cd}} e^{-\frac{n}{320c}}\\
&= n\max_{c \in [\frac{1}{2}, d^6]}\left(ecd\right)^{\frac{n}{cd}} e^{-\frac{n}{320c}}
= \max_{c \in [\frac{1}{2}, d^6]}{e^{n\left(\frac{\log(n)}{n} + \frac{\log(ecd)}{cd} - \frac{1}{320c}\right)}}.
\end{split}
\end{equation*}

Since $c$ grows at most polynomially in $d$, for $d$ a large enough constant, the exponent becomes negative, this probability is $e^{-\Theta(n)}$.
\end{proof}

\subsection{Large Domain}
\subsubsection{Large Case 1}
\begin{lemma}[Large Case 1 for $k$-core]\label{lemma:large}
Let $A \sim \mathbb{A}(n, d/n)$. For a set of rows $R$ and a set of columns $S$, let $\mathcal{E}_{S, R}$ be the event that 
\begin{enumerate}
    \item $\frac{2n}{d} \leq |S| \leq \frac{n}{C}$;
    \item $|R| \geq \left(1 - \frac{1}{d^7}\right)n$;
    \item $A^Rx = 0$ for some $x$ with $\supp(x) = S$.
\end{enumerate}
There exists a constant $d_0$ and $C$ such that for all $d_0 \leq d \leq 3\log(n)$,
\[\Pr\left[\exists S, R: \mathcal{E}_{S, R} ~\on{occurs}\right] = e^{-\Theta(n)}.\]

\end{lemma}
\begin{proof}
We union bound over all $k$ satisfying $\frac{2n}{d} \leq k \leq \frac{n}{C}$, and all $\binom{n}{k}$ sets $S$ of $k$ columns, and over all sets $R$ of size $\left(1 - \frac{1}{d^7}\right)n$. Fix a set of $k$ columns $S$ and a subset $R$ of rows. We will consider only the at least $n - k - \frac{n}{d}$ mutually independent rows of $A_S^R$, that is, the matrix $A_{S}^{R \setminus S}$.

Consider the following process, where we expose the independent rows of $A_{S}^{R \setminus S}$ one at a time for $i = 1$ to $|R \setminus S|$. Let $\mathcal{A}_i \in \mathbb{R}^k$ be the nullspace of the first $i$ rows vectors exposed, and let $D_i$ be the dimension of the smallest subspace of $\mathcal{A}_i$ which contains all vectors in $\mathcal{A}_i$ which have no zeros.

By Lemma~\ref{lo_sparse}, if $D_i > 0$, then we can choose an arbitrary vector $v$ in $\mathcal{A}_i$ with support $k$, and with probability at least $1 - \frac{1}{\sqrt{kd/n}}$, the $(i + 1)$th row exposed is not orthogonal to $v$. In this case $D_{i + 1} = D_i - 1$. If $D_i$ ever becomes $0$, then $A_{S}^{R \setminus S}$ has no kernel vectors with no zeros, so there is no kernel vector of $A^R$ with support $S$. Since $D_0 = k$, a Chernoff bound (see Lemma~\ref{chernoff}) yields
\begin{equation*}
\begin{split}
\Pr[D_{|R \setminus S|} \neq 0] &\leq \Pr\left[\Bin\left(n - k - \frac{n}{d}, 1 - \frac{1}{\sqrt{kd/n}}\right) < k\right] 
\leq \Pr\left[\Bin\left(\frac{n}{2}, 1 - \frac{1}{\sqrt{kd/n}}\right) < k\right]\\
&\leq e^{-\frac{n}{4}\left(1 - \frac{1}{\sqrt{kd/n}}\right)\left(1 - \frac{1}{\sqrt{kd/n}} - \frac{2k}{n}\right)^2}
\leq  e^{-n\epsilon},
\end{split}
\end{equation*}
where $\epsilon > 0.001$. To achieve this value of $\epsilon$, we impose that $C \geq 12$ such we can plug in  $2k/n \leq \frac{1}{6}$ in the last inequality. Finally union bounding over all $R$ and $S$, the probability that $\mathcal{E}_{S, R}$ occurse for some $R$ and $S$ is at most
\begin{equation*}
\sum_{k = \frac{2n}{d}}^{\frac{n}{C}}\binom{n}{k}\binom{n}{\frac{n}{d^7}}e^{-\epsilon n} \leq n\left(eC\right)^{\frac{n}{C}}\left(ed^7\right)^{\frac{n}{d^7}}e^{-\epsilon n} \leq e^{n\left(\frac{\log(n)}{n} + \frac{\log(eC))}{C} + \frac{\log(ed^7))}{d^7} - \epsilon\right)},
\end{equation*}
which for constants $C$ and $d$ large enough, is $e^{-\Theta(n)}$. This concludes the lemma.
\end{proof}

\subsubsection{Large Case 2}
We will use the following lemmas from Section~\ref{sec:large2}, which we restate here.

\nullspace*
\basis*

We will use the following notation: For an adjacency matrix $M$, let $C_k(M)$ be the set of vertices in the $k$-core of the graph with adjacency matrix $M$. Further, let $M(k)$ denote adjacency matrix of the $k$-core of $M$. For a matrix $M$ and set $S$, we let $M[S] := M_S^S$ denote the restriction of $M$ to the columns \textit{and} rows in $S$. Thus $M(k) = M[C_k(M)]$. Recall that we define $M^{(i)}$ to be $M$ with both the $i$th row and $i$th column removed.

The following lemma, which follows from \cite{pittel},  states that the $k$-core of $A$ is large.
\begin{lemma}[c.f. Theorem 2 in \cite{pittel}]\label{lemma:core_size}
Let $A \sim \mathbb{A}(n, d/n)$. For any constant $k \geq 3$, and for $d \geq d_0(k)$, with probability at least $1 - O(\exp(-n^{0.1}))$, \begin{equation*}
 |C_k(A)| \geq  n\left(1 - d^ke^{-d/2}/(k - 1)! - n^{-1/6}\right).
\end{equation*}

If $d = \omega(1)$ and $d \leq 3\log(n)$, then with probability at least $1 - O(\exp(-n^{0.1}))$, \[|C_k(A)| \geq  n\left(1 - e^{-d/20}\right).\]
\end{lemma}

Our main tool to prove Theorem~\ref{thm:core} is the following lemma, which rules out large dependencies in $A(k)$.

\begin{lemma}\label{lemma2core}
Let $A \sim \mathbb{A}(n, d/n)$ for $d \leq 3
\log(n)$ and $d = \omega(1)$. For any $u, r < s, t \in [n]$,
\begin{equation}\label{eq:lem_core}
\begin{split}
\Pr\left[\exists x: A(k)x = 0, \supp(x) > t\right] &\leq 
\frac{n}{t - n/d}\max_{x \in \mathbb{R}^{n}: \supp(x) \geq s}\Pr\left[x^\top A_n = 0\right]\\
&\quad +\frac{n}{t - n/d}\Pr\left[\exists x: A^{(n)}(k)x = 0, r < |\supp(x)| < s \right]\\
&\quad + \frac{dr}{t - n/d}\\
&\quad+ \frac{n}{t - n/d}\frac{n}{u}\Pr\left[\exists x \neq 0: A^{(n)}(k)x = e_1, |\supp(x)| < s\right]\\
&\quad+ \frac{n}{t - n/d}\max_{X \in \mathcal{T}^{n}_{n - n/d^2 - r - u, s}}\Pr\left[A_n^{\top}XA_n = 0\right] \\
&\quad+ o\left(\frac{n}{t - n/d}\right),
\end{split}
\end{equation}
where $\mathcal{T}^{m}_{p, q}$ denotes the set of matrices in $\mathbb{R}^{m \times m}$ with some set of $p$ columns that each have at least $q$ non-zero entries.

\end{lemma}

\begin{proof}

For $i \in [n]$, let $S_i := C_k(A^{(i)}) \cup i$, and let $\mathcal{H}_i$ be the space spanned by the column vectors \[A_1^{S_i}, A_2^{S_i}, \cdots A_{i - 1}^{S_i}, A_{i + 1}^{S_i}, \cdots, A_n^{S_i}.\]

Let $T := \{i: C_k(A) = S_i\}.$ Then $A(k)x = 0$ for some $x$ implies that for all $i \in \supp(x) \cap T,$ we have $A_i^{S_i} = A_i^{C_k(A)} \in \mathcal{H}_i,$ since \[x_i A_i^{C_k(A)} = -\sum_{j} x_j A_j^{C_k(A)}.\]

We first claim that the set $T$ is large, such that if $\supp(x)$ is large, we must have many $i$ for which $A_i^{S_i} \in H_i$. Let $c_i := |C_k(A^{(i)})|$.

\begin{claim}\label{claim:T}
With probability $1 - o(1)$, $|T| \geq n\left(1 - \frac{1}{d}\right)$.
\end{claim}
\begin{proof}
Fix $i$, and consider the probability that $C_k(A) = C_k(A^{(i)}) \cup i$.

Notice that $C_k(A^{(i)})$ is independent of $A_i$, and hence
\begin{equation*}
\begin{split}
\Pr[C_k(A) = C_k(A^{(i)}) \cup i] &\geq \Pr[\supp(A_i) \geq k \land A_{ij} = 0 \: \forall j \notin C_k(A^{(i)})] \\
& = \mathbb{E}_{c_i}\left[\left(1 - \frac{d}{n}\right)^{n - 1 - c_i}\Pr\left[\Bin\left(c_i, \frac{d}{n}\right) \geq k\right]\right]\\
& \geq \Pr\left[c_i \geq n\left(1 - \frac{1}{2d^3}\right)\right]e^{-\frac{1}{d^2}}\Pr\left[\Bin\left(n(1 - 1/2d^3), \frac{d}{n}\right) \geq k\right]\\
& \geq \Pr\left[c_i \geq n\left(1 - \frac{1}{2d^3}\right)\right]e^{-\frac{1}{d^2}}\left(1 - e^{-\frac{n\left(1 - \frac{1}{2d^3}\right)n\left(\frac{d}{n} - \frac{k}{n\left(1 - \frac{1}{2d^3}\right)}\right)^2}{2d}}\right)\\
&\geq \Pr\left[c_i \geq n\left(1 - \frac{1}{2d^3}\right)\right]e^{-\frac{1}{d^2}}(1 - e^{-d/3})\\
&\geq \Pr\left[c_i \geq n\left(1 - \frac{1}{2d^3}\right)\right]\left(1 - \frac{1}{d^{3/2}}\right),
\end{split}
\end{equation*}
where we used a Chernoff bound (see Lemma~\ref{chernoff}) to bound the binomial, and the last two inequalities hold for $d$ large enough relative to $k$.
By Lemma~\ref{lemma:core_size}, with probability $1 - O(\exp(-n^{0.1}))$, we have $c_i \geq n\left(1 - e^{-d/20}\right)$, so we have 
\begin{equation*}
\Pr[C_k(A) = C_k(A^{(i)}) \cup i] \geq 1 - \frac{1}{d^{3/2}} - O(\exp(-n^{0.1})).
\end{equation*}

Now by Markov's inequality, 
\begin{equation*}
\begin{split}
\Pr\left[|[n] \setminus T| \geq \frac{n}{d}\right] &\leq \frac{d}{n}\sum_i\left(1 - \Pr[C_k(A) = C_k(A^{(i)}) \cup i]\right) \\
&\leq d\left(1 - \left(1 - d^{-3/2} - O(\exp(-n^{0.1})\right)\right) = o(1).
\end{split}
\end{equation*}
\end{proof}

Let $\mathcal{E}_i$ denote the event that $A_i^{S_i} \in \mathcal{H}_i$. Let $X_i$ be the indicator of this event, and let $X = \sum_i X_i$. If $|T| \geq n\left(1 - \frac{1}{d}\right)$, at most $n/d$ vertices can be in the $k$-core but not in $T$. Thus by Markov's inequality, we have
\begin{equation}\label{markov_core}
\begin{split}
\Pr\left[\exists x : A(k)x = 0, |\supp(x)| \geq t \right] &\leq \Pr\left[|T| < n\left(1 - \frac{1}{d}\right)\right] + \Pr\left[X \geq t - n/d\right]\\
&\leq o(1) + \frac{\mathbb{E}[X]}{t - n/d}
= o(1) + \frac{\sum_i\Pr[\mathcal{E}_i]}{t - n/d}
= o(1) + \frac{n}{t - n/d}\Pr[\mathcal{E}_n]. 
\end{split}
\end{equation}

We will break down the probability $\Pr[\mathcal{E}_n]$ into several cases, depending on the size of the support of vectors in the kernel of $A^{(n)}(k)$. Let $S \subset C_k(A^{(n)})$ be the set of all $j$ such that $e_j \in \Span(A^{(n)}(k))$. By Claim~\ref{claim:nullspace} and  Claim~\ref{claim:basis},
\begin{equation*}
\ell:= \max(\supp(x): A^{(n)}(k)v = 0) = c_n - |S|.  
\end{equation*}

\textbf{Case 1:} $A^{(n)}(k)$ has a vector $x$ with large support, that is, $\ell \geq s$. 

\textbf{Case 2:} $A^{(n)}(k)$ has a vector $x$ with medium support, that is $r < \ell < s$. 

\textbf{Case 3:} $A^{(n)}(k)$ does not have any vectors with large or medium support vectors in its kernel, that is, $\ell \leq r$. 

We can expand
\begin{equation}\label{cases_core}
\Pr[\mathcal{E}_n] = \Pr[\mathcal{E}_n| \ell \geq s]\Pr[\ell \geq s] + \Pr[\mathcal{E}_n|r < \ell < s]\Pr[r < \ell < s] + \Pr[\mathcal{E}_n| \ell \leq r]\Pr[\ell \leq r].   
\end{equation}

Define $\textbf{a} := A_i^{C_k(A^{(n)})}$ to be the restriction of $A_i$ to the indices in the $k$-core of $A^{(n)}$.

To evaluate the probability of the first case, we condition on $A^{(n)}$, and let $x$ be any vector of support at least $s$ in the kernel of $A^{(n)}(k)$. Observe that $\mathcal{E}_n$ cannot hold if $x^\top \textbf{a}$ is non-zero. Since $\textbf{a}$ is independent from $x$, we have 
\begin{equation}\label{case_1_core}
\begin{split}
\Pr[\mathcal{E}_n| \ell \geq s]\Pr[\ell \geq s] &\leq \max_{x: \supp(x) \geq s}\Pr[x^\top \textbf{a} = 0] \leq  \max_{x \in \mathbb{R}^{n}: \supp(x) \geq s}\Pr\left[x^\top A_n = 0\right].
\end{split}
\end{equation}
Here the second inequality comes from considering $x$ over a larger domain. Combined with eq.~\ref{markov_core}, the contribution from this case yields the first term in the right hand side of eq.~\ref{eq:lem_core}.

For the second case, we bound
\begin{equation}\label{case_2_core}
\Pr[\mathcal{E}_n|r < \ell <  s]\Pr[r < \ell < s] \leq \Pr[r < \ell < s] \leq \Pr\left[\exists x: A^{(n)}(k)x = 0, r < |\supp(x)| < s \right].   
\end{equation}

Combined with eq.~\ref{markov_core}, the contribution from this case yields the second term in the right hand side of eq.~\ref{eq:lem_core}.

In this third case, we will show conditions under which we can algebraically construct a vector $v$ such that $A[C_k(A^{(n)}) \cup n]v = e_n$. This will imply by Claim~\ref{claim:basis} that under those conditions, $A_n^{C_k(A^{(n)}) \cup n} \notin \mathcal{H}_n$.

To define these conditions, recall that $S \subset C_k(A^{(n)})$ is the set of all $j$ such that $e_j \in \Span(A[C_k(A^{(n)})])$. For $j \in S$, let $w_j$ be any vector such that $A^{(n)}(k)w_j = e_j$. We next construct a sort of ``pseudoinverse" matrix $B \in \mathbb{R}^{C_k(A^{(n)}) \times C_k(A^{(n)})}$ as follows: For $j \in S$, define the column $B_j$ to to be $w_j$. Define all other entries columns of $B$ to be zero. The following claim shows a condition for $\mathcal{E}_n$ not holding.
\begin{claim}
If $\supp(\textbf{a}) \subset S$ and $\textbf{a}^\top B\textbf{a} \neq 0$, then $e_n \in \Span(A[C_k(A^{(n)}) \cup n])$.
\end{claim}
\begin{proof}
Let $w' := B\textbf{a} = \sum_{j \in S}{\textbf{a}_jw_j}$ such that $A^{(n)}(k)w' = \sum_{j \in S}{\textbf{a}_je_j}$. Hence if $\supp(\textbf{a}) \subset S$, $A^{(n)}(k)w' = \textbf{a}.$ Define $w \in \mathbb{R}^{C_k(A^{(n)}) \cup n}$ to be the vector equal to $w'$ on the coordinates in $C_k(A^{(n)})$ and $-1$ on coordinate $n$. Then 

\begin{equation*}
(A[C_k(A^{(n)}) \cup n]w)_j = 0 \qquad  \forall j \in C_k(A^{(n)})
\end{equation*}
and 

\begin{equation*}
(A[C_k(A^{(n)}) \cup n]w)_n = \textbf{a}^\top B\textbf{a}.
\end{equation*}

Evidently, if $\textbf{a}^\top B\textbf{a} \neq 0$, then \[\frac{A[C_k(A^{(n)}) \cup n]w}{\textbf{a}^\top B\textbf{a}} = e_n,\] so $e_n \in \Span(A[C_k(A^{(n)}) \cup n])$.
\end{proof}

By definition, in the third case, we have $|S| \geq c_n - r$. Hence
\begin{equation}\label{case_3_core}
\begin{split}
\Pr[\mathcal{E}_n \land \ell \leq r]
&\leq \Pr\left[\supp(\textbf{a}) \not\subset S \land |S| \geq c_n - r\right] + \Pr[\textbf{a}^\top B\textbf{a} = 0 \land |S| \geq c_n - r].
\end{split}
\end{equation}


Notice that $S$ is a function of $A^{(n)}$, so $\textbf{a}$ is independent from $S$. Thus for any set $S$, \[\Pr\left[\supp(\textbf{a}) \not\subset S\right] \leq 1 - \left(1 - \frac{d}{n}\right)^{n - |S|} \leq \frac{d(n - |S|)}{n}.\] 
By Lemma~\ref{lemma:core_size}, with probability $1 - o(1)$, we have $c_n \geq n\left(1 - \frac{1}{d^2}\right)$, thus
\begin{equation}\label{case_3_a_S_core}
\Pr\left[\supp(\textbf{a}) \not\subset S \wedge |S| \geq c_n - r\right] \leq o(1) + \frac{d(n - n(1 -\frac{1}{d^2}) + r)}{n} = o(1) + \frac{1}{d} + \frac{dr}{n} = o(1) + \frac{dr}{n}.   
\end{equation}

Combined with eq.~\ref{markov_core}, the contribution from this equation yields the third term in the right hand side of eq.~\ref{eq:lem_core}.

We will break up the second term in eq.~\ref{case_3_core} by conditioning on whether the support of $B$ has many entries or not, and then by using the independence of $\textbf{a}$ from $B$:

\begin{equation}\label{eq:T_term_core}
\Pr[\textbf{a}^\top B\textbf{a} = 0 \land |S| \geq c_n - r] \leq \Pr\left[B \notin \mathcal{T}^{c_n}_{c_n - r - u, s}  \wedge |S| \geq c_n - r\right] + \max_{X \in \mathcal{T}^{c_n}_{c_n - r - u, s}}\Pr\left[\textbf{a}^\top X\textbf{a} = 0\right].
\end{equation}

For the second probability on the right hand side, notice that for any positive integers $a \leq b$, if $\textbf{x}$ and $\textbf{y}$ are random vectors from some product distributions $P^{\otimes a}$ and $P^{\otimes b}$ respectively, then
\begin{equation*}
\max_{X \in \mathcal{T}^{a}_{a - r - u, s}}\Pr\left[\textbf{x}^\top X\textbf{x} = 0\right] \leq \max_{X \in \mathcal{T}^{b}_{b - r - u, s}}\Pr\left[\textbf{y}^\top X\textbf{y} = 0\right].
\end{equation*}

Hence since $c_n \geq n\left(1 - \frac{1}{d^2}\right)$ with probability $1 - o(1)$, we have
\begin{equation}\label{c_n_quad}
\max_{X \in \mathcal{T}^{c_n}_{c_n - r - u, s}}\Pr\left[\textbf{a}^\top X\textbf{a} = 0\right] \leq o(1) + \max_{X \in \mathcal{T}^{n}_{n - n/d^2 - r - u, s}}\Pr\left[A_n^\top XA_n = 0\right].
\end{equation}

To bound the first probability on the right hand side of eq.~\ref{eq:T_term_core}, observe that if $|S| \geq c_n - r$ and $B \notin \mathcal{T}^{c_n}_{c_n - r - u, s}$, there must exist at least $u$ different $j \in S$ such that $\supp(w_j) \leq s$. So
\begin{equation*}
\begin{split}
\Pr\left[B \notin \mathcal{T}^{c_n}_{c_n - r - u, s}  \wedge |S| \geq c_n - r\right] &\leq \Pr\left[|\{j: \exists x \neq 0: A^{(n)}(k)x = e_j, |\supp(x)| < s\}| \geq u \right] \\ &\leq \frac{n}{u}\Pr\left[\exists x \neq 0: A^{(n)}(k)x = e_1, |\supp(x)| < s\right],
\end{split}
\end{equation*}
where the last inequality follows by Markov's inequality. Plugging this and eq.~\ref{c_n_quad} into eq.~\ref{eq:T_term_core} yields
\begin{equation*}
\begin{split}
\Pr[\textbf{a}^\top B\textbf{a} = 0 \land |S| \geq c_n - r] &\leq \frac{n}{u}\Pr\left[\exists x \neq 0: A^{(n)}(k)x = e_1, |\supp(x)| < s\right]\\
& \qquad + \max_{X \in \mathcal{T}^{n}_{n - n/d^2 - r - u, s}}\Pr\left[A_n^\top XA_n = 0\right] + o(1).   
\end{split}
\end{equation*}

Combining this with \ref{case_3_a_S_core} and \ref{case_3_core} yields

\begin{equation*}
\Pr[\mathcal{E}_n \land k \leq r] \leq o(1) + \frac{dr}{n} + \max_{X \in \mathcal{T}^{n - 1}_{n - n/d^2 - r - u, s}}\Pr\left[A_n^\top XA_n = 0\right] + \frac{n}{u}\Pr\left[\exists x \neq 0: A^{(n)}(k)x = e_1, |\supp(x)| < s\right].
\end{equation*}

Plugging this and eqs.~\ref{case_1_core} and \ref{case_2_core} into eq.~\ref{cases_core} and finally eq.~\ref{markov_core} yields the lemma.
\end{proof}

\subsection{Proof of Theorem~\ref{thm:core}}
We are now ready to prove our main theorem, which we restate here.
\thmcore*

\begin{proof}[Proof of Theorem~\ref{thm:core}]
The proof follows by instantiating Lemma~\ref{lemma2core} with the following values:
$t = \frac{n}{C}$, $s = \frac{n}{C}$, $r = \frac{n}{8e^4d^2}$, $u = \frac{n}{2}$. Here $C$ is the constant from Lemma~\ref{lemma:large}.

With the following values, it is immediate from the sparse Littlewood-Offord (Lemma~\ref{lo_sparse}) and quadratic Littlewood-Offord theorems (Lemma~\ref{lo_quadratic}) that the first and last terms in Lemma~\ref{lemma2core} respectively are $o(1)$. The third term is $O(1/d) = o(1)$. The two lemmas stated immediately after this proof (Lemmas \ref{lem:mediumcore} and \ref{lem:basiscore}) will show that the second and fourth terms are $o(1)$.

Assuming these two lemmas, it follows that with probability $1 - o(1)$, there are no kernel vectors with support size at least $n/C$ in $A(k)$. To rule out any minimal dependencies of less than $n/C$ columns of $A(k)$, we use Lemmas~\ref{lemma:small}, \ref{lemma:medium_core}, and \ref{lemma:large}: If there were a minimal dependency among $\ell$ columns $S$ of $A(k)$, then it would mean that the matrix $A_S^R$, with $R = C_k(A)$ would:
\begin{enumerate}
    \item Contain $\geq 3\ell$ $1$'s, because $A(k)$ is a $k$-core for $k \geq 3$ and
    \item Be a minimal dependency.
\end{enumerate}
Thus the event $\mathcal{E}_{S, R}$ of Lemma~\ref{lemma:small} occurs for some $S, R$, which occurs with probability $o(1)$. 

If $ \frac{n}{d^7}< \ell < n/C$, then applying Lemmas \ref{lemma:medium_core}, and \ref{lemma:large} with $T = C_k(A)$ means that, if $|C_k(A)| \geq n(1 - 1/d^7)$, with probability at least $1 - e^{-\Theta(n)}$, there is no dependency of these sizes in $A(k)$. By Lemma~\ref{lemma:core_size}, $|C_k(A)| \geq n(1 - 1/d^7)$ with probability $1 - o(1)$.

This rules out all dependencies in $A(k)$ with probably $1 - o(1)$, proving the theorem.
\end{proof}

It remains to prove lemmas \ref{lem:mediumcore} and \ref{lem:basiscore}.
\begin{lemma}\label{lem:mediumcore}
Let $A \sim \mathbb{A}(n, d/n)$ for $d = \omega(1)$, and let $C$ be the constant in Lemma~\ref{lemma:large}. For any constant $k \geq 3$,
\begin{equation*}
\Pr\left[\exists x: A(k)x = 0, \frac{n}{8e^4d^2} < |\supp(x)| < \frac{n}{C} \right] = o(1).
\end{equation*}
\end{lemma}

\begin{proof}
This is immediate from the medium and large case Lemmas~\ref{lemma:medium_core} and \ref{lemma:large}, since the $k$-core has size at least $n\left(1 - \frac{1}{d^7}\right)$ with probability $1 - o(1)$. Hence if there was a dependency of $\ell$ columns $S$ in the $k$-core for $\frac{n}{8e^4d^2}< \ell < \frac{n}{C}$, it would mean there would have to be a set $R = C_k(A)$ of size at least $n\left(1 - \frac{1}{d^7}\right)$ with $S \subset R$ such that matrix $A_S^R$ is a minimal dependency. This implied that the even $\mathcal{E}_{S, R}$ of Lemma~\ref{lemma:medium_core} or \ref{lemma:large} occurs.
\end{proof}

\begin{lemma}\label{lem:basiscore}
Let $A \sim \mathbb{A}(n, d/n)$ for $d = \omega(1)$, and let $C$ be the constant in Lemma~\ref{lemma:large}. Then for any constant $k \geq 3$,
\begin{equation*}
\Pr\left[\exists x: A(k)x = e_1, |\supp(x)| < \frac{n}{C} \right] = o(1)
\end{equation*}
\end{lemma}

We reduce Lemma~\ref{lem:basiscore} to Lemma~\ref{lem:smallcore}, which will be easier to prove using Lemmas~\ref{lemma:medium_core} and \ref{lemma:large}.

\begin{lemma}\label{lem:smallcore}
Let $A \sim \mathbb{A}(n, d/n)$ for $d = \omega(1)$, and let $C$ be the constant in Lemma~\ref{lemma:large}. Let $A'$ be $A(k)$ with the first row removed.
Then for any constant $k \geq 3$,
\begin{equation*}
\Pr\left[\exists x: A'x = 0, |\supp(x)| \leq \frac{n}{C}\right] = o(1).
\end{equation*}
\end{lemma}

\begin{proof}[Proof of Lemma~\ref{lem:basiscore} assuming Lemma~\ref{lem:smallcore}]
If such a vector $x$ satisfied $A(k)x = e_1$, then necessarily $A'x' = 0$, where $x'$ is $x$ with the first (zero) coordinate removed, and $A'$ is $A(k)$ with the first row removed. This occurs with probability $o(1)$ by Lemma~\ref{lem:smallcore}.
\end{proof}

\begin{proof}[Proof of Lemma~\ref{lem:smallcore}]
By Lemma~\ref{lemma:large}, plugging in $R = C_k(A) \setminus 1$, with probability $1 - o(1)$, there are no vectors $x$ such that $A^Rx = 0$ where $\frac{2n}{d} \leq |\supp(x)| < n/C$. Since $A' = A_R^R$, this implies that there there are no vectors $x$ such that $A'x = 0$ where $\frac{2n}{d} \leq |\supp(x)| < n/C$.

Now suppose there was some set $S$ of size $k \leq \frac{2n}{d}$ such that $A'_S$ is a minimal dependency. Let $R = C_k(A)$, such that $A_S^{R \setminus \{1\}}$ is a minimal dependency. It follows that:
\begin{enumerate}
    \item $S \subset R$;
    \item $\supp(A_S^R) \geq 3|S|$ by definition of the $k$-core, since $k \geq 3$.
    \item $A_S^{R \setminus \{1\}}$ contains no rows with a single $1$, and hence  $A_S^{R}$ contains at most one row with a single $1$. 
\end{enumerate}
This implies that the event $\mathcal{E}_{S, R}$ of either Lemma~\ref{lemma:small} or Lemmas~\ref{lemma:medium_core} occurs. Such an event occurs with probability $o(1)$. This concludes the lemma.
\end{proof}

\section{Application to Gradient Codes}\label{sec:gc}

In the following section, we will shift our focus to the problem of Gradient Coding. The purpose of this section is to present further applications of our structural characterization beyond analyzing the rank of sparse random matrices.

\subsection{Introduction to Gradient Coding}

\begin{center}
\begin{table}\caption{Comparison of Related Work. We have normalized the decoding error by $1/n$.}\label{comp_table}
\begin{tabular}{ |c|c|c| }
 \hline
Assignment Matrix Design & Expected Decoding Error & Adversarial Decoding Error\\  \hline \hline 
 Expander Code (Cor. 23 \cite{ExpanderCode}) &  -  & $< \frac{4p}{d(1 - p)}$ \\ \hline
 Pairwise Balanced (\cite{PB}) & $\geq \frac{p}{d(1 - p)}$ & - \\
\hline
BIBD ( Const. 1 \cite{bibd}) & - & \begin{minipage}{1in} \begin{center} $O(\frac{1}{\sqrt{m}})$ \\ \footnotesize $ d = \Omega(\sqrt{m})$ \end{center} \end{minipage} \\
 \hline
 BRC (\cite{limits}) &
\begin{minipage}{1.5in}\begin{center}
$p^{\Theta(d)}$ \\ \footnotesize $d = \Omega\left(-\frac{\log\log(m)}{\log(p)}\right)$\end{center}\end{minipage}& - \\
\hline
rBGC (\cite{Charles})  & $< \frac{1}{(1 - p)d}$ & - \\
\hline
FRC of \cite{Tandon} (and \cite{ErasureHead}) & $p^{d}$ & $p$ \\ 
\hline
Graph-based Assignment \cite{GW}  & $p^{d - o(d)}$ & $\frac{(1 + o(1))p}{2(1 - p)}$  \\
\hline
{\bf{This work (Stacked ABC)}} & $p^{d - o(d)}$ & $\Theta(p/d)$ \\
\hline
\end{tabular}
\end{table}
\end{center}

Gradient codes are data replication schemes used in distributed computing to provide robustness against \em stragglers\em, machines that are slow or unresponsive \cite{Tandon, ExpanderCode}. Specifically, such replication schemes can be used to approximately compute the sum $\sum_i f(X_i)$ of the evaluations of a function $f: \mathcal{X} \rightarrow \mathbb{R}$ over a set of data points $X_1, \ldots, X_n \in \mathcal{X}$. Typically in machine learning, the function $f$ represents the gradient of a loss function. 

To distribute the computation of $\sum_i f(X_i)$ while maintaining robustness against stragglers, the data points $\{X_i\}$ are distributed redundantly among $m$ machines according to an assignment matrix $A_0 \in \mathbb{R}^{n \times m}$. Each machine $j$ is tasked with computing $g_j := \sum_i (A_0)_{ij}f(X_i)$, a weighted sum of $f$ over the data it stores. A central server then collects these values $g_j$ for all machines $j$ which are not stragglers, and linearly combines the $g_j$ using weights $w_j$ to best approximate $\sum_i f_i(X_i)$.

Formally, we define the \em decoding error \em of an assignment $A_0$ along with a set of stragglers $S \subset [m]$ to be 
the squared distance between the all-$1$'s vector, $\mathbbm{1}$, and the span of the columns $\{(A_0)_j\}_{j \notin S}$:
\begin{equation*}
    \textnormal{err}(A_0, S) := \min_{w: w_j = 0 \: \forall \: j \in S}|A_0w - \mathbbm{1}|_2^2.
\end{equation*}

The decoding error immediately gives us a bound on the approximation error of $\sum_i f(X_i)$. Let $f^*$ be the vector in $\mathbb{R}^n$, where $f^*_i := f(X_i)$. Then for any $w \in \mathbb{R}^m$,
\begin{equation*}
    \sum_j w_j g_j - \sum_i f(X_i) = f^{*T}A_0w - f^{*T}\mathbbm{1} = f^{*T}(A_0w - \mathbbm{1}).
\end{equation*}
Hence for any $S$, there exists a vector $w \in \mathbb{R}^{[m] \setminus S}$ such that 
\begin{equation*}
    \left(\sum_{j \in [m] \setminus S} w_j g_j - \sum_i f(X_i)\right)^2 \leq \textnormal{err}(A_0, S)|f^*|_2^2.
\end{equation*}

The gradient coding literature considers two models of decoding error. One in which the set of stragglers $S$ is chosen adversarially and one where the set of stragglers $S$ is chosen uniformly at random. In the case where $S$ is an adversarially chosen $p$ fraction of the machines, we consider the \em adversarial decoding error \em of $A_0$:
$$\max_{S \in \binom{[m]}{pm}}\left( \textnormal{err}(A_0, S)\right).$$
If $S$ is a random $p$ fraction of the machines, we consider the \em expected decoding error \em of $A_0$:
$$\mathbb{E}_{S \sim \binom{[m]}{pm}}{ \textnormal{err}(A_0, S)}.$$

A good assignment matrix $A_0$ minimizes the decoding error while keeping the maximum amount of computation per machine --- measured by the maximum column sparsity of $A_0$ --- small. For an $n\times n$ assignment matrix with at most $d$ non-zero entries in each column, optimal lower bounds on both the adversarial and random decoding error are known. In particular, Raviv et al. (\cite{ExpanderCode}) showed that the adversarial decoding error must be at least $\Omega(np/d)$. On the other hand, it is known that the random decoding error is at least $np^d$, and further, this optimal error is achieved by a design known as the Fractional Repetition Code (FRC) \cite{Tandon}. However, despite its optimality under random stragglers, the FRC of \cite{Tandon} only achieves an adversarial decoding error of $np$. Several works  \cite{GW, sbm} have aimed to design assignment matrices that have small adversarial \em and \em expected decoding error, but to our knowledge, no existing assignments have achieved adversarial error that decays in $d$ (at any rate) while simultaneously achieving random error that decays faster than like $1/d$. This leads to the following natural question asked in \cite{GW}:
\begin{question}
Does there exist an assignment matrix $A_0$ which simultaneously has near-optimal random decoding error $np^{d-o(d)}$ and near-optimal adversarial decoding error $\Theta(np/d)$?
\end{question}
We answer this question in the affirmative. In the rest of this section, we construct a novel gradient code called the \em Augmented Biregular Code \em (ABC), and using our structural characterization of linear dependencies, prove that it achieves an expected decoding error of $np^{d - o(d)}$ and adversarial decoding error of $\Theta(np/d)$.  In Table~\ref{comp_table}, we compare the adversarial and expected decoding errors of our work and existing work on gradient coding.

\subsection{The ABC and Stacked-ABC Distributions} 

We use the following process to generate a random matrix $A_0 \in \{0, 1\}^{\gamma n \times n}$ from the distribution $\textnormal{ABC}(n, \gamma, d)$:

\begin{definition}[ABC Distribution]\label{def:ABC}
Define $A_0 \sim \textnormal{ABC}(n, \gamma, d)$ to be sampled in the following way. Sample a random permutation $\rho \sim \mathcal{S}_{\gamma d n}$. Create $\gamma n$ \em row-nodes \em and $n$ \em column-nodes \em and associate to each row-node $d$ half-edges and to each column node $\gamma d$ half-edges. Create a multi-graph $G$ by pairing the $i$ right half-edge to the $\rho(i)$th left half-edge. Given this bipartite graph, take $A_0 \in \{0, 1\}^{\gamma n \times n}$ to be the matrix where $(A_0)_{ij} = 1$ if and only if there is at least one edge from row-node $i$ to column-node $j$.
\end{definition}

We define a second distribution  $\on{ABC}_{\gamma-\on{stacked}}(n,d)$ that results from stacking together several copies of the same transposed ABC matrix $A_0 \sim \textnormal{ABC}(n, \gamma, d)$ to form a square $n \times n$ matrix, pictured in Figure~\ref{fig:stacked}.

\begin{definition}[Stacked ABC Distribution]
For $\gamma,d,n\in\mathbb{Z}^+$ such that $\gamma \mid d$ and $\gamma \mid n$, we define the \em $\gamma$-stacked \em Augmented Biregular Code $B$ to be an $n\times n$ matrix formed by sampling $A_0\sim \textnormal{ABC}(n/\gamma,\gamma,d/\gamma)$ and stacking $\gamma$ identical copies of $A_0^\top$. We will denote the distribution of such matrices as $\on{ABC}_{\gamma-\on{stacked}}(n,d)$.
\end{definition}
\begin{figure}
    \centering
    \includegraphics[width=4cm]{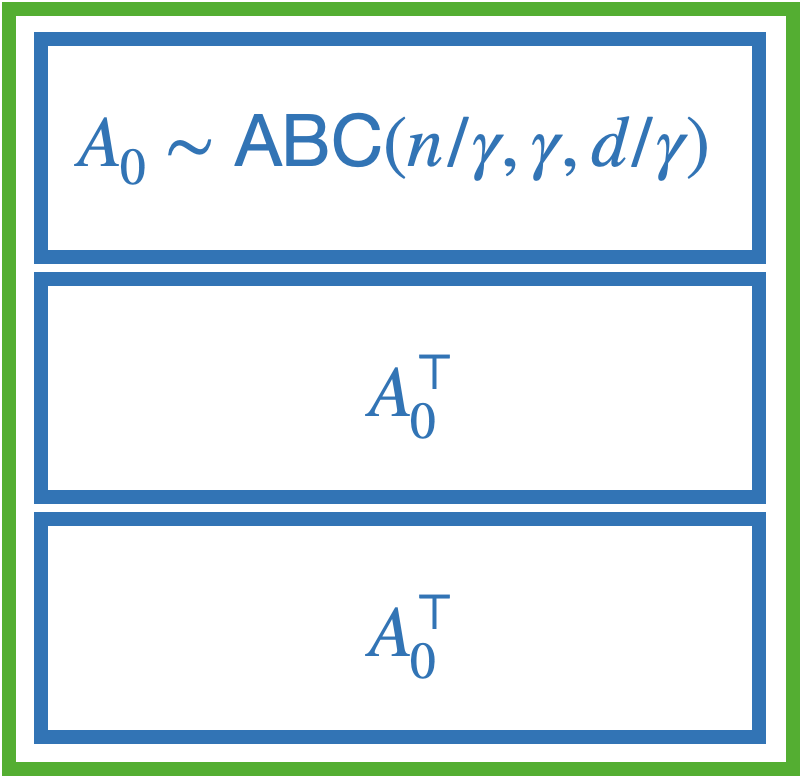}
    \caption{The stacked ABC}
    \label{fig:stacked}
\end{figure}

Using this new distribution of assignment matrices, we prove the following theorem on the decoding error.

\begin{restatable}[]{theorem}{abcstacked}\label{abc_stacked} 
Let $c,\gamma_0,d_0$ be the universal constants from Corollary $\ref{abc_random}$. Choose any $\gamma,d\in\mathbb{Z}^+$ such that $\gamma \geq \gamma_0$, $\gamma \mid d$ and $\frac{d}{\gamma}\geq d_0$. For any $n$ divisible by $\gamma$, let $B \sim \textnormal{ABC}_{\gamma-\on{stacked}}(n, d)$. Then with constant probability over the choice of $B$:
\begin{equation*}
\frac{1}{n}\mathbb{E}_{S \sim \binom{[n]}{pn}}{ \textnormal{err}(B, S)} \leq p^{d - c\log(d)}+o(1),
\end{equation*}
and
\begin{equation*}
    \frac{1}{n}\max_{S \in \binom{[n]}{pn}}\left( \textnormal{err}(B, S)\right) \leq  \left(\frac{8\gamma^3p}{d}\right)+o(1).
\end{equation*}
\end{restatable}

The main technical parts of Theorem~\ref{abc_stacked} follow from the following theorem, which characterizes the minimal dependencies among the columns of the matrix $A_0$ with a random $p$ fraction of rows removed. Formally, we define the ensemble of matrices $A \sim \textnormal{ABC}_p(n, \gamma, d)$ to be given by setting $A = (A_0)^{[n\gamma (1 - p)]}$, where $A_0 \sim \textnormal{ABC}(n, \gamma, d)$.

\begin{restatable}[Characterization ABC]{theorem}{charabc}\label{char:abc}
There exist universal constants $c$ and $\gamma_0$ such that for any $p < 1/2$, $\gamma > \gamma_0$, and $d \geq d_0(\gamma)$, for $A \sim\textnormal{ABC}_p(n, \gamma, d)$, with probability $1 - o(1)$:
\begin{enumerate}
    \item All minimal dependencies of $k$ columns of $A$ are in $\mathcal{T}_k \cup \mathcal{T}_k^{+} \cup \mathcal{T}_k^C$.
    \item The number of columns involved in linear dependencies of $A$ is at most $np^{\gamma d - c\log(\gamma d)}$, ie.,
    $$\left|\bigcup_{x : Ax = 0} \supp(x) \right| \leq np^{\gamma d + c\log_p(\gamma d)}.$$
\end{enumerate}
\end{restatable}
The following corollary allows us to deduce the expected decoding error of $A_0$, leading directly to Theorem~\ref{abc_stacked}.
\begin{corollary}[ABC Distance]\label{abc_random}
Let $A^\top \sim \textnormal{ABC}_p(n, \gamma, d)$. For any $p < 1/2$, there exists constants $c$, $\gamma_0$ and $d_0$ such that for $\gamma \geq \gamma_0$ and $d \geq d_0$, with probability $1 - o(1)$, 
\begin{equation*}
 \frac{1}{n}\min_{w}|A w - \mathbbm{1}|_2^2 \leq p^{\gamma d + c\log_p(\gamma d)}.
\end{equation*}
\end{corollary}
\begin{proof}
Let $D = \bigcup_{x : A^\top x = 0} \supp(x)$ be the set of columns involved in linear dependencies of $A^\top$. Then by Lemma~{\ref{claim:basis}}, for any $i \in [n] \setminus D$, we have $e_i \in \on{span}(A)$. It follows that $\mathbbm{1}_{[n] \setminus D}$, the vector with $0$ in entries indexed by $D$, and $1$ elsewhere is in the span of $A$. Thus $\frac{1}{n}\min_{w}|A w - \mathbbm{1}|_2^2 \leq \frac{|D|}{n}$. Plugging in Theorem~\ref{char:abc} which give a high probability bound on the size of $D$, yields the corollary.  
\end{proof}
\begin{remark}
Corollary~\ref{abc_random} it tight in the sense that with probably $1 - o(1)$, $\frac{1}{n}\min_{w}|A w - \mathbbm{1}|_2^2 \geq (1 + o(1))p^{\gamma d}$. It is easy to check this by counting the number of all-zero rows in $A$.
\end{remark}

\subsection{Proof Overview}
As with the main results of this paper, we split the proof of Theorem~\ref{char:abc} into two domains: small and large. The goal of the small case is to prove that all small minimal dependencies of columns of $A$ must be contained in $\mathcal{T}_k \cup \mathcal{T}_k^{+} \cup \mathcal{T}_k^C$ with high probability. In particular, we prove the following lemma:

\begin{restatable}{lemma}{abcsmall}\label{lemma:abc_small}
Let $A \sim \textnormal{ABC}_p(n, \gamma, d)$ for $\gamma \geq 1$ and $p < \frac{1}{2}$. Let $S\subset[n]$ be any set of size $k\in[1,\frac{n}{18e\gamma^2 d^2}]$. There exists universal constants $c_{\ref{lemma:abc_small}}$ and $d_0$ such that if $d>d_0$, then:
\begin{enumerate}
    \item $\Pr[A_S \in \mathcal{M}_k \setminus \left(\mathcal{T}_k \cup \mathcal{T}_k^+ \cup \mathcal{T}_k^C \right)] = O\left(e^{-k}\left(\frac{k}{n}\right)^{k+1/2}\right)$.
    \item $\Pr[A_S \in \mathcal{T}_k]\leq \left(p^{\gamma d+ c_{\ref{lemma:abc_small}}\log_p(\gamma d)}\right)^k\left(\frac{k}{n}\right)^{k-1}$.
    \item  $\Pr[A_S \in\mathcal{T}_k^+]\leq \left(p^{\gamma d+c_{\ref{lemma:abc_small}}\log_p(\gamma d)}\right)^k\left(\frac{k}{n}\right)^{k}$.
    \item  $\Pr[A_S \in \mathcal{T}_k^C]\leq \left(p^{\gamma d+c_{\ref{lemma:abc_small}}\log_p(\gamma d)}\right)^k\left(\frac{k}{n}\right)^{k}$.
\end{enumerate}
\end{restatable}

As for the previous results, the goal of the large domain will be to show that with high probability, $A$ contains no large set of columns that are linearly dependent.

\begin{restatable}[ABC Large Case]{lemma}{abclarge}\label{abc_large}
Let $A \sim \textnormal{ABC}_p(n, \gamma, d)$ for some constant $\gamma \geq 16$. Then there exists a constant $d_0(\gamma)$ such that for $d \geq d_0$, 
\begin{equation*}
    \Pr\left[\exists x: Ax = 0, |\supp(x)| \geq \frac{n}{18e\gamma^2 d^2}\right] \leq o(1).
\end{equation*}
\end{restatable}

Our main tool in proving this Lemma is an anti-concentration Lemma that is based on the sparse Littlewood-Offord Theorem from \cite{cv_sparse}. While typically such anti-concentration lemmas concern the dot product of a deterministic vector and a random vector with independent entries, we derive a weaker result which concerns the dot product of a deterministic vector with a vector in which there are a fixed number of non-zero entries whose positions are random. 

\begin{restatable}[Anti-concentration for Sparse Regular Vectors]{lemma}{acslice}\label{lo_reg2}
Let $v\in \mathbb{R}^N$ be an arbitrary vector whose most common entry is $a$. Then for any $d \leq \sqrt{\frac{N}{2}}$, if $x \in \{0, 1\}^N$ is sampled uniformly from the set of vectors with exactly $d$ 1s, for any $w\in \mathbb{R}\backslash\{da\}$, we have: $$\Pr\left[x\cdot v=w\right]\leq 1/2+\frac{d^2}{N}.$$
\end{restatable}
\begin{remark}
As $d \rightarrow \infty$, the anti-concentration probability approaches the smaller value of $1/e$, though $1/2$ is tight for $d = 2$.
\end{remark}

\subsection{Notation and Further Details on Construction of $A \sim \textnormal{ABC}_p(n, \gamma, d)$}
\begin{figure}
    \centering
        \centering
        \begin{tabular}{ccc}
                 \includegraphics[width=3cm]{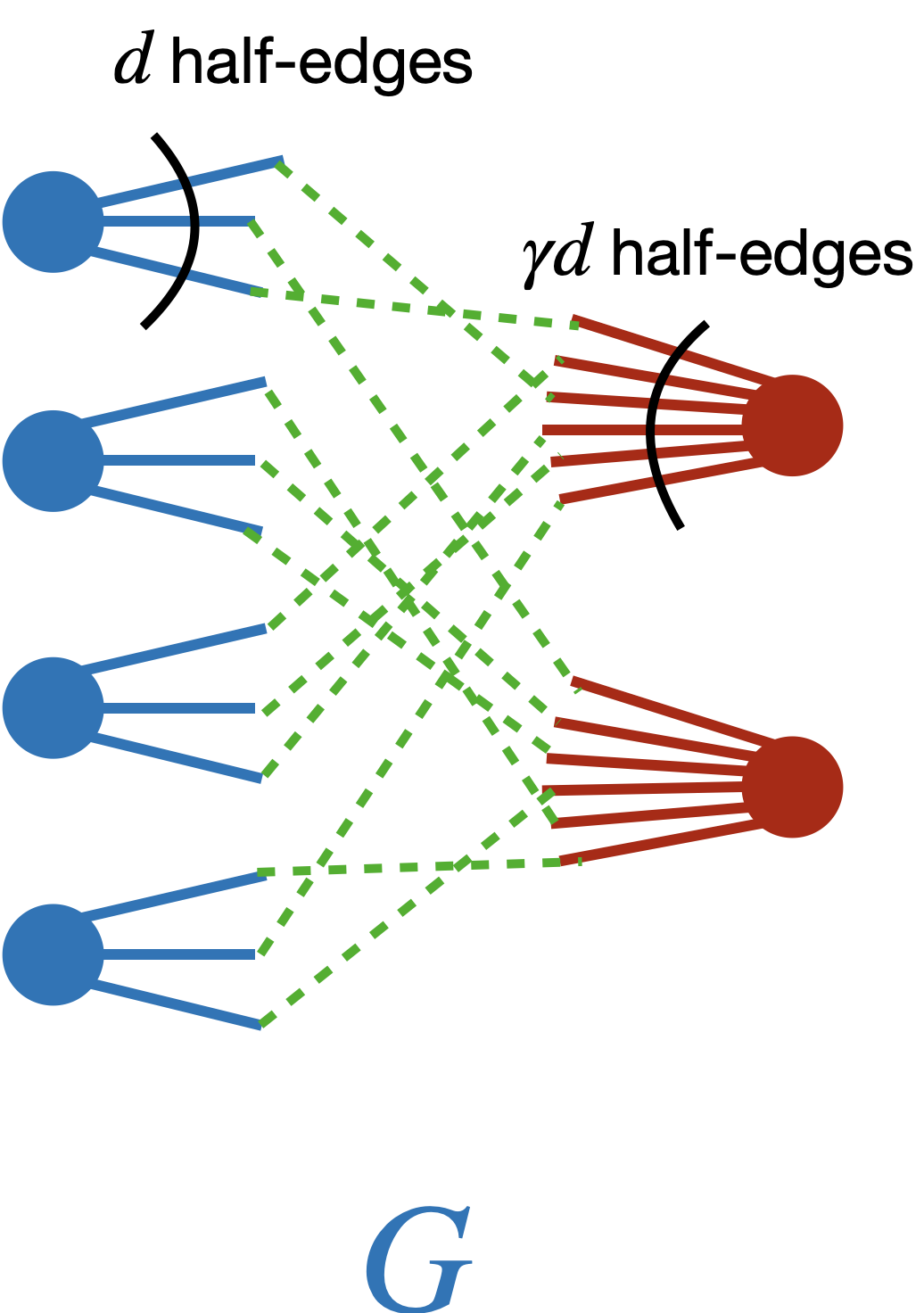} & \includegraphics[width=4cm]{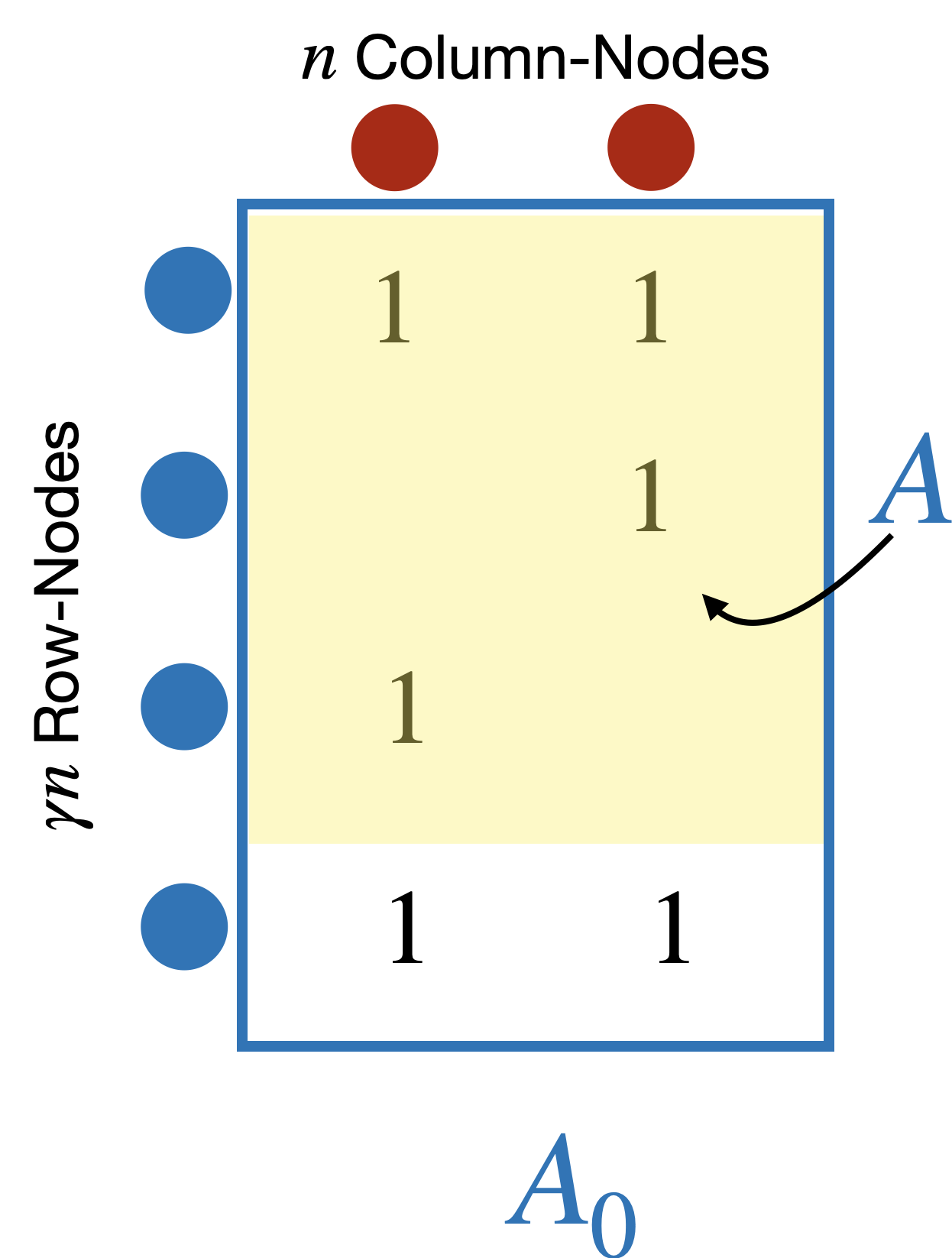}  & \includegraphics[width=6cm]{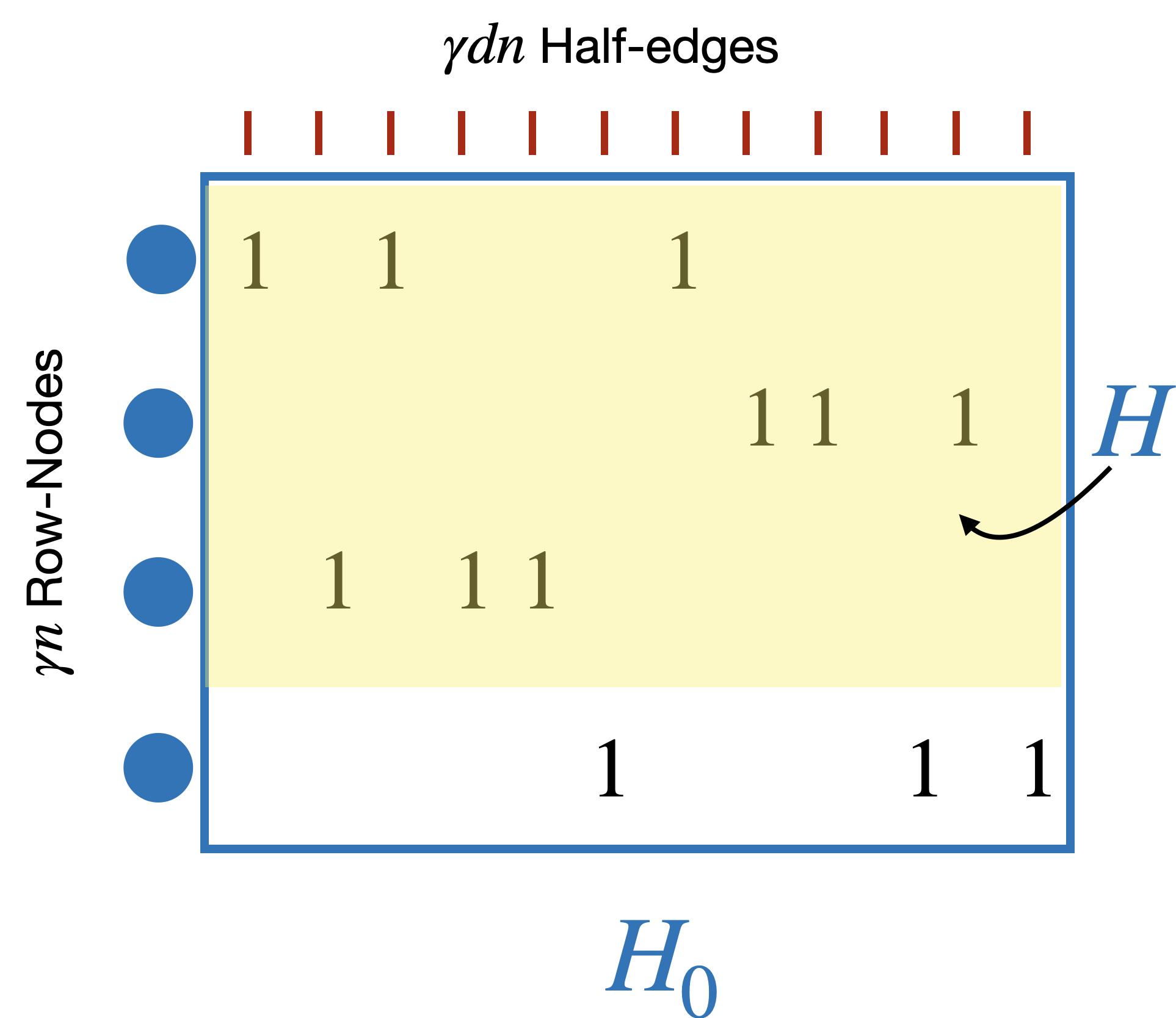} 
        \end{tabular}
    \caption{The Construction of $A_0$ and $H_0$. In the pictured example, $n = 2$, $\gamma = 2$, $d = 3$.}
    \label{fig:occupancy}
\end{figure}
In our proofs, we will study $A$ via the half-edge pairing process used to construction $A_0$ in Definition~\ref{def:ABC}. To do this, we will sometimes explicitly consider the permutation $\rho \in \mathcal{S}_{\gamma d n}$ which defines the random pairing of half-edges. From $\rho$, we can can construct an additional matrix $H_0 = H_0(\rho)$, which we call the \em column half-edge occupancy matrix \em. Let $P \in \{0, 1\}^{\gamma d n \times \gamma d n}$ be the permutation matrix of $\rho^{-1}$. From $P$, construct $H_0 \in \{0, 1\}^{\gamma n \times \gamma d n}$ by summing adjacent rows of $P$ corresponding to half-edges of the same row nodes. Symbolically,
$$(H_0)_{ij} = \sum_{k=(i-1)d + 1}^{id} P_{kj}.$$

Note that $H_0$ has $\gamma n$ rows; row $i$ represents row-node $i$ in the configuration model. $H_0$ has $\gamma d n$ columns: columns $(j-1)\gamma d + 1$ through $j\gamma d$ represent the $\gamma d$ half edges of column-node $j$, for $1 \leq j \leq n$. Each row contains exactly $d$ $1$'s, corresponding to that row-node's half edges.

We may write $A_0 = A_0(\rho)$ in terms of the random matrix $H_0$. Because $(A_0)_{ij} = 1$ if and only if there is at least one edge from vertex $i$ to $j$, we have $$(A_0)_{ij} = \mathbbm{1}(\exists k, d(j-1) < k \leq dj, H_{ik} = 1 ).$$

Recall that when generating a matrix $A \sim \textnormal{ABC}_p(n, \gamma, d)$, we sample $A_0 \sim \textnormal{ABC}(n, \gamma, d)$ and take $A := (A_0)^{[\gamma n(1 - p)]}$ to be the matrix containing the first $\gamma n(1 - p)$ rows of $A_0$. We define $H := H_0^{[\gamma n(1 - p)]}$ to be the first $\gamma n(1-p)$ rows of the row half-edge occupancy matrix associated with $A_0$. For a set $S \subset [n]$ with $|S| = k$, let $H(S) \in \mathbb{R}^{\gamma n(1-p) \times \gamma d k}$ be the matrix $H$ restricted to the $\gamma d k$ columns corresponding to half-edges of the column-nodes in $S$.

\subsubsection{Further Notation}
We use $\text{HypGeo}(N, K, n)$ to denote the hypergeometric distribution, the random variable which equals the number of successes in $n$ draws, without replacement, from a finite population of size $N$ that contains exactly $K$ objects which are considered successes.

\subsection{Small Domain}
As noted above, the goal of this section will be to prove the following lemma:

\abcsmall*

\begin{proof}[Proof of Lemma~\ref{lemma:abc_small}]
We condition on $L_S$, the number of $1$'s in $H(S)$. We observe that $L_S\sim \text{HypGeo}(\gamma d n, \gamma d(1-p)n, \gamma dk)$. Note that $L_S$ is greater than or equal to the number of entries which are $1$ in the submatrix $A_S$, with equality if and only if there does not exist a column-node in $S$ which pairs multiple half-edges to the same row-node in $[\gamma n (1 - p)]$. 

Let $\mathcal{E}_S$ denote the event that there are no rows in $H(S)$ with exactly $1$ one. It follows from Observation~\ref{observation1} that
$$\Pr[A_S\in \mathcal{S}_{\ell,k}]\leq \Pr[\mathcal{E}_S |L_S=\ell]\cdot\Pr[L_S=\ell].$$

Our general strategy to bound $\Pr[\mathcal{E}_S|L_S = \ell]$ relies on the following claim. 

\begin{claim}\label{claim:X}
Let $X$ be the number 1s in $H(S)$ which are \em not \em the first (leftmost) 1 in their row. If $\mathcal{E}_S$ occurs, then $X \geq L_S/2$. 
\end{claim}
\begin{proof}
If $X < L_S/2$, then by the pigeonhole principle, there must be at least one row in $H(S)$ with a leftmost 1 but no 1s to the right of it, ie. this row has exactly one 1.
\end{proof}

We will bound the probability that $X$ is large by considering a random walk which this quantity as half-edges are paired one at time. We formalize this random walk as follows.

Conditioned on $L_S$, the matrix $H(S) \in \{0, 1\}^{\gamma n(1-p)\times \gamma d k}$ is distributed like a uniformly random matrix on the set of all matrices in $\{0, 1\}^{\gamma n(1-p)\times \gamma d k}$ with exactly $L_S$ $1$'s and at most one $1$ per column. We construct $H(S)$ via the following random process $M^0, M^1, \ldots , M^{L_S} = H(S)$ on matrices in $\{0, 1\}^{\gamma d k \times \gamma n(1-p)}$, in which the $L_S$ half-edges represented in $H(S)$ are paired one at a time. Formally, at each step $i$, we construct $M^i$ from $M^{i - 1}$ by placing a $1$ in a uniformly random location in $M^{i - 1}$ without a $1$. For $i \in 0, 1, \ldots L_S$, let $X_i$ equal the number of 1s in $M^i$ which are \em not \em the first (leftmost) 1 in their row, such that $X_{L_S} = X$.

When placing the $i$-th $1$, there are at most $i-1$ non-zero rows so far. In particular, since $i \leq L_S$, there are at most $L_S$ non-zero rows throughout this process. Because $H(S)$ has $\gamma(1-p)n$ rows, we thus have the bound $\Pr[X_i - X_{i - 1} = 1] \leq \frac{L_S}{\gamma(1-p)n}$.
 It follows that the random variable $X$ is first-order stochastically dominated by the random variable $\sum_{i=1}^{L_S} \textnormal{Bernoulli}(\frac{2L_S}{n})$. So in particular, we have:
\begin{align*}
    \Pr\left[X\geq \lfloor L_S/2\rfloor \right] &\leq \Pr\left[\Bin\left(L_S, \frac{2L_S}{n}\right) \geq L_S/2 \right],
\end{align*}
which implies
\begin{align*}
    \Pr[A_S\in \mathcal{S}_{\ell, k}]&\leq \Pr\left[\Bin\left(\ell, \frac{2\ell}{n}\right) \geq \ell/2 \right]\cdot\Pr[ L_S=\ell].
\end{align*}
We will use the following claim proved in Appendix~\ref{apx:abc}:
\begin{restatable}{claim}{claimabc}\label{claim:abc1}
Let $p<1/2$. There exists constants $c_{\ref{claim:abc1}}$, and $q_0$ such that for all $q \geq q_0$, and $K\leq \min\left(\frac{3}{2}pN, \frac{N}{18eq^2}\right)$, the following two bounds hold.

For $\ell\in \{2K-2,2K-1,2K\}$, and $j \geq \ell/2$, we have:
\begin{align*}
    \Pr\left[\Bin\left(\ell, \frac{2\ell}{N}\right)\geq j\right] \Pr[\textnormal{HypGeo}(q N,q(1-p)N,q K)=\ell]\leq \left(p^{q - c_{\ref{claim:abc1}}\log_p(q)}\right)^K \left(\frac{K}{N}\right)^{j}.
\end{align*}
Further,
\begin{align*}
    \sum_{\ell=2K+1}^{qK}\Pr\left[\Bin\left(\ell, \frac{2\ell}{N}\right)\geq \frac{\ell}{2}\right]\cdot\Pr[\textnormal{HypGeo}(qN,q(1-p)N,qK)=\ell]\leq e^{-k}\left(\frac{K}{N}\right)^{K+1/2}.
\end{align*}


\end{restatable}
Since $L_S \sim \textnormal{HypGeo}(\gamma d n, \gamma d(1-p)n, \gamma d k)$, it must be the case that $L_S \geq \gamma d k - p\gamma dn.$ Hence if $L_S \in \{2k-2,2k-1,2k\}$, then necessarily $k \leq \frac{p\gamma d n}{\gamma d - 2} \leq \frac{3}{2}pn,$
for $\gamma d \geq 6$, which is guaranteed if $d \geq 6$. Thus we can apply the first bound in Claim~\ref{claim:abc1} with $K = k$, $N = n$ and $q = \gamma d$ to study $\Pr[A_S\in \mathcal{S}_{\ell, k}] \leq \Pr\left[X \geq \ell/2]\cdot\Pr[ L_S=\ell\right]$, yielding:
\begin{align*}
    \Pr[A_S\in \mathcal{S}_{2k-2,k} ]&=\left(p^{\gamma d-c_{\ref{claim:abc1}}\log_p(\gamma d)}\right)^k \left(\frac{k}{n}\right)^{k-1}.\\
    \Pr[A_S\in \mathcal{S}_{2k-1,k} ]&=\left(p^{\gamma d- c_{\ref{claim:abc1}}\log_p(\gamma d)}\right)^k \left(\frac{k}{n}\right)^{k}.\\
\Pr[A_S\in \mathcal{S}_{2k,k}']\leq \Pr[A_S\in \mathcal{S}_{2k,k} ]&=\left(p^{\gamma d-c_{\ref{claim:abc1}}\log_p(\gamma d)}\right)^k \left(\frac{k}{n}\right)^{k}.
\end{align*}

We use the equivalences of $\mathcal{S}_{2k - 2, k} = \mathcal{T}_k$, $\mathcal{S}_{2k - 1, k} = \mathcal{T}_k^+$, and $\mathcal{S}_{2k, k}' = \mathcal{T}_k^C$ from Lemma~\ref{lemma:classification} to yield statements 2, 3, and 4 in this lemma.

Next we obtain a bound on the event that $A_S \in \mathcal{M}_k\backslash (\mathcal{T}_{k}\cup \mathcal{T}_{k}^+\cup \mathcal{T}_{k}^C)$. Observe that 
\begin{equation*}
A_S \in \mathcal{M}_k\backslash (\mathcal{T}_{k}\cup \mathcal{T}_{k}^+\cup \mathcal{T}_{k}^C) \Rightarrow A_S \in \left(\mathcal{S}_{2k,k}\setminus \mathcal{S}_{2k,k}'\right) \cup \bigcup_{\ell \geq 2k + 1}\mathcal{S}_{\ell, k}.
\end{equation*}

If $A_S\in \mathcal{S}_{2k,k}\setminus \mathcal{S}_{2k,k}'$, there must be a column with at least three $1$'s. So by a similar argument to Claim~\ref{claim:X}, it must be the case that $X \geq k + 1$. Hence
\begin{align*}
    \Pr[A_S\in \mathcal{S}_{2k,k}\setminus \mathcal{S}_{2k,k}']&\leq \Pr\left[\Bin\left(2k, \frac{4k}{n}\right) \geq k+1\right]\cdot\Pr[L_S=2k].
\end{align*}
Another application of the Claim~\ref{claim:abc1} gives us: 
\begin{align*} 
    \Pr[A_S\in \mathcal{S}_{2k,k}\setminus \mathcal{S}_{2k, k}']&\leq \left(p^{\gamma d+c_{\ref{claim:abc1}}\log_p(\gamma d)}\right)^k \left(\frac{k}{n}\right)^{k+1}.
\end{align*}

It follows that
\begin{align*}
    \Pr[A_S\in \mathcal{M}_k\backslash (\mathcal{T}_{k}\cup \mathcal{T}_{k}^+\cup \mathcal{T}_{k}^C)]&\leq \left(p^{\gamma d+c_{\ref{claim:abc1}}\log_p(\gamma d)}\right) \left(\frac{k}{n}\right)^{k+1}+\sum_{\ell=2k+1}^{\gamma d k}\Pr[A_S\in \mathcal{S}_{\ell, k}]\\
    &\leq \left(p^{\gamma d+c_{\ref{claim:abc1}}\log_p(\gamma d)}\right) \left(\frac{k}{n}\right)^{k+1} +\sum_{\ell=2k+1}^{\gamma d k}\Pr\left[\Bin\left(\ell, \frac{2\ell}{n}\right) \geq \frac{\ell}{2}\right]\cdot\Pr[L_S = \ell].
\end{align*}
To bound the final term, we use the second bound in Lemma~\ref{claim:abc1}.


This gives us the desired result for sufficiently large $\gamma$ and $d$:
\begin{align*}
    \Pr[A_S\in \mathcal{M}_k\backslash (\mathcal{T}_{k}\cup \mathcal{T}_{k}^+\cup \mathcal{T}_{k}^C)]&\leq \left(p^{\gamma d+c_{\ref{claim:abc1}}\log_p(\gamma d)}\right)^k \left(\frac{k}{n}\right)^{k+1}+ e^{-k}\left(\frac{k}{n}\right)^{k+1/2} = O\left(e^{-k}\left(\frac{k}{n}\right)^{k+1/2}\right).
\end{align*}
This yields the first statement in the lemma.
\end{proof}

\subsection{Large Case}
The main goal of the large case is to prove the following lemma. 

\abclarge*

We begin by proving the following anti-concentration Lemma.

\acslice*
\begin{proof}
For $i \in \{1,...,d\}$, let $j_i \sim \text{Uniform}([N])$, and let $X_i := v_{j_i}$.
Then
\begin{equation}\label{eq:anti}
    \Pr[x\cdot v=w] = \Pr\left[\sum_{i=1}^d X_i =w\ | \text{All $j_i$ are unique}\right] \leq \frac{\Pr\left[\sum_{i=1}^d X_i =w\right]}{\Pr\left[\text{All $j_i$ are unique}\right]}.
\end{equation}

\begin{claim}
\begin{equation*}
\Pr\left[\text{All $j_i$ are unique}\right] \geq 1 - \frac{d^2}{N}.
\end{equation*}
\end{claim}
\begin{proof}
By a union bound, 
\begin{align*}
    \Pr\left[\exists i \neq \ell: j_i = j_{\ell} \right] &\leq \binom{d}{2}\cdot \Pr\left[j_1 = j_2\right]=\binom{d}{2}\left(\frac{1}{N}\right) \leq \frac{d^2}{N}.
\end{align*}
\end{proof}

Next we show by induction on $d$ that $\Pr\left[\sum_{i=1}^d X_i = w\right] \leq 1/2$ for all $w\in \mathbb{R}\backslash\{da\}$. For $d=1$, we note that the chosen element $w$ cannot be the most common element $a$. Thus, $w$ is at worst equally as common as the most common element $a$. This implies that the number of times $w$ appears in $v$ is at most $\frac{N}{2}$, hence $\Pr[X_1 = w]\leq \frac{1}{2}$. 

Now assume that $\Pr\left[\sum_{i=1}^{d-1} X_i = u\right]\leq 1/2$ holds for all $u\not=(d-1)a$. For $w\not=da$, we write:
\begin{align*}
    \Pr\left[\sum_{i=1}^{d} X_i = w\right] &=\sum_{x\in \text{Supp($X_d$)}} \Pr\left[X_d = x\right]\cdot \Pr\left[\sum_{i=1}^{d-1} X_i = w-x\right]\\
    &=\Pr\left[X_d=a\right]\cdot \Pr\left[\sum_{i=1}^{d-1}X_i=w-a\right] +\sum_{x\in \text{Supp($X_d$)$\backslash\{a\}$}} \Pr\left[X_d = x\right]\cdot \Pr\left[\sum_{i=1}^{d-1} X_i = w-x\right]
\end{align*}
Let $p_x := \Pr\left[\sum_{i=1}^{d-1}X_i=w-x\right]$ such that by the induction hypothesis $p_a \leq 1/2$ as $w-a=(d-1)a$ if and only if $w=da$. Thus, we conclude:
\begin{align*}
    \Pr\left[\sum_{i=1}^{d} X_i = w\right] &\leq \max_{p_a \leq 1/2, \sum_{x}p_x \leq 1}\left(\Pr\left[X_d=a\right]\cdot p_a + \sum_{x\in \text{Supp($X_d$)$\backslash\{a\}$}} \Pr\left[X_d = x\right]\cdot p_x\right) \\
    &\leq \max_{\sum_{x \neq a}p_x \leq 1/2}\left(\Pr\left[X_d=a\right]\cdot \left(1/2\right) + \sum_{x\in \text{Supp($X_d$)$\backslash\{a\}$}} \Pr\left[X_d = x\right]\cdot p_x\right) \\
    &\leq \max_{\sum_{x \neq a}p_x \leq 1/2}\left(\Pr\left[X_d=a\right]\cdot (1/2) + \left(\sum_{x\in \text{Supp($X_d$)$\backslash\{a\}$}} P(X_d = x)\right)\cdot \left(\sum_{x\in \text{Supp($X_d$)$\backslash\{a\}$}}p_x\right)\right)\\
    &= \Pr\left[X_d=a\right]\cdot (1/2)+(1-\Pr[X_d=a])\cdot(1/2) = 1/2.
\end{align*}
Here the second inequality follows from the fact that the optimum over the $p_x$ is achieved by putting the maximum possible mass on $p_a$, that is, $p_a = 1/2$.

Returning to eq.~\ref{eq:anti}, we have for all $w \neq ad$ and $d \leq \sqrt{\frac{N}{2}}$, 
\begin{align*}
    \Pr[x\cdot v=w] &\leq \frac{\Pr\left[\sum_{i=1}^d X_i =w\right]}{\Pr\left[\text{All $j_i$ are unique}\right]} \leq \frac{1/2}{1 - \frac{d^2}{N}} \leq 1/2 + \frac{d^2}{N}, 
\end{align*}
where the final inequality follows from the fact that $\frac{1}{1 - x} \leq 1 + 2x$ for $0 \leq x \leq 1/2$.
\end{proof}

To prove Lemma~\ref{abc_large}, we take a similar approach to the large case for the BGC. For a fixed set $S$ of size $k$, we consider the stochastic process $A_S^{[1]}, A_S^{[2]}, \ldots$ in which we add the rows of $A_S$ one by one. We define the space: $$D(A_S^{[j]})=\Span\left(\{v\in (\mathbb{R}\backslash\{0\})^k: A_S^{[j]}v =0 \}\right),$$
and let $R_j := \text{Rank}(D(A_S^{[j]})).$ If $R_{\gamma n(1-p)} =0$, then $A_S$ is not a minimal dependency. 

Each time we add a new column, $R_j$ either stays constant or decreases by at least $1$ (note that $R_j$ can decrease by more than $1$: if a new row $j$ has exactly one $1$, then $R_j = R_{j+1} = \ldots = 0$ no matter what $R_{j-1}$ was). We will use the following lemma to show for $j \leq \gamma n/2$, each column we add is close to random. Then, using Lemma~\ref{lo_reg2}, we show that with decent probability, $R_j$ decreases. We can then apply a Chernoff Bound to show that $R_{\gamma n/2} = 0$ with high probability. 

More formally, we consider the process of constructing $H \in \{0,1\}^{\gamma n(1-p) \times \gamma d n}$ (the column half-edge occupancy matrix) one row at a time by pairing the $d$ half-edges from each row-node at each step to a random set of $d$ unpaired column-half-edges. 

Define $e(S, j) := \gamma d k - |\supp(H^{[j]}(S))|$, that is the number of unpaired half-edges among the $k$ column-nodes in $S$ after the first $j$ row-nodes have randomly paired their half-edges. Observe that $e(S, j) \sim \textnormal{HypGeo}(\gamma d n, \gamma d k, \gamma d n - dj)$. The following lemma uses standard concentration bounds to show that for the first $\gamma n/2$ columns, there are many unpaired half-edges out of row-nodes in $S$. 

\begin{lemma}\label{lemma:controlling_large}
Let $A \sim \textnormal{ABC}_p(n, \gamma, d)$. Let $\Omega$ be the event in which for any set $S$ with $|S| \geq \frac{n}{18e\gamma^2 d^2}$ and for any $j\in[1,\frac{\gamma}{2}n]$, we have 
 $\frac{e(S,j)}{\gamma dn-dj}\geq \frac{k}{2n}$.
For $\gamma \geq 2$, there exists a constant $d_0(\gamma)$ such that for $d \geq d_0$, we have $\Pr[\Omega] \geq 1 - \gamma ne^{-n/d^3}$.
\end{lemma}
\begin{proof}
For $d$ larger than some constant $d_0(\gamma)$, we have $\frac{n}{d^3} \leq \frac{n}{18e\gamma^2 d^2}$, hence it suffices to prove the result for all $k\in[n/d^3,n)$. Fix $k$ and let us define $\eta_k$ to be the ratio $k/n$. Choose a subset $S$ of size $k$. Furthermore, fix $j\leq \frac{\gamma}{2}n$. We proceed by applying the following tail bound on Hypergeometric distributions to $e(S,j)$ which is clear from \cite{hgtail}.

\begin{lemma}[Hypergeometric Tail Bound~\cite{hgtail}]\label{HGtail}
Let $X \sim \textnormal{HypGeo}(A, B, c)$. Then for all $t < \frac{B}{A}$,
\begin{equation*}
\Pr\left[X \leq tc \right] \leq e^{-cD_{KL}(t||B/A)}.
\end{equation*}
Similarly, letting $Y \sim \textnormal{HypGeo}(A, A - B, c)$, for $t > \frac{B}{A}$, we have
\begin{equation*}
\Pr\left[X \geq tc \right] = \Pr[Y \leq (1 - t)c] \leq e^{-cD_{KL}(1 - t||1 - B/A)} = e^{-cD_{KL}(t||B/A)}.
\end{equation*}
\end{lemma}
Using this bound, we have
\begin{align*}
    \Pr[e(S,j)\leq ({\eta_k-t})(\gamma d n-dj)]&\leq e^{-D_{KL}\left(\eta_k-t||\eta_k\right)(\gamma d n-dj)}\leq e^{-D_{KL}\left(\eta_k-t||\eta_k\right)\frac{\gamma dn}{2}}
\end{align*}
for all $t\in(0,\eta_k)$. The following claims expands this KL-divergence for $t = \eta_k/2$.
\begin{claim}
\begin{equation*}
\begin{split}
D_{KL}\left(\eta_k-\eta_k/2||\eta_k\right) \geq \frac{\eta_k}{12}.
\end{split}
\end{equation*}
\end{claim}
\begin{proof}
\begin{equation*}
\begin{split}
D_{KL}\left(\eta_k/2||\eta_k\right)  &= \frac{\eta_k}{2}\ln(1/2) + \left(1 - \frac{\eta_k}{2}\right)\ln\left(\frac{1 - \eta_k/2}{1 - \eta_k}\right) \\
&\geq \frac{\eta_k}{2}\ln(1/2) + \left(1 - \frac{\eta_k}{2}\right)\left(\frac{\eta_k}{2} + \frac{3\eta_k^2}{8}\right) \\
&= \frac{\eta_k}{2}\left(\ln(1/2) + \left(1 - \frac{\eta_k}{2}\right)\left(1 + \frac{3\eta_k}{4}\right)\right) \geq \frac{\eta_k}{12}.
\end{split}
\end{equation*}
Here the first inequality follows by taylor expanding $
\ln(1 - x)$. The final inequality follows by noting that the quadratic $\ln(1/2) + \left(1 - \frac{\eta_k}{2}\right)\left(1 + \frac{3\eta_k}{4}\right)$ is minimized over $\eta_k \in [0, 1]$ at $\eta_k \in \{0, 1\}$. 
\end{proof}

Union bounding over all $\binom{n}{k}$ sets $S$ of size $k$, we have 
\begin{align*}
     \Pr\left[\exists S, k:=|S| \in [n/d^3, n]: e(S,j)\leq \frac{\eta_k}{2}(\gamma dn-dj)\right]
     &\leq \sum_{k=n/d^3}^n \binom{n}{k}e^{-\frac{\eta_k\gamma dn}{24}}
     \leq \sum_{k=n/d^3}^n \left(\frac{en}{k}\right)^ke^{-\frac{k\gamma d}{24}} \\
     &\leq \sum_{k=n/d^3}^n e^{-k\left(\frac{\gamma d}{24}- \ln(d^3) - 1\right)} 
    \leq \sum_{k=n/d^3}^n e^{-k}
    \leq 2e^{-n/d^3}.
\end{align*}

We take a union bound over $j \in [1, \gamma n/2]$ to achieve
\begin{align*}
     \Pr\left[\exists S, k := |S| \in[n/d^3,n] : e(S,j)\leq \frac{\eta_k}{2}(\gamma dn-dj)\right] \leq \gamma ne^{-n/d^3}.
\end{align*}
\end{proof}

Conditioned on $\Omega$, we can use the anti-concentration bound in Lemma~\ref{lo_reg2} to show that for $j \leq \gamma n/2$, $R_j$ often decreases.

\begin{lemma}\label{lemma:decrease}
For $j \leq \gamma n/2$, conditioned on $\Omega$, if $R_j \geq 1$, then with probability at least $$\mu_k := \frac{1}{2}\left(1 - e^{-\frac{dk}{2n}}\right) - \frac{3\gamma d^3}{k},$$ we have $R_{j} \leq R_{j-1} - 1.$
\end{lemma}
\begin{proof}
Suppose $R_{j - 1} \geq 1$, and let $v$ be any vector in $(\mathbb{R}\backslash\{0\})^k$ such that $A_S^{[j - 1]}v=0$. Let $v' \in \mathbb{R}^{\gamma dk}$ be the vector which repeats each coordinate of $v$ $\gamma d$ times: that is, $v'_i = v_{\lceil \frac{i}{\gamma d}\rceil}$.

Let $h_j \in \{0, 1\}^{\gamma dn}$ be $j$th row of $H$, which has exactly $d$ $1$'s indicating the half-edges matched to the $d$ half-edges from the row-node $j$. Let $h_j(S) \in \{0, 1\}^{\gamma d k}$ be the restriction of $h_j$ to the entries corresponding to half-edges from column-nodes in $S$.

\begin{claim}\label{claim:success}
If at most one half-edge is matched from column-node $j$ to a single row-node in $S$ and $h_j(S) \cdot v' \neq 0$, then $R_{j} \leq R_{j-1} - 1$.
\end{claim}
\begin{proof}
If at most one half-edge is matched from column-node $j$ to a single row-node in $S$, then $$(A_S)^j v =  \sum_{i \in S}A_{ji}v_i = \sum_{i \in S}\mathbbm{1}(\exists \ell, \gamma d(i-1) < \ell \leq \gamma di, H_{j \ell} = 1 )v_i = \sum_{i \in S}v_i\sum_{\ell = \gamma d(i - 1) + 1}^{\gamma di} H_{j\ell} = h_j(S)\cdot v' \neq 0.$$  Hence $v \notin D(A_S^{[j]})$, so the rank of $D(A_S^{[j]})$ is strictly smaller than that of $D(A_S^{[j-1]})$.
\end{proof}

\begin{claim}\label{claim:using_lo}
\begin{equation*}
\Pr[h_j(S) \cdot v' = 0 | \Omega] \leq  \frac{1}{2}\left(1 + e^{-\frac{dk}{2n}}\right) +  \frac{d}{\gamma k}.
\end{equation*}
\end{claim}
\begin{proof}

We condition on the number of $1$'s in $h_j(S)$, which we denote $s$. Observe that conditioned on $s$, the vector $h_j(S)$ is a uniformly random vector from the set of all vectors in $\{0, 1\}^{\gamma d k}$ with $s$ 1s. This holds even when we conditioned on $\Omega$, because this event says nothing about \em which \em half-edges among the nodes in $S$ have been paired. By Lemma~\ref{lo_reg2}, since $v'$ contains no zeros (and hence its most common element is not zero), we have 
\begin{equation}\label{eq:sneq0}
\Pr[h_j(S) \cdot v' = 0 | s \geq 1, \Omega] \leq 1/2 + \frac{d}{\gamma k}.
\end{equation}

It remains to consider the probability that $s = 0$, since in this case, we always have $h_j(S) \cdot v' = 0$. We know that $s$ is distributed like a hypergeometric random variable $\text{HypGeo}(\gamma d n - dj, d, e(S, j))$. Indeed, there are $d$ half-edges that are paired with the addition of the $(j + 1)$-th column, there are $e(S, j)$ unpaired half-edges among the nodes in $S$, and there are $\gamma d n - dj$ total unpaired half-edges among the row nodes. Since we have conditioned on $\Omega$, we know that $\frac{e(S,j)}{\gamma d n-dj}\geq \frac{k}{2n}$. Hence we have
\begin{equation}\label{eq:s0}
\Pr[s = 0 | \Omega] \leq \left(1 - \frac{k}{2n}\right)^{d} \leq e^{-\frac{dk}{2n}}.
\end{equation}

Combining eqs. \ref{eq:sneq0} and \ref{eq:s0}, it follows that 
\begin{equation*}
\begin{split}
\Pr[h_j(S) \cdot v' = 0 | \Omega] &= \Pr[s = 0| \Omega] + \left(1 - \Pr[s = 0| \Omega]\right)\Pr[h_j(S) \cdot v' = 0 | s \geq 1, \Omega]\\
&\leq \Pr[s = 0| \Omega] + \left(1 - \Pr[s = 0| \Omega]\right)\left(1/2 + \frac{d}{\gamma k}\right)\\
&\leq \Pr[s = 0| \Omega] + \left(1 - \Pr[s = 0| \Omega]\right)\left(1/2\right) + \frac{d}{\gamma k} \\
&= \frac{1}{2}\left(1 + \Pr[s = 0| \Omega]\right) +  \frac{d}{\gamma k}
\leq \frac{1}{2}\left(1 + e^{-\frac{dk}{2n}}\right) +  \frac{d}{\gamma k}.
\end{split}
\end{equation*}

\end{proof}
\begin{claim}\label{claim:double_match}
The probability that more than one half-edge is matched from a column-node $j$ to a single row-node in $S$ is at most $\frac{2\gamma d^3}{k}$. 
\end{claim}
\begin{proof}
$\binom{d}{2}\left(\frac{\gamma d}{k/4}\right) \leq \frac{2\gamma d^3}{k}.$
\end{proof}
It follows from the previous two claims that the probability that $h_j(S) \cdot v' \neq 0$ and at most one half-edge is match from a row-node $j$ to a single column-node in $S$ is at least 
\begin{equation*}
\begin{split}
1 - \Pr[\textnormal{Event in Claim~\ref{claim:double_match} occurs}] - \Pr[h_j(S)\cdot v' = 0 | \Omega ]&\geq
1 - \frac{2\gamma d^3}{k} - \frac{d}{\gamma k} - \frac{1}{2}\left(1 + e^{-\frac{dk}{2n}}\right) \\
&= \frac{1}{2}\left(1 - e^{-\frac{dk}{2n}}\right) - \frac{2\gamma d^3}{k} - \frac{d}{\gamma k} \\
&\geq \frac{1}{2}\left(1 - e^{-\frac{dk}{2n}}\right) - \frac{3\gamma d^3}{k}.
\end{split}
\end{equation*}

Using Claim~\ref{claim:success}, this proves the lemma.
\end{proof}

We are now ready to prove Lemma~\ref{abc_large}. 
\begin{proof}[Proof of Lemma~\ref{abc_large}]
Recall that our goal is to show that with high probability, for all $k \geq \frac{n}{18e\gamma^2 d^2}$, for all sets $S$ of size $k$, we have $R_{\gamma n/2} = 0$. Throughout the rest of the proof, we assume that we have conditioned on $\Omega$, since $\Pr[\Omega] = 1 - o(1)$.

For a fixed set $S$ of size $k$, by Lemma~\ref{lemma:decrease}, conditioned on each term being positive, the random process $R_0, R_1, R_2, \cdots, R_{\gamma n/2}$ is stochastically dominated by the random process $Y_0 = R_0, Y_1, Y_2, \ldots ,Y_{\gamma n/2}$, where $Y_{i + 1} = Y_i - \Ber(\mu_k)$ and $$\mu_k = \frac{1}{2}\left(1 - e^{-\frac{dk}{2n}}\right) - \frac{3\gamma d^3}{k}.$$

Hence $\Pr[R_{\gamma n/2} > 0] \leq \Pr[Y_{\gamma n/2} > 0]$. By a Chernoff Bound, since $Y_0 = k$, for any $\mu \leq \mu_k$, 
\begin{equation*}
\begin{split}
\Pr[Y_{\gamma n/2} > 0] &= \Pr[Y_{0} - Y_{\gamma n/2} < k] = \Pr[\Ber(k, \mu_k) < k] \leq e^{-\gamma n\mu/2}\left(\frac{e\gamma n\mu}{2k}\right)^k.
\end{split}
\end{equation*}

Define $\eta_k := k/n$. Taking a union bound over all sets $S$ of size $k$, conditioned on $\Omega$, the probability that at least one set $S$ has $R{\gamma n/2} > 0$ is at most
\begin{equation}\label{large_union}
\begin{split}
\binom{n}{k}e^{-\gamma n\mu/2}\left(\frac{e\gamma n\mu}{2k}\right)^k &\leq \left(\frac{en}{k}\right)^ke^{-\gamma n\mu/2}\left(\frac{e\gamma n\mu}{2k}\right)^k 
= e^{-\gamma n\mu/2}\left(\frac{e^2\gamma n^2\mu}{2k^2}\right)^k \\
&= e^{-\gamma n\mu/2+ k\log\left(\frac{e^2\gamma n^2\mu}{k^2}\right)}
= e^{k\left(-\frac{\gamma\mu}{2\eta_k} + \log\left(\frac{e^2\gamma\mu}{2\eta_k^2}\right)\right)}.
\end{split}
\end{equation}

We consider the two cases $\frac{n}{18e\gamma^2 d^2} \leq k \leq \frac{5n}{d}$ and $k \geq \frac{5n}{d}$.

In the first case, since $1 - e^{-x} \geq x/2$ for $0 < x < 1$, for $n$ large enough, we have $$\mu_k \geq \frac{1}{2}\left(\frac{d\eta_k}{4}\right) - \frac{3\gamma d^3}{k} \geq 
\frac{d\eta_k}{10}.$$

Hence for $\mu = \frac{d\eta_k}{10}$, we have
\begin{equation*}
-\frac{\gamma \mu}{2\eta_k} + \log\left(\frac{e^2\gamma \mu}{2\eta_k^2}\right) \leq -\frac{\gamma d}{20} + \log\left(\frac{e^2\gamma d}{20\eta_k^2}\right) \leq -\frac{\gamma d}{20} + \log\left(\frac{18^2e^3\gamma^5d^5}{20}\right) \leq  -0.25
\end{equation*}
for $d$ larger than some constant $d_0$. 

In the second case, we have $\mu_k \geq 0.45$, and hence for $\mu = 0.45$ and $\gamma \geq 16$, we have

\begin{equation*}
-\frac{\gamma \mu_k}{2\eta_k} + \log\left(\frac{e^2\gamma \mu_k}{2\eta_k^2}\right) \leq \max_{\eta \leq 1} -\frac{3.6}{\eta} + \log\left(\frac{3.6e^2}{\eta^2}\right)  \leq -0.25.
\end{equation*}

Returning to eq.~\ref{large_union}, conditioned on $\Omega$, the probability that there is a set $S$ of size $k$ for which $R_{\gamma n/2} > 0$ is at most $e^{-k/4}$. Summing over all $k \geq \frac{n}{18e\gamma^2 d^2}$ yields a probability of failure among any $k$ of at most $$\frac{e^{-\frac{n}{18e\gamma^2 d^2}}}{1 - e^{-1/4}} \leq 5e^{-\frac{n}{18e\gamma^2 d^2}}.$$
Unioning with the probability that $\Omega$ doesn't occur from Lemma~\ref{lemma:controlling_large} yields Lemma~\ref{abc_large}.
\end{proof}

\subsection{Proof of Theorem~\ref{char:abc}}
We are now ready to prove Theorem~\ref{char:abc}. 

We will use the following lemma on the concentrations of functions of a random permutation from \cite{core_fkss}, Lemma 3.3.\footnote{In an earlier version of this paper predating \cite{core_fkss}, we instead used a longer second moment argument. For simplicity, we have replaced this argument with a shorter argument which uses this lemma.}
\begin{lemma}\label{lem:injection-concentration}
Let $m\in \mathbb{N}$, let $S$ be any finite set of size $|S| \ge m$, let $\mathcal F$ be the set of functions $[m] \to S$ and let $\mathcal{I}\subseteq \mathcal F$ be the set of injections
$[m] \to S$. Consider a function $f:\mathcal F\to\mathbb{R}$
satisfying the property that if $\pi,\pi'$ agree except that $\pi(i)\ne \pi'(i)$, then $|f(\pi)-f(\pi')|\le c_{i}$.
Let $\pi\in\mathcal{I}$ be a uniformly random injection. Then for $t\ge 0$,
\[
\mb{P}\left[|f(\pi)-\mb{E} f(\pi)|\ge t\right]\le 2\exp\left(-\frac{t^{2}}{8\sum_{i=1}^{m}c_{i}^{2}}\right).
\]
\end{lemma}

We restate Theorem~\ref{char:abc} for the reader's convenience.

\charabc*

\begin{proof}[Proof of Theorem~\ref{char:abc}]
Let $\Omega$ be the event the there are no minimal dependencies among $k$ columns of $A$ that are not in $$\bigcup_{k \leq \log(n)}\mathcal{T}_k \cup \mathcal{T}_k^+ \cup \mathcal{T}_k^C .$$ By combining the result of the small case, Lemma~\ref{lemma:abc_small}, with the result of the large case, Lemma~\ref{abc_large}, we observe that $\Pr[\Omega] = 1 - o(1)$. This proves the first statement in the theorem. Next we bound the size of the set of columns involved in linear dependencies. We will do this by studying the expected number of linearly dependent columns and showing concentration via Lemma~\ref{lem:injection-concentration}. We begin with some set-up.

We extend the mapping in Definition~\ref{def:ABC} from permutations $\rho \in \mathcal{S}_{\gamma d n}$ to matrices in $\{0, 1\}^{\gamma n \times n}$ to be a more general mapping from $\mathcal{F}$, the set of functions $[\gamma n d] \rightarrow [\gamma n d]$ to matrices in $\{0, 1\}^{\gamma n \times n}$. Create $\gamma n$ row-nodes with $d$ half-edges each and $n$ column-nodes with $\gamma d$ half-edges as per Definition~\ref{def:ABC}. Given a function $\pi \in \mathcal{F}$, for $i \in [\gamma d n]$, attach the $i$th right half-edge to the $\pi(i)$th left half-edge. Note that some half-edges on the right may have multiple half-edges connected to it on the right. Let $A_0 \in \{0, 1\}^{\gamma n \times n}$ be the matrix where $(A_0)_{ij} = 1$ if and only if there is at least one half-edge from row-node $i$ connected to a half-edge from column-node $j$. As before, let $A$ be the submatrix containing the first $(1-p)\gamma n$ rows of $A_0$.

Let $\mathcal{S}(\pi) := \{S: A_S \in \mathcal{T}_k \cup \mathcal{T}_k^+ \cup \mathcal{T}_k^C, |S| \leq \log(n)\}$, and let $\mathcal{D}(\pi) := \cup_{S \in \mathcal{S}(\pi)} S$. For $\pi \in \Omega$, define $f(\pi) := |\mathcal{D}(\pi)|$. The following shows that $f$ satisfies a certain Lipshitz property on $\Omega.$ 

\begin{claim}
For any $\pi, \pi' \in \Omega$, we have $|f(\pi) - f(\pi')|\leq \log(n)\on{dist}(\pi, \pi')$.
\end{claim}
\begin{proof}
Let $\tau = \{i : \pi(i) \neq \pi'(i)\}$. Without loss of generality, assume $|\mathcal{D}(\pi)| > |\mathcal{D}(\pi')|$, and let $\Delta = \mathcal{D}(\pi) \setminus \mathcal{D}(\pi')$. Let $\mathcal{X}$ be a set of linearly independent vectors constructed in the following way: While there exists some $j \in \Delta$ which is not contained in any of the vectors in $\mathcal{X}$, add to $\mathcal{X}$ some vector $x$ such that $A(\pi)x = 0$, and $j \in \supp(x)$ and $|\supp(x)| \leq \log(n)$. Since $\pi \in \Omega$, it is always possible to find such a vector. Further, letting $\mathcal{X}_{\Delta} = \{x_{\Delta} : x \in \mathcal{X}\}$ be the set $\mathcal{X}$ restricted to indices in $\Delta$, we see that the set of vectors $\mathcal{X}_{\Delta}$ is linearly independent. Finally, we must have $|\mathcal{X}| \geq \frac{|\Delta|}{\log(n)}$. 

Assume for the sake of contradiction that $\frac{|\Delta|}{\log(n)} > |\tau|$. Then let $X \in \mathbb{R}^{n \times |\mathcal{X}|}$ be the matrix with the vectors of $\mathcal{X}$ as its columns, and let $a$ be some vector such that:
\begin{enumerate}
    \item $(Xa)_i = 0$ for all $i \in \tau$;
    \item $\supp(Xa) \cap \Delta \neq \emptyset$.
\end{enumerate}
Such a vector $a$ must exist because the rank of $X^{\Delta}$ is at least $\frac{|\Delta|}{\log(n)} > |\tau|$, which means that there mush be some row in $X^{\Delta}$ which is not in the span of the $|\tau|$ rows of $X^{\tau}$. 

Let $x = Xa$. For any $k \in \supp(x)$, by Lemma~\ref{lemma:support_minimal}, there exists some set $S \subset \supp(x)$ such that $A(\pi)_S \in \mathcal{M}_{|S|}$. However, since $\tau \cap \supp(x) = \emptyset$, it follows that $A(\pi')_S \in  \mathcal{M}_{|S|}$, and hence since $\pi' \in \Omega$, we must have $\supp(x) \subset \mathcal{D}(\pi')$. Since $\supp(x) \cap \Delta \neq \emptyset$, this is a contradiction.
\end{proof}

We now extend $f$ to $\mathcal{F} \setminus \Omega$. Define $\on{dist}(\pi, \pi') = |\{i : \pi(i) \neq \pi'(i)\}|$  Let $G$ be a graph with the node set $\mathcal{F}$, and let there be an edge of weight $\log(n)$ between any pair of vertices $\pi$ and $\pi'$ with $\on{dist}(\pi, \pi') = 1$. Construct an additional vertex $s$ with an edge of weight $f(\pi)$ from $s$ to $\pi$ for any $\pi \in \Omega$. Now define $f(\pi) := \on{dist}_G(\pi, s)$ to be the shortest-path distance in $G$ from $s$ to $\pi$ for any $\pi \in \mathcal{F}$. It is immediate from the triangle inequality that for any $\pi, \pi'$ with $\on{dist}(\pi, \pi') = 1$, $|f(\pi) - f(\pi')|\leq \log(n)$.

In the following claim, we bound the expectation of $f$. \begin{claim}
For some constant $c$, 
$\mathbb{E}_{\rho \sim \mathcal{S}_{\gamma d n}} f(\rho) \leq np^{\gamma d - c\log(\gamma d)}$.
\end{claim}
\begin{proof}
For a set $S \subset [n]$, let $X_S$ be the indicator of the event that $A_S \in \mathcal{T}_k \cup \mathcal{T}_k^+ \cup \mathcal{T}_k^C$. Notice that if $\rho \in \Omega$, $f(\rho)$ is at most
$X := \sum_{S \subset [n], |S| \leq \log(n)}|S|X_S$.

By Lemma~\ref{lemma:abc_small}, for some constant $c_{\ref{lemma:abc_small}}$, we have
\begin{equation*}
\begin{split}
\mathbb{E}[X] &= \sum_{S: |S| \leq \log(n)}{|S|\Pr[A_S \in \mathcal{T}_k \cup \mathcal{T}_k^+ \cup \mathcal{T}_k^C]}
\leq \sum_{k \geq 1}k\binom{n}{k}\left(\frac{k}{n}\right)^{k - 1}p^{\gamma dk + c_{\ref{lemma:abc_small}}k\log(\gamma d)} \\
&\leq n \sum_{k \geq 1}p^{\gamma d k + (c_{\ref{lemma:abc_small}} + 1)k\log(\gamma d)} \leq n p^{\gamma d - c_1\log(\gamma d)}
\end{split}
\end{equation*}
for some constant $c_1$. Since $\Pr[\Omega] = 1 - o(1)$, and $f \leq n$ always, we have $\mathbb{E}[f(\rho)] \leq o(n) + n p^{\gamma d - c_1\log(\gamma d)} \leq n p^{\gamma d - c_1\log(\gamma d) - o(1)} \leq n p^{\gamma d - c\log(\gamma d)}$ for some constant $c$.
\end{proof}
We now use Lemma~\ref{lem:injection-concentration} with $t = n^{3/4}$ to show that with probability $1 - o(1)$, $f(\pi) \leq \mathbb{E}[f(\pi)] + n^{3/4} \leq n p^{\gamma d - c\log(\gamma d)}$ for some constant $c$. Finally, using the fact that $f(\pi) = \left|\bigcup_{x : Ax = 0} \supp(x) \right|$ for $\pi \in \Omega$, we have $$\Pr\left[\left|\bigcup_{x : Ax = 0} \supp(x) \right|\leq n p^{\gamma d - c\log(\gamma d)}\right] \geq 1 - o(1).$$





\end{proof}

\subsection{Gradient Coding Results}



In this section, we prove the applications to gradient coding.

\abcstacked*

While the first bound will be a direct corollary of Corollary $\ref{abc_random}$, the second bound relies on some external lemmas that will be introduced below. Thus, for the sake of clarity, we will split the proof of Theorem $\ref{abc_stacked}$ into the following two lemmas.
\begin{lemma}\label{stacked_prop1}
Let $c,\gamma_0,d_0$ be the universal constants from Corollary $\ref{abc_random}$. Choose any $\gamma, d\in\mathbb{Z}^+$ such that $\gamma \geq \gamma_0$, $\gamma \mid d$, and $\frac{d}{\gamma }\geq d_0$. For any $n$ such that $\gamma \mid n$, let $B\sim \on{ABC}_{\gamma-\on{stacked}}(n,d)$. Then with probability $1-o(1)$ over the choice of $B$, we have that for any $p<\frac{1}{2}$:
\begin{equation*}
\frac{1}{n}\mathbb{E}_{S\sim\binom{[n]}{pn}} \textnormal{err}(B, S) \leq p^{d + c\log(d)}+o(1).
\end{equation*}
\end{lemma}

\begin{lemma}\label{stacked_prop2}
Let $\gamma,d$ be such that $\gamma \mid d$ and $\gamma \mid n$. Let $B\sim \on{ABC}_{\gamma-\on{stacked}}(n,d)$. With constant probability, we have:
\begin{equation*}
    \frac{1}{n}\max_{S \in \binom{[n]}{pn}}\left( \textnormal{err}(B, S)\right) \leq  \left(\frac{8\gamma^3p}{d}\right)+o(1).
\end{equation*}
\end{lemma}
Temporarily assuming these lemmas, we note that Theorem \ref{abc_stacked} follows immediately. We will now prove Lemma $\ref{stacked_prop1}$. 

\begin{proof}[Proof of Lemma \ref{stacked_prop1}]

Let $A_0\sim \textnormal{ABC}(n/\gamma,\gamma,d/\gamma)$  denote the ABC matrix which is stacked $\gamma$ times to generate $B$, then it follows that:
\begin{align*}
        \frac{1}{n}\mathbb{E}_{S \sim \binom{[n]}{pn}}{ \textnormal{err}(B, S)}&= \frac{1}{n}\mathbb{E}_{S \sim \binom{[n]}{pn}}\min_{w: w_j = 0 \: \forall \: j \in S}|Bw - \mathbbm{1}|_2^2\\
        &=  \frac{\gamma}{n}\mathbb{E}_{S \sim \binom{[n]}{pn}}\min_{w: w_j = 0 \: \forall \: j \in S}|A_0^\top w - \mathbbm{1}|_2^2\\
        &=  \frac{\gamma}{n}\mathbb{E}\left[\min_{w}|A^\top w - \mathbbm{1}|_2^2\right]
\end{align*}
where $A\sim \textnormal{ABC}_p(n/\gamma,\gamma,d/\gamma)$. By Corollary \ref{abc_random}, for some constant $c$, there is a $1-o(1)$ chance that $A_0$ has the property that the following holds with probability $1-o(1)$ over the choice of $S$:
\begin{equation*}
    \frac{\gamma}{n}\min_{w}|A^\top w - \mathbbm{1}|_2^2\leq p^{d-c\log(d)}.
\end{equation*}
Thus with probability $1 - o(1)$ over the choice of $B$, since $ \frac{\gamma}{n}\min_{w}|A^\top w - \mathbbm{1}|_2^2$ is always at most $1$,
\begin{equation*}
    \frac{1}{n}\mathbb{E}_{S \sim \binom{[n]}{pn}}{ \textnormal{err}(B, S)} = \frac{\gamma}{n}\mathbb{E}\left[\min_{w}|A^\top w - \mathbbm{1}|_2^2\right]\leq p^{\gamma d+c\log_p(\gamma d)} + o(1).
\end{equation*}

\end{proof}

Lemma \ref{stacked_prop2} requires a bit more machinery. We employ the following lemma from~\cite{GW} which allows us to bound the adversarial error of a gradient code as a function of the second largest singular value. Formally, we have the following lemma. 
\begin{lemma}[Proposition 4.1 of \cite{GW}]\label{stacked_lemma1}
Let $A\in \{0,1\}^{N\times M}$ be an assignment matrix such that each row has exactly $D$ $1$'s. Let $\sigma_2$ be the largest singular value of $A$. Then for any set of stragglers $S$ such that $|S|=s$, we have:
\begin{equation*}
    \frac{1}{N}\textnormal{err}(A,S)\leq\frac{1}{N}\left(\frac{\sigma_2}{D}\right)^2\frac{sM}{M-s} 
\end{equation*}
\end{lemma}
To calculate an upper bound on the second largest singular value of an ABC matrix, we first need a result from \cite{randomgraphs} which states that with constant probability, the configuration model we use to generate our ABC matrix has no rows which map to the same column node more than once ie., the bipartite graph produced is simple. 

Let $\mathcal{G}(n, \gamma n, d, \gamma d)$ denote the uniform distribution on simple $(d, \gamma d)$-biregular bipartite graphs with $n$ left nodes and $\gamma n$ right nodes.
\begin{lemma}[\cite{randomgraphs}]\label{stacked_lemma2}
Let $A_0\sim \on{ABC}(n,\gamma,d)$, then the probability that $\begin{pmatrix}
0 & A_0\\
A_0^\top & 0
\end{pmatrix}$
is the adjacency matrix of a bipartite, biregular random graph $G$ is at least $\varepsilon(d)>0$. Furthermore, if we condition on this event occurring, then $G\sim\mathcal{G}(n,\gamma n,d,\gamma d)$.
\end{lemma}
The previous lemma allows us to apply the following result of \cite{spectralgap} to bound the second largest singular value of these well behaved ABC matrices. 
\begin{lemma}[Theorem 4 of \cite{spectralgap}]\label{stacked_lemma3}
Let $A=\begin{pmatrix}
0 & X\\
X^\top & 0
\end{pmatrix}$ be the adjacency matrix of a bipartite, biregular random graph $G\sim \mathcal{G}(n,\gamma n,d,\gamma d)$. Then, with probability $1-o(1)$, $A$'s second largest eigenvalue $\lambda_2$ satisfies $$\lambda_2 \leq \sqrt{d_1-1}+\sqrt{d_2-1}+o(1).$$
\end{lemma}
We are now ready to prove Lemma \ref{stacked_prop2}. 
\begin{proof}[Proof of Lemma \ref{stacked_prop2}]
Let $A_0\sim \textnormal{ABC}(n/\gamma ,\gamma,d/\gamma)$ denote the $ABC$ matrix which is stacked $\gamma$ times to generate $B$. It follows that: 
\begin{align*}
     \frac{1}{n}\max_{S \in \binom{[n]}{pn}}{ \textnormal{err}(B, S)}&= \frac{1}{n}\max_{S \in \binom{[n]}{pn}}\min_{w: w_j = 0 \: \forall \: j \in S}|Bw - \mathbbm{1}|_2^2\\
        &=  \frac{\gamma}{n}\max_{S \in \binom{[n]}{pn}}\min_{w: w_j = 0 \: \forall \: j \in S}|A_0^\top w - \mathbbm{1}|_2^2\\
        &= \frac{\gamma}{n}\max_{S \in \binom{[n]}{pn}} \on{err}(A_0^\top,S).
\end{align*}
By Lemma $\ref{stacked_lemma2}$, with constant probability, $A'=\begin{pmatrix}
0 & A_0\\
A_0^\top & 0
\end{pmatrix}$ is the adjacency matrix of a bipartite, biregular random graph $G$. Condition on this event occurring. Then $G$ is uniformly sampled from $\mathcal{G}(n/\gamma,n,d/\gamma,d)$. By Lemma $\ref{stacked_lemma3}$, with probability $1-o(1)$, the second largest eigenvalue of $A'$ is $\lambda_2\leq \sqrt{d/\gamma-1}+\sqrt{d-1}+o(1)$. Thus, the second largest singular value of $A_0$ is $\sigma_2 \leq \sqrt{d/\gamma-1}+\sqrt{d-1}+o(1)$. Hence, by Lemma $\ref{stacked_lemma1}$, we have
\begin{align*}
    \frac{\gamma}{n}\max_{S \in \binom{[n]}{pn}} \on{err}(A_0^\top,S)&\leq\frac{\gamma}{n}\left(\frac{\gamma\sigma_2}{d}\right)^2\frac{pn^2}{(1-p)n}
    \leq \gamma^3\left(\frac{2\sqrt{d}+o(1)}{d}\right)^2\frac{p}{(1-p)}\\
    &\leq \left(\frac{4\gamma^3p}{d(1-p)}\right)+o(1) 
    \leq \left(\frac{8\gamma^3p}{d}\right)+o(1).
\end{align*}
\end{proof}

\section{Acknowledgements}
Thanks to Mary Wootters, Matthew Kwan, Vishesh Jain, Ashwin Sah, and Mehtaab Sawhney for helpful conversations and feedback.

\bibliography{main}
\bibliographystyle{plain}

\appendix

\section{Tail Bounds}
We include the following version of the Chernoff Bound for reference.
\begin{lemma}[Chernoff Bound]\label{chernoff}
Let $X_i \in [0, 1]$ be i.i.d. random variables with $\mathbb{E}[X_i] = p$. Then for $q \leq p$,
\begin{equation*}
    \Pr\left[\sum_i X_i \leq qn \right] \leq e^{\frac{-n(p - q)^2}{2p}}.
\end{equation*}
\end{lemma}

We also use the following lemma.
\begin{restatable}[Tail Bound on Binomial]{lemma}{taillemma}\label{lem:tail}
If $t \geq 2np$, then
\begin{equation*}
 \Pr\left[\Bin(n, p) \geq t\right] \leq 2\left(\frac{enp}{t}\right)^t.
\end{equation*}
\end{restatable}

\begin{proof}

\begin{equation*}
\begin{split}
\Pr\left[\Bin(n, p) \geq t\right] &\leq \sum_{i = t}^{n}\binom{n}{i}p^i 
\leq \sum_{i = t}^{n}\binom{n}{t}p^i\prod_{j = t + 1}^i\left(\frac{n - j + 1}{j}\right) \\
&\leq \sum_{i = t}^{n}\binom{n}{t}p^i\prod_{j = t + 1}^i\left(\frac{n}{t}\right) 
= \sum_{i = t}^{n}\binom{n}{t}p^t\left(\frac{pn}{t}\right)^{i - t} \\
&\leq  \binom{n}{t}p^t\sum_{j = 0}^{\infty}\left(\frac{pn}{t}\right)^{j} 
= \frac{\binom{n}{t}p^t}{1 - \frac{pn}{t}}.
\end{split}
\end{equation*}
For $t \geq 2np$, plugging in Sterling's formula yields the lemma.
\end{proof}

\section{Proof of Lemma~\ref{lemma:classification} and ~\ref{lemma:connect} for classification of minimal dependencies}\label{apx:matrices}

We restate Lemma~\ref{lemma:connect}.
\lemmaconnect*
\begin{proof}
Suppose $G$ was not connected, and without loss of generality let $[1,\ldots,L], [L+1,\ldots,k]$ be two disconnected components of $G$ for $1 \leq L \leq k-1$. By definition of $\mathcal{M}_k$ there is non-zero $v \in \mathbb{R}^k$ with $M v = 0$, and $v$ has all non-zero entries. Let $v^{(L)}, v^{(R)} \in \mathbb{R}^k$ with $v^{(L)} = (v_1,v_2,\ldots,v_L,0,0,\ldots,0)$ and $v^{(R)} = (0,0,\ldots,0,v_{L+1},v_{L+2},\ldots,v_k)$, such that $$M (v^{(L)} + v^{(R)}) = M  v = 0.$$ By the disconnectedness of $[1,\ldots,L]$ and $[L+1,\ldots,k]$, the set of rows of $M$ having non-zero entries in columns $[1,\ldots,L]$ is disjoint from the set of rows of $M$ having non-zero entries in columns $[L+1,\ldots, k]$. Thus $M v^{(L)}$ and $M v^{(R)}$ are non-zero in disjoint coordinates. Thus $M(v^{(L)} + v^{(R)}) = 0$ implies that $Mv^{(L)} = M  v^{(R)} = 0$. This implies the kernel of $M$ has rank at least two, which contradicts that $M$ has rank $k-1$. Thus $G$ is a connected graph.
\end{proof}

Recall Lemma~\ref{lemma:classification}:

\lemmaclassification*
We break the proof of Lemma~\ref{lemma:classification} into three lemmas, which we prove independently.

\begin{lemma}\label{class:tree}
$\mathcal{S}_{2k-2, k} = \mathcal{T}_k$. Further, for $B \in \mathcal{T}_k = \mathcal{S}_{2k - 2, k}$, there is a unique (up to constant multiple) non-zero vector $v$ satisfying $B  v = 0$, where $v_i = (-1)^{d_G(i,j)} v_j$, and $d_G(i,j)$ is the path length from vertex $i$ to $j$ in the graph $G$ with edge-vertex incidence matrix $B$.
\end{lemma}
\begin{proof}

We prove this by showing both inclusions in the following two claims.

\begin{claim}
$\mathcal{T}_k \subseteq \mathcal{S}_{2k-2, k}$. 
\end{claim}
\begin{proof}
Let $B \in \mathcal{T}_k$. Since $B$ has $k-1$ non-zero rows, it has rank at most $k-1$, and so the nullspace of $B$ has dimension at least $1$. Let $v$ be any vector such that $Bv = 0$. Suppose $(i,j)$ is an edge in the tree $G$ with edge-vertex incidence matrix $B$. Then there is a row $r$ of $B$ such that row $B^{\{r\}}$ has $1$'s at, and only at, coordinates $i$ and $j$. We have $B^{\{r\}} \cdot v = 0$ thus $v_j = -v_i$. Applying this to paths of multiple edges gives that $v_i = (-1)^{d_G(i,j)} v_j$. By this formula, if any $v_i$ is non-zero, then all other entries of $v$ are non-zero, and are uniquely determined by $v_i$. Thus $B$ has rank $k-1$. We also know that $B$ has $k-1$ non-zero rows each with two $1$'s, so it has $2k-2$ ones total. Thus $B \in \mathcal{S}_{2k - 2, k}$. In addition, by the description of $v$, this establishes the second statement of the lemma.
\end{proof}

\begin{claim}
$\mathcal{S}_{2k - 2, k} \subseteq \mathcal{T}_k$.
\end{claim}
\begin{proof}
Suppose $B \in \mathcal{S}_{2k - 2, k}$. Because $B$ has rank $k-1$, $B$ must have at least $k-1$ non-zero rows. Recall that by Observation~\ref{observation1}, no row of $B$ can have exactly one $1$. Thus, $B$ must have exactly $k-1$ non-zero rows, each with exactly two $1$'s. Let $G$ be be the graph on vertices $[1,\ldots,k]$ encoded by edge-vertex incidence matrix given by the non-zero rows of $B$. By Lemma~\ref{lemma:connect}, $G$ is connected.  Because $G$ is a connected graph on $k$ vertices with at most $k-1$ edges, $G$ must be a tree, and all its $k-1$ non-zero rows encode distinct edges. Thus $B \in \mathcal{T}_k$.
\end{proof}
\end{proof}

\begin{lemma}\label{class:hyperedge}
$\mathcal{T}_k^+ = \mathcal{S}_{2k - 1, k}$
\end{lemma}
\begin{proof}We prove this by showing both inclusions in the following two claims.

\begin{claim}
$\mathcal{T}_k^+ \subseteq \mathcal{S}_{2k - 1, k}$.
\end{claim}
\begin{proof}
Suppose $B \in \mathcal{T}_k^+$, such that $B$ is the hyperedge-vertex incidence matrix of a forest with two trees connected by a $3$-hyperedge. Without loss of generality, assume the first row of $B$ encodes this $3$-hyperedge and that $B_{11} = B_{12} = B_{1k} = 1$, with vertices $1$ and $2$ in the same tree and connected by an even path, and vertex $k$ in the other tree. We may further assume without loss of generality that rows $1,2,\ldots,L$ correspond to one tree, and $L+1,\ldots,k$ correspond to the other, for some $L$ with $2 \leq L \leq k-1$. 

Observe that $B^{/1}_{[L]}$ is the edge-vertex incidence matrix of some tree $T_1$. So by the second statement of Lemma~\ref{class:tree}, there is a unique (up to constant multiple) non-zero vector $v_1 \in \mathbb{R}^L$ satisfying $B^{/1}_{[L]}v_\ell = 0$. Similarly, there is a unique (up to constant multiple) non-zero vector $v_2 \in \mathbb{R}^{k-L}$ satisfying $B^{/1}_{[k] \setminus [L]}  v_2 = 0$. By the disjointness of the two trees, the nullspace of $B^{/1}$ is precisely spanned by the two vectors $v^1 := (v_1, 0^{k - L})$ and $v^2 := (0^L, v_2)$. It follows that the null space of $B$ is the set of vectors
$$\alpha v^1 + \beta v^2: \alpha, \beta \in \mathbb{R} \textnormal{ subject to }\alpha (v^1_1 + v^1_2) + \beta v^2_k = 0$$

Since vertices $1$ and $2$ are connected by an even path $T_1$, we have $v^1_1 = v^1_2$, so $v^1_1 + v^1_2 \neq 0$. Thus, the nullspace of $B$ is spanned solely by the vector $v := v^2_k v^1 - (v^1_1 +v^1_2)v^2$. Observe that $v$ has all non-zero entries because the coefficients $v^2_k$ and $v^1_1 + v^1_2$ are non-zero, and vectors $v^1$ and $v^2$ have disjoint support. Thus $B \in \mathcal{S}_{2k - 1, k}$.
\end{proof}
\begin{claim}
$\mathcal{S}_{2k - 1, k} \subseteq \mathcal{T}_k^+$.
\end{claim}
\begin{proof}
Suppose $B \in \mathcal{S}_{2k - 1, k}$. Since $B$ is a minimal dependency, none of its rows have a single $1$. Further, $B$ must have at least $k - 1$ non-zero rows since $B$ must have rank $k - 1$. Hence $B$ has $k-2$ rows with two $1$'s and one row with three $1$s. Without loss of generality, assume $B$'s first row has three $1$'s. Because $B$'s non-zero rows are linearly independent, $B^{/1}$ has rank $k-2$.

We first show that $B^{/1}$ is the edge-vertex incident matrix of a graph $G$ with two disjoint trees. First, we show that $G$ cannot have more than two disjoint components. Suppose it did have three disjoint components $X, Y$, and $Z$, with $X \cup Y \cup Z = [k]$. Then for a set $S \subseteq [k]$ and a vector $v \in \mathbb{R}^k$, let $v_S \in \mathbb{R}^k$ denote the vector with $(v_S)_i = v_i$ for $i \in S$, $(v_S)_i = 0$ otherwise. Thus, $v = v_X + v_Y + v_Z$, and $v_X, v_Y, v_Z$. Now if $Bv = 0$, then $B^{/1}v = 0$, and hence by the disjointness of the components $X, Y, Z$ in $G$, we have $B^{/1}v_X = B^{/1}v_Y = B^{/1}v_Z = 0$. Thus the nullity of $B^{/1}$ is at least $3$, which contradicts the fact that $\rank(B^{/1}) = k-2$. Thus, $G$ has at most two connected components. Because $B^{/1}$ has $k-2$ non-zero rows, $G$ has $k-2$ edges and $k$. So $G$ must be two disjoint trees $T_1$ and $T_2$.

We return to the first row of $B$. Without loss of generality let vertices $[1,\ldots, L]$ correspond to one of the disjoint trees in $G$, and $[L+1,\ldots,k]$ the others, with $1 \leq L \leq k-1$. We want to reason about the locations $a,b,c$ of the three coordinates with $B_{1a} = B_{1a} = B_{1a} = 1$. Applying Lemma \ref{class:tree} to $B^{/1}_{[L]}$ and $B^{/1}_{[k] \setminus [L]}$, yields two vectors $v^{(\ell)}, v^{r)} \in \mathbb{R}^k$ with $\supp(v^{(\ell)}) = [L]$ and $\supp(v^{(r)}) = [k] \setminus [L]$, and $B^{/1}v^{(\ell)} = B^{/1}v^{(r)} = 0$. We can rule out that all of $a,b,c \in [L]$: Suppose for sake of contradiction this was the case. Then letting $u := v^{(r)}$ satisfies $(B u)_1 = 1 v^{(r)}_a + 1 v^{(r)}_b + 1 v^{(r)}_c = 0 + 0 + 0 = 0$, and $B^{/1} v^{(r)} = 0$, so $Bu = 0$. Then $u$ is a non-zero kernel vector of $B$, which contradicts $B \in \mathcal{S}_{2k - 2, k}$. Similarly, we cannot have all of $a,b,c \in [L+1,\ldots, k]$. 

Without loss of generality, assume $a,b \in [1,\ldots,L], c \in [L+1,\ldots,k]$. Suppose for sake of contradiction that $a$ and $b$ are connected by an odd length path in $T_1$. Then $v^{(\ell)}_a = -v^{(\ell)}_b$. It follows that $v^{(\ell)}$ is a kernel vector of $B$, which again contradicts having $B \in \mathcal{S}_{2k - 2, k}$. Thus $a$ and $b$ are connected by an even length path in $T_1$.

\end{proof}
\end{proof}

\begin{lemma}\label{class:evencycle}
$\mathcal{S}_{2k, k}' = \mathcal{T}_k^C$.
\end{lemma}

\begin{proof}We prove this by showing both inclusions in the following two claims.
\begin{claim} $\mathcal{T}_k^C \subseteq \mathcal{S}_{2k, k}'$
\end{claim}
\begin{proof}
Suppose $B \in \mathcal{T}_k^C$. Then $B$ has $k$ rows with two $1$'s each. Without loss of generality, assume the first row of $B$ is an edge in the cycle. Then $B^{/1} \in \mathcal{T}_k$, so $B^{/1} \in \mathcal{S}_{2k - 2, k}$ by Lemma \ref{class:tree}. Hence there is a unique (up to constant multiple) $v$ with $B^{/1} v = 0$, and all non-zero entries. Write the first row of $B$ as $e_x^\top + e_y^\top$, for $1 \leq x < y \leq k$. 

We claim that $e_x^\top + e_y^\top$ is in the span of the rows of $B^{/1}$. By definition, the edge $(x,y)$ is part of an even length cycle in the graph with incidence matrix $B$, given by the ordered vertices $x = a_1, a_2, \ldots, a_n = y$ for some even $n$. For $1 \leq i \leq n - 1$, $e_{a_i}^\top + e_{a_{i + 1}}^\top$ is a row of $B^{/1}$. Observe that $$\sum_{i = 1}^{n - 1}(-1)^{i + 1}(e_{a_i}^\top + e_{a_{i + 1}}^\top) = e_{a_1}^\top + e_{a_{n}}^\top = e_x^\top + e_y^\top.$$
Thus the first row of $B$ is in the row span of $B^{/1}$, so $Bv = 0$. Because $B^{/1}$ has rank $k-1$, so does $B$. So $B \in \mathcal{S}_{2k, k}$. Moreover, $B \in \mathcal{S}_{2k, k}'$ because it has $k$ non-zero columns each with two $1$'s.
\end{proof}

\begin{claim} $\mathcal{S}_{2k, k}' \subseteq \mathcal{T}_k^C$.
\end{claim}
\begin{proof}

Let $B \in \mathcal{S}_{2k, k}'$. Because $B$ has rank $k-1$, and $k$ non-zero rows, there is a subset of $k-1$ rows of $B$ with rank $k-1$. Without loss of generality, assume the first row of $B$ is in the span of the other rows. Let $v$ be such that $B v = 0$, and $\supp(v). =[k]$. Then clearly $B^{/1} v = 0$. Further because $B^{/1}$ has rank $k-1$, the rank of the nullspace of $B^{/1}$ is $1$. Thus $v$ must be the unique (up to constant multiple) non-zero vector satisfying $B^{/1} v = 0$. So $B^{/1} \in \mathcal{S}_{2k - 2, k} = \mathcal{T}_{k}$. Write the first row of $B$ as $e_x^\top + e_y^\top$ for $1 \leq x < y \leq k$. To show $B \in \mathcal{T}_k^C$, we must show that $x$ and $y$ are connected by an odd length path in the tree with incidence matrix $B^{/1}$. We know that $e_x^\top + e_y^\top$ is in the span of the rows of $B^{/1}$, so we can write
$e_x + e_y = \sum_{i=2}^{m} \alpha_i (e_{a_i} + e_{b_i}),$
where $m$ is the height of $B$, and for each $i$ where $\alpha \neq 0$ and the row $B^{\{i\}}$ is non-zero, we have $B^{\{i\}} = e_{a_i}^\top + e_{b_i}^\top$. Let $S = \{i : \alpha_i \neq 0 \land B^{\{i\}} \neq 0\}$. If some $1 \leq j \leq k$ appears in only one pair in $\{(a_i,b_i)\}_{i \in S}$, it must be the case that $\sum_{i=1}^{m} \alpha_i (e_{a_i} + e_{b_i})$ has non-zero $e_j$ coefficient. So $\{(a_i,b_i)\}_{i \in S}$, is a collection of edges of a tree such that at most two vertices ($x$ and $y$) belong to only one edge. Then $\{(a_i,b_i)\}_{i \in S}$, must be a path, with $x$ and $y$ at the end points. Without loss of generality let $a_2 = x, b_{|S| + 1} = y$ and $a_i = b_{i+1}$ for $2 \leq i \leq |S|$. For $2 \leq i \leq |S|$, because $e_{b_i} \perp e_x + e_y$, we must have that $\alpha_{i+1} = -\alpha_i$ so that the sum will cancel in coordinate $b_i$. It follows that
$$e_x + e_y = \sum_{i=2}^{|S| + 1} (-1)^i \alpha_2 (e_{a_i} + e_{b_i}) = \alpha_2 e_x + (-1)^{|S| + 1} \alpha_2 e_y.$$

It is clear that $|S|$ must be odd. Thus $(x, y)$ fulfills the conditions of the additional edge forming an even cycle in $\mathcal{T}_k^C$. We thus have $B \in \mathcal{T}_k^C$.
\end{proof}
\end{proof}

\section{Proofs of Lemmas \ref{masterlemma:small} and \ref{large_general} on bounds of Binomial distributions}\label{sec:bounds}

\mastersmall*

\begin{proof}
We break down this sum as follows. 
\begin{equation}\label{masterbreakdown}
\begin{split}
\sum_{\ell \geq 1}^{\infty}&\Pr\left[\Bin\left(\ell, \frac{\ell + k}{\gamma n}\right) \geq \max\left(j, \frac{\ell}{2}\right)\right]\Pr[\Bin(\gamma nk, d/n) = \ell] \\
&\leq  \sum_{\ell = j}^{2j}\Pr\left[\Bin\left(\ell, \frac{\ell + k}{\gamma n}\right) \geq j\right]\Pr[\Bin(\gamma nk, d/n) = \ell]\\
&\qquad + \sum_{\ell \geq 2j + 1}^ {n/3}\Pr\left[\Bin\left(\ell, \frac{\ell + k}{\gamma n}\right) \geq \frac{\ell}{2}\right]\Pr[\Bin(\gamma nk, d/n) = \ell] \\
&\qquad + \sum_{\ell \geq n/3}^{\infty}\Pr\left[\Bin\left(\ell, \frac{\ell + k}{\gamma n}\right) \geq \frac{\ell}{2}\right]\Pr[\Bin(\gamma nk, d/n) = \ell]
\end{split}
\end{equation}

We bound the first term in the following claim.
\begin{claim}\label{bdclaim1}
\begin{equation*}
\sum_{\ell = j}^{2j}\Pr\left[\Bin\left(\ell, \frac{\ell + k}{\gamma n}\right) \geq j\right]\Pr[\Bin(\gamma nk, d/n) = \ell] \leq 8k\left(\frac{20ek}{\gamma n}\right)^j\left(2e\gamma d\right)^{4k}e^{-\gamma dk}.
\end{equation*}
\end{claim}
\begin{proof}
\begin{equation}\label{breakdown1}
\begin{split}
\sum_{\ell = j}^{2j}&\Pr\left[\Bin\left(\ell, \frac{\ell + k}{\gamma n}\right) \geq j\right]\Pr[\Bin(\gamma nk, d/n) = \ell] \\
&\leq 2j\Pr\left[\Bin\left(2j, \frac{2j + k}{\gamma n}\right) \geq j\right]\Pr[\Bin(\gamma nk, d/n) \leq 2j] \\
& \leq 4j\left(\frac{e2j(2j + k)}{\gamma nj}\right)^j\Pr[\Bin(\gamma nk, d/n) \leq 2j] 
\leq 8k\left(\frac{20ek}{\gamma n}\right)^j\Pr[\Bin(\gamma nk, d/n) \leq 2j].
\end{split}
\end{equation}
Here the first inequality follows from the fact that the summand is highest for $\ell = 2j$, the second inequality follows from the tail bound in Lemma~\ref{lem:tail}, and the third inequality follows from the fact that $j \leq k + 1 \leq 2k$. Now 
\begin{equation*}
\begin{split}
\Pr[\Bin(\gamma  nk, d/n) \leq 2j] &\leq \Pr[\Bin(\gamma  nk, d/n) \leq 2j]
\leq \binom{\gamma  nk}{2j}\left(\frac{d}{n}\right)^{2j}\left(1 - \frac{d}{n}\right)^{\gamma  nk - 2j} \\
&\leq \left(\frac{e\gamma  nk}{2j}\right)^{2j}\left(\frac{d}{n - d}\right)^{2j}\left(1 - \frac{d}{n}\right)^{\gamma nk} 
\leq \left(\frac{e\gamma dk}{j}\right)^{2j}e^{-\gamma dk}
\leq \left(2e\gamma d\right)^{4k}e^{-\gamma dk}.
\end{split}
\end{equation*}

Combining this with eq.~\ref{breakdown1} yields the claim.
\end{proof}

We bound the second term in eq.~\ref{masterbreakdown} in the following claim.
\begin{claim}\label{bdclaim2}
\begin{equation*}
\sum_{\ell \geq 2j + 1}^ {n/3}\Pr\left[\Bin\left(\ell, \frac{\ell + k}{\gamma n}\right) \geq \frac{\ell}{2}\right]\Pr[\Bin(\gamma nk, d/n) = \ell] \leq 4e^{-\gamma dk}\left(\frac{8e^3\gamma d^2k}{n}\right)^{j + 1}.
\end{equation*}.
\end{claim}
\begin{proof}

For $\ell \leq n/3$, using Lemma~\ref{lem:tail} and Sterling's forumla, we have
\begin{equation}\label{eq:med}
\begin{split}
\Pr&\left[\Bin\left(\ell, \frac{\ell + k}{\gamma n}\right) \geq \frac{\ell}{2}\right]\Pr[\Bin(\gamma nk, d/n) = \ell]\\
&\leq 2\left(\frac{e\ell(\ell + k)}{\gamma n\frac{\ell}{2}}\right)^{\lceil{\ell/2}\rceil}\Pr[\Bin(\gamma nk, d/n) = \ell] 
= 2\left(\frac{2e(\ell + k)}{\gamma n}\right)^{\lceil{\ell/2}\rceil}\Pr[\Bin(\gamma nk, d/n) = \ell] \\
&\leq 2\left(\frac{2e(\ell + k)}{\gamma n}\right)^{\lceil{\ell/2}\rceil}\left(\frac{e\gamma nk}{\ell}\right)^{\ell}\left(\frac{d}{n - d}\right)^\ell\left(1 - \frac{d}{n}\right)^{\gamma nk}
\leq 2\left(\frac{4e\ell}{\gamma n}\right)^{\lceil{\ell/2}\rceil}\left(\frac{2e\gamma dk}{\ell}\right)^{\ell} e^{-\gamma dk}.
\end{split}
\end{equation}
We do casework on the parity of $\ell$.
Let $\ell = 2a + b$, where $b \in \{0, 1\}$. If $b = 1$, then 
\begin{equation*}
\begin{split}
\left(\frac{4e\ell}{\gamma n}\right)^{\lceil{\ell/2}\rceil}\left(\frac{2e\gamma dk}{\ell}\right)^{\ell} &= \left(\frac{4e(2a+1)}{\gamma n}\right)^{a + 1}\left(\frac{2e\gamma dk}{2a + 1}\right)^{2a + 1} \\
&= \left(\frac{16e^3\gamma ^2d^2k^2(2a+1)}{\gamma n(2a + 1)^2}\right)^a\left(\frac{8e^2\gamma dk(2a+1)}{\gamma n(2a + 1)}\right) \\
&\leq \left(\frac{8e^3\gamma d^2k^2}{na}\right)^a\left(\frac{8e^2dk}{n}\right).
\end{split}
\end{equation*}
Now since the maximum over $x$ of $f(x) = \left(\frac{y}{x}\right)^x$ is achieved at $x = y/e$, and above this value of $x$, the $f(x)$ is decreasing, since $j \geq k - 1 \geq \frac{e8e^3\gamma d^2k^2}{n}$, we have for all $a \geq j$,
\begin{equation*}
\begin{split}
\left(\frac{8e^3\gamma d^2k^2}{na}\right)^a\left(\frac{8e^2dk}{n}\right) &\leq 
\left(\frac{8e^3\gamma d^2k^2}{nj}\right)^a\left(\frac{8e^2dk}{n}\right)
\leq \left(\frac{8e^3\gamma d^2k^2}{n(k - 1)}\right)^a\left(\frac{8e^2dk}{n}\right) 
\leq \left(\frac{8e^3\gamma d^2k}{n}\right)^{a + 1}. 
\end{split}
\end{equation*}
If $b = 0$, then 
\begin{equation*}
\left(\frac{4e\ell}{\gamma n}\right)^{\lceil{\ell/2}\rceil}\left(\frac{2e\gamma dk}{\ell}\right)^{\ell} = \left(\frac{8ea}{\gamma n}\right)^{a}\left(\frac{2e\gamma dk}{2a}\right)^{2a}
= \left(\frac{8e^3\gamma d^2k^2}{na}\right)^a.
\end{equation*}

By the same reasoning as before, we have for all $a \geq j + 1$,
\begin{equation*}
\left(\frac{8e^3\gamma d^2k^2}{na}\right)^a \leq \left(\frac{8e^3\gamma d^2k^2}{n(j + 1)}\right)^{a} \leq \left(\frac{8e^3\gamma d^2k}{n}\right)^{a}.
\end{equation*}

Combining these two cases back into eq.~\ref{eq:med}, we have for all $\ell \geq 2j + 1$,
\begin{equation*}
\Pr\left[\Bin\left(\ell, \frac{\ell + k}{\gamma n}\right) \geq \frac{\ell}{2}\right]\Pr[\Bin(\gamma nk, d/n) = \ell] \leq 2e^{-\gamma dk}\left(\frac{8e^3\gamma d^2k}{n}\right)^{\lceil{\frac{\ell}{2}}\rceil}.
\end{equation*}

Summing over all $\ell \geq 2j + 1$, we have 
\begin{equation*}
\sum_{\ell \geq 2j + 1}^ {n/3}\Pr\left[\Bin\left(\ell, \frac{\ell + k}{\gamma n}\right) \geq \frac{\ell}{2}\right]\Pr[\Bin(\gamma nk, d/n) = \ell] \leq 4e^{-\gamma dk}\left(\frac{8e^3\gamma d^2k}{n}\right)^{j + 1}.
\end{equation*}.
\end{proof}

Finally, we bound the third term in eq.~\ref{masterbreakdown} in the following claim.
\begin{claim}\label{bdclaim3}
\begin{equation*}
\sum_{\ell \geq n/3}^{\infty}\Pr\left[\Bin\left(\ell, \frac{\ell + k}{\gamma n}\right) \geq \frac{\ell}{2}\right]\Pr[\Bin(\gamma nk, d/n) = \ell] \leq  2\left(\frac{3e\gamma dk}{n}\right)^{n/3}. 
\end{equation*}
\end{claim}
\begin{proof}
It suffices to bound the probability $\Pr[\Bin(\gamma nk, d/n) \geq n/3]$. Using Lemma~\ref{lem:tail}, we have 
\begin{equation*}
\Pr[\Bin(\gamma nk, d/n) \geq n/3] \leq 2\left(\frac{3e\gamma dk}{n}\right)^{n/3}. 
\end{equation*}
\end{proof}

Combining claims~\ref{bdclaim1}, \ref{bdclaim2} and \ref{bdclaim3}, we have 
\begin{equation*}
\begin{split}
\sum_{\ell \geq 1}^{\infty}&\Pr\left[\Bin\left(\ell, \frac{\ell + k}{\gamma n}\right) \geq \max\left(j, \frac{\ell}{2}\right)\right]\Pr[\Bin(\gamma nk, d/n) = \ell] \\
&\leq 8k\left(\frac{20ek}{\gamma n}\right)^j\left(2e\gamma d\right)^{4k}e^{-\gamma dk} 
 + 4e^{-\gamma dk}\left(\frac{8e^3\gamma d^2k}{n}\right)^{j + 1}
 + 2\left(\frac{3e\gamma dk}{n}\right)^{n/3}.
\end{split}
\end{equation*}
One can check that this sum is dominated by the first term, so for a universal constant $c_{\ref{masterlemma:small}}$,
\begin{equation*}
\begin{split}
\sum_{\ell \geq 1}^{\infty}&\Pr\left[\Bin\left(\ell, \frac{\ell + k}{\gamma n}\right) \geq \max\left(j, \frac{\ell}{2}\right)\right]\Pr[\Bin(\gamma nk, d/n) = \ell] \\
&\leq 16k\left(\frac{20ek}{\gamma n}\right)^j\left(2e\gamma d\right)^{4k}e^{-\gamma dk}
\leq \left(\frac{k}{n}\right)^je^{-\gamma dk + c_{\ref{masterlemma:small}}k\log(\gamma d)}.
\end{split}
\end{equation*}
\end{proof}

\largegeneral*

\begin{proof}

Because $n/2 < n-k-1$, we may bound the probability by:
\begin{equation*}
\begin{split}
\Pr\left[\Bin\left(n - k - 1, 1 - \frac{1}{\sqrt{kd/n}}
\right) < k\right] &\leq \Pr\left[\Bin\left(\frac{n}{2}, 1 - \frac{1}{\sqrt{kd/n}}\right) < k\right].
\end{split}
\end{equation*}

We use the Chernoff bound $\Pr[X \leq (1 - \delta)\mu] \leq e^{-\mu \delta^2/2}$ for $\mu = \mathbb{E}[X]$, plugging in $\mu = \frac{n}{2}\left(1 - \frac{1}{\sqrt{kd/n}}\right)$ and $\delta = (1 - \frac{1}{\sqrt{kd/n}})^{-1} \left(1 - \frac{1}{\sqrt{kd/n}} - \frac{2k}{n}\right)$. This gives

\begin{equation*}
\begin{split}
\Pr\left[\Bin\left(n - k - 1, 1 - \frac{1}{\sqrt{kd/n}}\right) < k\right] &\leq \Pr\left[\Bin\left(\frac{n}{2}, 1 - \frac{1}{\sqrt{kd/n}}\right) < k\right]\\
&\leq e^{-\frac{n}{2}\left(1 - \frac{1}{\sqrt{kd/n}}\right)\left(1 - \frac{1}{\sqrt{kd/n}} - \frac{2k}{n}\right)^2}
\leq  e^{-n\epsilon},
\end{split}
\end{equation*}
where $\epsilon \geq 1/36$. To achieve this value of $\epsilon$, we plugged in  $k < \frac{n}{12}$ in the last inequality.

Now we compute the sum over $k$:

\begin{equation*}
\sum_{k = \frac{2n}{d}}^{\frac{n}{C}}\binom{n}{k}e^{-\epsilon n} \leq n \binom{n}{n/C} e^{-\epsilon n} \leq n (eC)^{n/C} e^{-\epsilon n} = e^{n(\frac{\log n}{n} + \frac{1 + \log C}{C} - \epsilon)},
\end{equation*}

which for constant $C$ large enough, is $e^{-\Theta(n)}$. 
\end{proof}

\section{Proof of Lemmas for $3$-core}\label{apx:core}
We prove Lemma~\ref{claim:xr}, which we restate here.

\claimxr*

The lemma will follow immediately from the following two lemmas:
\begin{lemma}\label{lem:bin_1}
For any $2 \leq k \leq n$ and $\ell \leq \frac{1}{4}n^{\epsilon}k^{1 - \epsilon}$,
\begin{equation*}
    \Pr\left[\Bin\left(\ell, \frac{k + \ell}{n}\right) \geq Q\right] \leq 2\left(\frac{k}{3n}\right)^{k + \frac{1}{3}},
\end{equation*}
where $\epsilon = \frac{1}{12}$ and $Q = \ceil*{\frac{3k-1}{2}}$.
\end{lemma}

\begin{lemma}\label{lem:bin_2}
For any $d$ larger than some constant, for any $2 \leq k \leq \frac{n}{d^7}$ and $\ell \leq 3dk$, with $Q = \ceil*{\frac{3k-1}{2}}$,
\begin{equation*}
    \Pr\left[\Bin\left(\ell, \frac{k + \ell}{n}\right) \geq Q\right] \leq 2\left(\frac{k}{3n}\right)^{k + \frac{1}{3}}.
\end{equation*}
\end{lemma}

\begin{proof}[Proof of Lemma~\ref{lem:bin_1}]
We will use Lemma~\ref{lem:tail} to bound the binomial. We first check that the hypothesis of the lemma is satisfied, namely that $2\ell \frac{k + \ell}{n} \leq Q$ for $\ell \leq \frac{1}{4}n^{\epsilon}k^{1 - \epsilon}$ and $\ell \geq Q \geq k$:

\begin{equation}\label{eq:eps_1}
\begin{split}
\frac{\ell(\ell + k)}{nQ} & \leq \frac{2\ell^2}{nQ} 
\leq \frac{2\left(\frac{1}{4}n^{\epsilon}k^{1 - \epsilon}\right)^2}{kn}
=  \frac{1}{8}\left(\frac{k}{n}\right)^{1 - 2\epsilon}.
\end{split}
\end{equation}

Since $k \leq n$, this value in eq.\ref{eq:eps_1} at most $\frac{1}{2}$. Hence the conditions of Lemma~\ref{lem:tail} are satisfied, so 
\begin{equation*}
\begin{split}
\Pr\left[\Bin\left(\ell, \frac{k + \ell}{n}\right) \geq Q\right] \leq 2\left(\frac{e\ell (k + \ell)}{Q n}\right)^Q.
\end{split}
\end{equation*}

Plugging in $Q = \ceil*{\frac{3k - 1}{2}}$ and $k \geq 2$, we can bound
\begin{equation*}
\begin{split}
\Pr\left[\Bin\left(\ell, \frac{k + \ell}{n}\right) \geq Q\right]  &\leq 2\left(\frac{e\ell(k + \ell)}{Q n}\right)^{\lceil{\frac{3k - 1}{2}}\rceil}
\leq  2\left(\frac{e}{8}\left(\frac{k}{n}\right)^{1 - 2\epsilon}\right)^{\lceil{\frac{3k - 1}{2}}\rceil}\\
&\leq 2\left(\left(\frac{k}{3n}\right)^{1 - 2\epsilon}\right)^{\lceil{\frac{3k - 1}{2}}\rceil}
\leq 2\left(\frac{k}{3n}\right)^{k + \frac{1}{3}}.
\end{split}
\end{equation*}
Here the second inequality plugs in the result of eq.\ref{eq:eps_1}, the third inequality uses the fact that $\left(\frac{e}{8}\right)^{6/5} \leq 1/3$, and the final inequality uses the fact that $(1 - 2\epsilon)\lceil{\frac{3k - 1}{2}}\rceil = \frac{5}{6}\lceil{\frac{3k - 1}{2}}\rceil \geq k + \frac{1}{3}$ for integers $k \geq 2$. This proves the lemma.

\end{proof}

\begin{proof}[Proof of Lemma~\ref{lem:bin_2}]
We will use Lemma~\ref{lem:tail} to bound the binomial. We first check that the hypothesis of the lemma is satisfied, namely that $2\ell \frac{k + \ell}{n} \leq Q$ for $\ell \leq 3dk$ and $\ell \geq Q \geq k$:

\begin{equation}\label{eq:eps_2}
\frac{\ell(\ell + k)}{nQ} \leq \frac{3d(3d+1)k^2}{kn} = \frac{3d(3d+1)k}{n} \leq \frac{10d^2k}{n}.
\end{equation}

Since we have $k \leq \frac{n}{d^7}$, if $d$ is larger than some constant, the value in eq.\ref{eq:eps_2} at most $\frac{1}{2}$. Hence the conditions of Lemma~\ref{lem:tail} are satisfied, so 
\begin{equation*}
\begin{split}
\Pr\left[\Bin\left(\ell, \frac{k + \ell}{n}\right) \geq Q\right] \leq 2\left(\frac{e\ell (k + \ell)}{Q n}\right)^Q.
\end{split}
\end{equation*}

Plugging in $Q = \ceil*{\frac{3k - 1}{2}}$ and using the fact that $$\ceil*{\frac{3k - 1}{2}} = k + \floor*{\frac{k}{2}} = 3\floor*{\frac{k}{2}} + \mathbf{1}_{k  \text{ odd}},$$ for any positive integer $k$,
\begin{equation}\label{eq:jenson}
\begin{split}
\Pr\left[\Bin\left(\ell, \frac{k + \ell}{n}\right) \geq Q\right]  &\leq 2\left(\frac{e\ell(k + \ell)}{Q n}\right)^{\lceil{\frac{3k - 1}{2}}\rceil}
\leq  2\left(\frac{10ed^2k}{n}\right)^{\ceil*{\frac{3k - 1}{2}}}\\
&= 2\left(\frac{k}{3n}\right)^{k + \frac{1}{3}}\left(\frac{k}{3n}\right)^{\lfloor{\frac{k}{2}}\rfloor - \frac{1}{3}}\left(\frac{30ed^2}{1}\right)^{\lceil{\frac{3k - 1}{2}}\rceil}\\
&= 2\left(\frac{k}{3n}\right)^{k + \frac{1}{3}}\left(\frac{k(30ed^2)^3}{3n}\right)^{\lfloor{\frac{k}{2}}\rfloor - \frac{1}{3}}\left(\frac{30ed^2}{1}\right)^{1 + \mathbf{1}_{k  \text{ odd}}}.
\end{split}
\end{equation}

Note that the function $$\left(\frac{k(30ed^2)^3}{3n}\right)^{\lfloor{\frac{k}{2}}\rfloor - \frac{1}{3}}\left(\frac{30ed^2}{1}\right)^{1 + \mathbf{1}_{k  \text{ odd}}} $$
is convex in $k$ (when the parity of $k$ is fixed), so its maximum is achieved at either the minimum or maximum value of $k$, up to parity.

It is easy to check by plugging in the minimum and maximum odd and even values of $k$ for $2 \leq k \leq \frac{n}{d^7}$, that since $\frac{n}{d^7} < \frac{3n}{2(30ed^2)^3}$ for $d$ and $n$ large enough, we have 
\begin{equation*}
\left(\frac{k(30ed^2)^3}{3n}\right)^{\lfloor{\frac{k}{2}}\rfloor - \frac{1}{3}}\left(\frac{30ed^2}{1}\right)^{1 + \mathbf{1}_{k  \text{ odd}}} < 1.   
\end{equation*}

Plugging this in to eq.\ref{eq:jenson}, we have
\begin{equation}
\Pr\left[\Bin\left(\ell, \frac{k + \ell}{n}\right) \geq Q\right] \leq 2\left(\frac{k}{3n}\right)^{k + \frac{1}{3}}.
\end{equation}
\end{proof}

Next we prove the following lemma, which we restate here.
\rmax*
\begin{proof}
By Lemma~\ref{lem:tail}, 
\begin{equation*}
\begin{split}
\Pr\left[\Bin\left(kn, \frac{d}{n}\right) \geq L_{max}\right] \leq 2\left(\frac{dke}{L_{max}}\right)^{L_{max}},
\end{split}
\end{equation*}
Let $\gamma := \log_n(n/k)$ such that $L_{max} = k\max\left(\frac{n^{\gamma\epsilon}}{4}, 3d\right)$. Then 
\begin{equation}\label{R_max_bd}
\begin{split}
\Pr\left[\Bin\left(kn, \frac{d}{n}\right) \geq L_{max}\right] &\leq 2\left(\frac{de}{\max\left(\frac{n^{\gamma\epsilon}}{4}, 3d\right)}\right)^{k\max\left(\frac{n^{\gamma\epsilon}}{4}, 3d\right)} \\
&= 2\left(\frac{k}{en}\right)^k\left(\left(\frac{de}{\max\left(\frac{n^{\gamma\epsilon}}{4}, 3d\right)}\right)^{\max\left(\frac{n^{\gamma\epsilon}}{4}, 3d\right)}\frac{ne}{k}\right)^k \\
&= 2\left(\frac{k}{en}\right)^k\left(\left(\frac{de}{\max\left(\frac{n^{\gamma\epsilon}}{4}, 3d\right)}\right)^{\max\left(\frac{n^{\gamma\epsilon}}{4}, 3d\right)}n^{\gamma}e\right)^{n^{1 - \gamma}}. \\
\end{split}
\end{equation}
If $n^{\gamma\epsilon} < (4de)^2$, then 
\begin{equation}\label{case:small_r}
\begin{split}
\left(\left(\frac{de}{\max\left(\frac{n^{\gamma\epsilon}}{e}, 3d\right)}\right)^{\max\left(\frac{n^{\gamma\epsilon}}{e}, 3d\right)}n^\gamma e\right)^{n^{1 - \gamma}} &\leq \left(\left(\frac{e}{3}\right)^{3d}e(4de)^{2/\epsilon}\right)^{n^{1-\gamma}} \\
&= e^{n^{1-\gamma}\left(3d\log(\frac{e}{3}) + \frac{4}{\epsilon}\log(4de) + 1\right)}
\leq e^{n^{1-\gamma}}
\leq e^{-n^{1/2}},
\end{split}
\end{equation}
where the second to last inequality follows from taking $d$ large enough, and the final inequality from the fact that by assumption, $\gamma \leq \frac{2\log(4ed)}{\epsilon\log(n)} \leq \frac{1}{2}$ for large enough $n$.

If $n^{\gamma\epsilon} \geq (4de)^2$, 
\begin{equation*}
\begin{split}
\left(\left(\frac{de}{\max\left(\frac{n^{\gamma\epsilon}}{4}, 3d\right)}\right)^{\max\left(\frac{n^{\gamma\epsilon}}{4}, 3d\right)}n^\gamma e\right)^{n^{1 - \gamma}} &=  \left(\left(\frac{de}{\frac{n^{\gamma\epsilon}}{4}}\right)^{\frac{n^{\gamma\epsilon}}{4}}n^\gamma e\right)^{n^{1 - \gamma}}
=  \left(\left(\frac{4de}{n^{\gamma\epsilon}}\right)^{\frac{n^{\gamma\epsilon}}{4}}n^{2\gamma}\right)^{n^{1 - \gamma}} \\
&\leq \left(\left(\frac{1}{n^{\gamma\epsilon}}\right)^{\frac{n^{\gamma\epsilon}}{8}}n^{2\gamma}\right)^{n^{1 - \gamma}}.
\end{split}
\end{equation*}
Taking a logarithm, we have
\begin{equation*}
\begin{split}
\log\left(\left(\frac{1}{n^{\gamma\epsilon}}\right)^{\frac{n^{\gamma\epsilon}}{8}}n^{2\gamma}\right)^{n^{1 - \gamma}} &= n^{1-\gamma}\left(2\gamma\log(n) -\frac{\gamma\epsilon}{8} n^{\gamma\epsilon}\log(n)\right)\\
&= n^{1-\gamma}\gamma\log(n)\left(2 - \frac{\epsilon n^{\gamma\epsilon}}{8}\right) 
\leq -2n^{1-\gamma}\gamma\log(n),
\end{split}
\end{equation*}
where in the inequality, we used $n^{\gamma\epsilon} > (4ed)^2$ to show $\frac{\epsilon n^{\gamma\epsilon}}{8} \geq \frac{\epsilon(4de^2)^{2}}{8} \geq 4$ for $d$ large enough. Now  $n^{\gamma\epsilon} > (4de)^2$ implies $\gamma > \frac{2\log(4de)}{\epsilon\log(n)}$. The function $n^{1-\gamma}\gamma$ achieves its minimum at one endpoint of the interval $\gamma \in [\frac{2\log(4de)}{\epsilon\log(n)}, 1]$ (one can verify that the log of the function is concave in $\gamma$). Thus 
\begin{equation*}
-2n^{1-\gamma}\gamma\log(n) \leq \max\left(-2n^{1- \frac{2\log(4de)}{\epsilon\log(n)}}\frac{2\log(4de)}{\epsilon}, -2\log(n)\right) \leq -2\log(n),
\end{equation*}

so if $n^{\gamma\epsilon} \geq (4de)^2$, we have 
\begin{equation}\label{case:big_r}
\left(\left(\frac{de}{\max\left(\frac{n^{\gamma\epsilon}}{4}, 3d\right)}\right)^{\max\left(\frac{n^{\gamma\epsilon}}{4}, 3d\right)}n^\gamma e\right)^{n^{1 - \gamma}} \leq e^{-2\log(n)}.
\end{equation}

Combining the two cases (eqs.~\ref{case:small_r} and \ref{case:big_r}) into eq.~\ref{R_max_bd}, we have 
\begin{equation*}
\Pr\left[\Bin\left(kn, \frac{d}{n}\right) \geq L_{max}\right] \leq 2\left(\frac{k}{en}\right)^ke^{-2\log(n)}.
\end{equation*}
\end{proof}

\section{Gradient Coding}\label{apx:abc}

In the following Appendix, we prove the claims stated in the proof of Lemma $\ref{lemma:abc_small}$. First, we need the approximation for the probability mass function of a HyperGeometric distribution.

\begin{lemma}\label{hypergeombound}
Let $\mathcal{X}\sim \text{HyperGeom}(A,B,n)$. Furthermore, let us define $(1-q)=\frac{B}{A}$. Then 
$$\Pr[\mathcal{X}=k]\leq\binom{n}{k}(1-q)^k\left(q+\frac{k}{A-n}\right)^{n-k}.$$
Furthermore, assuming $n\leq \frac{3}{2}qA$ and $q\leq 1/2$, we have:
\begin{equation*}
\begin{split}
\Pr[\mathcal{X}=k]&\leq \binom{n}{k}(1-q)^k q^{n-k}\left(e^{\left({{6ek}}\right)}\right)
\leq \left(\frac{en}{k}\right)^k(1-q)^k q^{n-k}\left(e^{\left({{6ek}}\right)}\right) 
\end{split}
\end{equation*} 
\end{lemma}
\begin{proof} 
Recall that, by definition of the hypergeometric distribution, we have:
$$\Pr[\mathcal{X}=k]=\frac{\binom{B}{k}\binom{A - B}{n-k}}{\binom{A}{n}}$$
We can expand out the binomial terms into their factorial representations to see:
\begin{align*}
    \Pr[\mathcal{X}=k]&=\frac{\frac{B!}{(B-k)!k!}\cdot\frac{(A - B)!}{(A - B-n+k)!(n-k)!}}{\frac{A!}{(A-n)!n!}}\\
    &=\binom{n}{k}\left(\frac{B!}{(B-k)!}\cdot\frac{(A-B)!}{(A-B-n+k)!}\cdot\frac{(A-n)!}{A!}\right)\\
    &=\binom{n}{k} \prod_{i=1}^{k}(B-k+i)\prod_{i=1}^{n-k}(A-B-n+k+i)\prod_{i=1}^n \frac{1}{A-n+i}\\
    &=\binom{n}{k} \prod_{i=1}^{k}\frac{B-k+i}{A-k+i}\prod_{i=1}^{n-k}\frac{A-B-n+k+i}{A-n+i}\\
    &=\binom{n}{k} \prod_{i=1}^{k}\frac{B-k+i}{A-k+i}\prod_{i=1}^{n-k}\left(1-\frac{B}{A-n+i}+\frac{k}{A-n+i}\right)   \\
    &\leq \binom{n}{k} \prod_{i=1}^{k}\frac{B-k+i}{A-k+i}\prod_{i=1}^{n-k}\left(1-\frac{B}{A}+\frac{k}{A-n}\right).
\end{align*}
By definition of the hypergeometric distribution, $k\leq B$. Thus, $k-i<B<A$ for all $i\in[k]$, so:
$$\frac{B-k+i}{A-k+i}\leq \frac{B}{A}=(1-q).$$
This gives us the first inequality of the lemma:
$$P[\mathcal{X}=k]\leq \binom{n}{k}\prod_{i=1}^k (1-q). \prod_{i=1}^{n-k} \left(q+\frac{k}{A-n}\right)=\binom{n}{k} (1-q)^k \left(q+\frac{k}{A-n}\right)^{n-k}.$$
We will now proceed under the assumption that $n\leq \frac{3}{2}qA$ to achieve the second bound. We write:
\begin{align*}
    \left(q+\frac{k}{A-n}\right)^{n-k}&\leq \left(q+\frac{k}{(1-\frac{3}{2}q)A}\right)^{n-k}
    =\sum_{i=0}^{n-k} \binom{n-k}{i}q^{n-k-i}\left(\frac{k}{(1-\frac{3}{2}q)A}\right)^{i}\\
    &\leq q^{n-k}\sum_{i=0}^{n-k} \left(\frac{\frac{3}{2}eqA}{i}\right)^i\left(\frac{1}{q}\right)^i\left(\frac{k}{(1-\frac{3}{2}q)A}\right)^{i}
    \leq q^{n-k}\sum_{i=0}^{n-k}\left(\frac{\frac{3}{2}ek}{i(1-\frac{3}{2}q)}\right)^{i}\\
    &\leq q^{n-k}\sum_{i=0}^{n-k} \left(\frac{1}{i!}\right)\left(\frac{\frac{3}{2}ek}{1-\frac{3}{2}q}\right)^{i}
    \leq q^{n-k}\left(e^{\left(\frac{\frac{3}{2}ek}{1-\frac{3}{2}q}\right)}\right)
    \leq q^{n-k}\left(e^{\left({6ek}\right)}\right).
\end{align*}
Applying this to the first inequality gives:
$$P[\mathcal{X}=k]\leq\binom{n}{k}(1-q)^kq^{n-k}\left(e^{\left(6ek\right)}\right).$$
\end{proof}

We are now ready to prove Claim \ref{claim:abc1}.
\claimabc*
\begin{proof}

Since $2\ell\left(\frac{2\ell}{N} \right) \leq \frac{4qK\ell}{N} \leq 
\frac{4q\ell}{18eq^2} \leq \frac{\ell}{2}$ for $q$ large enough, by Lemma~\ref{lem:tail}, we have
\begin{equation}\label{eq:tail_app}
\begin{split}
  \Pr\left[\Bin\left(\ell, \frac{2\ell}{n}\right)\geq j\right] \leq 2\left(\frac{4e\ell}{n}\right)^{j}.
\end{split}
\end{equation}

Using $K\leq \frac{3}{2}N$, we employ the second bound of Lemma \ref{hypergeombound} to find that for sufficiently large $q$,
\begin{align*}
    \Pr&\left[\Bin\left(\ell, \frac{2\ell}{N}\right)\geq j\right]\cdot \Pr[\textnormal{HyperGeom}(qN,q(1-p)N,qK)=\ell]\\
    &\leq 2\left(\frac{4e\ell}{N}\right)^{j}\left(\binom{qK}{\ell}(1-p)^{\ell}p^{qK-\ell}e^{6e\ell}\right) 
    \leq 2\left(\frac{8eK}{N}\right)^{j}\binom{qK}{2K}p^{qK-2K}e^{12eK} \\
    &\leq 2\left(8e\right)^{2K}\left(\frac{K}{ N}\right)^{j}\left(\frac{eq}{2}\right)^{2K}p^{(q-2)K}e^{12eK} 
    = 2\left(16e^{12e+4}q^2\right)^{K} \left(\frac{K}{N}\right)^{j} p^{(q-2)K}
    \leq\left(p^{q-c\log_p(q)} \right)^K\left(\frac{K}{N}\right)^{j}.
\end{align*}
for some universal constant $c$.

For the second statement, we combine eq.~\ref{eq:tail_app} and the second bound of Lemma \ref{hypergeombound}, achieving
\begin{align*}
    &\sum_{\ell=2K+1}^{qK}\Pr\left[\Bin\left(\ell, 1-\frac{2\ell}{N}\right)\geq \frac{\ell}{2}\right]\cdot\Pr[\textnormal{HyperGeom}(qN,q(1-p)N,qK)=\ell]\\
    &\leq \sum_{\ell=2K+1}^{qK}2\left(\frac{4e\ell}{N}\right)^{\frac{\ell}{2}}\left(\binom{qK}{\ell}(1-p)^K\left(p+\frac{\ell}{q(N-K)}\right)^{qK-\ell}\right)\\
    &\leq \sum_{\ell=2K+1}^{qK}2\left(\frac{4e\ell}{N}\right)^{\frac{\ell}{2}}\left(\frac{1}{2}+\frac{K}{(N-K)}\right)^{qK - \ell} \left(\frac{eqK}{\ell}\right)^\ell\\
    &\leq \sum_{\ell=2K+1}^{qK}2\left(\frac{4e\ell}{N}\right)^{\frac{\ell}{2}}\left(0.625\right)^{qK - \ell} \left(\frac{eqK}{\ell}\right)^\ell\\
    &\leq 2(0.625)^{qK} \sum_{\ell=2K+1}^{qK}\left(\frac{4e^{3}q^2K^2}{\ell N}\right)^{\ell/2}\\
    &\leq 2qK(0.625)^{qK}\max_{\ell\in\{2K+1,...,qK\}}\left(\frac{4e^{3}q^2K^2}{\ell N}\right)^{\ell/2}.
\end{align*}
To show the function in the expression above is maximized at $\ell = 2K+1$ when $K\leq \frac{N}{18eq^2}$, note that the maximum over $x$ of $f(x) = \left(\frac{y}{x}\right)^x$ is achieved at $x = y/e$, and above this value of $x$, $f(x)$ is decreasing. Let $C=\frac{4e^{3}q^2K^2}{N}$, and consider the function $f(x) = \left(\frac{C/2}{x}\right)^x$, such that $f(\ell/2) = \left(\frac{4e^{3}q^2K^2}{\ell N}\right)^{\ell/2}$. Thus $f$ is decreasing for $\ell/2 \geq \frac{C}{e} = \frac{4e^{2}q^2K^2}{N}$. Observe that $\frac{8e^2q^2K^2}{N} \leq 2K$, since $K \leq \frac{N}{18eq^2} \leq \frac{N}{4e^2q^2}$, and so $\max_{\ell\in\{2K+1,...,qK\}}\left(\frac{4e^{3}q^2K^2}{\ell N}\right)^{\ell/2} \leq \left(\frac{4e^{3}q^2K^2}{(2K + 1)N}\right)^{(2K + 1)/2}$. It follows that
\begin{align*}
    \sum_{\ell=2K+1}^{qK}&\Pr\left[\Bin\left(\ell, 1-\frac{2\ell}{N}\right)\geq \frac{\ell}{2}\right]\cdot\Pr[\textnormal{HyperGeom}(qN,q(1-p)N,qK)=\ell]\\
    &\leq 2qK(0.625)^{qK}\left(\frac{4e^3q^2K^2}{(2K+1)N}\right)^{K+1/2}
    \leq 2qK(0.625)^{qK}(2e^3q^2)^{K+1/2}\left(\frac{K}{N}\right)^{K+1/2}\\
    &\leq e^{-K}\left(\frac{K}{N}\right)^{K+1/2}.
\end{align*}


\end{proof}

\end{document}